\documentclass[11pt,fleqn,reqno]{article}

\usepackage{amsmath}
\usepackage{amssymb}
\usepackage{amsfonts}
\usepackage{amsthm}
\usepackage{mathtools}
\mathtoolsset{
above-intertext-sep = -1ex
}

\usepackage{caption,multirow} 
\usepackage[flushleft]{threeparttable}
\usepackage[utf8]{inputenc}
\usepackage[T1]{fontenc}

\usepackage[varg]{txfonts}
\let\mathbb=\varmathbb

\usepackage[sans]{dsfont}

\usepackage[left=2.5cm, right=2.5cm, top=2.5cm, bottom=2.5cm]{geometry}

\usepackage[%
cal=euler,
bb=fourier,
scr=zapfc,
]
{mathalfa}

\usepackage[font=small,labelfont=bf]{caption}
\usepackage{subfigure}
\usepackage{graphicx}
\graphicspath{{Figures/}} 

\usepackage{acronym}
\usepackage{latexsym}
\usepackage{paralist}
\usepackage{wasysym}
\usepackage{xspace}
\usepackage{framed}
\usepackage{palatino,pxfonts}
\usepackage{authblk}
\usepackage[textsize=footnotesize,textwidth=0.8in]{todonotes}
\usepackage[numbers,sort&compress]{natbib}

\newcommand{\ms}[1]{\textcolor{black}{#1}}
\newcommand{\us}[1]{\textcolor{black}{#1}}
\newcommand{\uss}[1]{\textcolor{black}{#1}}
\newcommand{\ic}[1]{\textcolor{black}{#1}}

\usepackage{hyperref}
\hypersetup{
colorlinks=true,
linktocpage=true,
pdfstartview=FitH,
breaklinks=true,
pdfpagemode=UseNone,
pageanchor=true,
pdfpagemode=UseOutlines,
plainpages=false,
bookmarksnumbered,
bookmarksopen=false,
bookmarksopenlevel=1,
hypertexnames=true,
pdfhighlight=/O,
urlcolor=black, linkcolor=blue, citecolor=blue, 
pdftitle={},
pdfauthor={},
pdfsubject={},
pdfkeywords={},
pdfcreator={pdfLaTeX},
pdfproducer={LaTeX with hyperref}
}

\numberwithin{equation}{section}  
\usepackage[sort&compress,capitalize,nameinlink]{cleveref}
\crefname{example}{Ex.}{Exs.}

\crefrangeformat{equation}{\upshape(#3#1#4)\textendash(#5#2#6)}
\usepackage[textsize=footnotesize,textwidth=0.8in]{todonotes}

\newcommand{\dd}{\:d}

\newcommand{\eps}{\varepsilon}

\newcommand{\dif}{\dd}

\DeclareMathOperator*{\argmin}{argmin}
\DeclareMathOperator*{\argmax}{argmax}

\DeclareMathOperator{\diag}{diag}

\DeclareMathOperator{\dom}{dom}

\DeclareMathOperator{\gr}{gr}

\DeclareMathOperator{\Id}{Id}

\newcommand{\const}{\mathtt{c}}
\newcommand{\ce}{\mathtt{e}}
\newcommand{\ca}{\mathtt{a}}

\newcommand{\cB}{\mathtt{B}}
\newcommand{\ct}{\mathtt{t}}
\newcommand{\cs}{\mathtt{s}}

\newcommand{\B}{\mathbb{B}}

\newcommand{\Real}{\mathbb{R}}

\renewcommand{\iff}{\Leftrightarrow}

\renewcommand{\emptyset}{\varnothing}
\newcommand{\eqdef}{\triangleq}
\newcommand{\wlim}{\rightharpoonup}

\newcommand{\scrB}{\mathcal{B}}
\newcommand{\scrC}{\mathcal{C}}
\newcommand{\setC}{\mathsf{C}}

\newcommand{\scrE}{\mathcal{E}}

\newcommand{\scrF}{\mathcal{F}}

\newcommand{\scrG}{\mathcal{G}}

\newcommand{\setH}{\mathsf{H}}

\newcommand{\scrL}{\mathcal{L}}

\newcommand{\scrO}{\mathcal{O}}

\newcommand{\scrQ}{\mathcal{Q}}

\newcommand{\scrS}{\mathcal{S}}
\newcommand{\setS}{\mathsf{S}}

\newcommand{\scrW}{\mathcal{W}}
\newcommand{\scrX}{\mathcal{X}}

\renewcommand{\Pr}{\mathbb{P}}
\newcommand{\Ex}{\mathbb{E}}

\newcommand{\measP}{{\mathsf{P}}}

\newcommand{\F}{{\mathbb{F}}}

\newcommand{\Uni}{{\mathtt{Uniform}}}

\newcommand{\Normal}{{\mathtt{N}}}

\newcommand{\R}{\mathbb{R}}

\newcommand{\N}{\mathbb{N}}

\DeclareMathOperator{\NC}{\mathsf{N}}

\newcommand{\Zer}{\mathsf{Zer}}

\newcommand{\gap}{\mathsf{Gap}}		
\DeclareMathOperator{\prox}{prox}

\newcommand{\residual}{\mathsf{res}}	

\usepackage{algorithmic}
\usepackage[ruled,vlined]{algorithm2e}

\theoremstyle{plain}
\newtheorem{theorem}{Theorem}
\newtheorem{corollary}[theorem]{Corollary}
\newtheorem*{corollary*}{Corollary}
\newtheorem{lemma}[theorem]{Lemma}
\newtheorem{proposition}[theorem]{Proposition}

\theoremstyle{definition}
\newtheorem{definition}[theorem]{Definition}
\newtheorem*{definition*}{Definition}
\newtheorem{assumption}{Assumption}
\newtheorem{problem}{Problem}
\renewcommand\qed{\hfill\small$\blacksquare$}

\theoremstyle{remark}
\newtheorem{remark}{Remark}
\newtheorem*{remark*}{Remark}
\newtheorem*{notation*}{Notational remark}
\newtheorem{example}{Example}

\numberwithin{theorem}{section}
\numberwithin{remark}{section}
\numberwithin{example}{section}

\DeclarePairedDelimiter{\abs}{\lvert}{\rvert}

\DeclarePairedDelimiter{\inner}{\langle}{\rangle}
\DeclarePairedDelimiter{\norm}{\lVert}{\rVert}

\newcommand{\acli}[1]{\textit{\acl{#1}}}

\newacro{UBV}[UBV]{uniformly bounded variance}
\newacro{SO}[SO]{stochastic oracle}
\newacroplural{NE}[NE]{Nash equilibria}
\newacro{VI}{variational inequality}
\newacroplural{VI}[VIs]{variational inequalities}
\newacro{SVI}{stochastic variational inequality}
\newacroplural{SVI}[SVIs]{stochastic variational inequalities}
\newacro{iid}[i.i.d.]{independent and identically distributed}
\newacro{FOM}{First-order method}
\newacro{SFBF}{stochastic forward-backward-forward}
\newacro{FBF}{forward-backward-forward}
\newacro{SEG}{stochastic extragradient}
\newacro{SFB}{stochastic forward-backward}
\newacro{IFB}{inertial forward-backward}
\newacro{RIFB}{relaxed inertial forward-backward}
\newacro{GNEP}{generalized Nash equilibrium problem}
\newacro{GNE}{generalized Nash equilibrium}
\newacro{RISFBF}{relaxed inertial stochastic forward-backward-forward}

\title{Stochastic Relaxed Inertial Forward-Backward-Forward splitting for Monotone Inclusions in Hilbert spaces}

\date{\today}

\author[1]{\small Shisheng Cui}
\author[1]{\small  Uday V. Shanbhag}
\author[2]{\small Mathias Staudigl\thanks{Corresponding author}}
\author[3]{\small Phan Tu Vuong}

\affil[1]{\footnotesize Department of Industrial and Manufacturing Engineering, Pennsylvania State University, University Park, PA 16802, USA }
\affil[2]{\footnotesize Department of Data Science and Knowledge Engineering, Maastricht University, P.O. Box 616, NL\textendash 6200 MD Maastricht, The Netherlands\\
(\href{mailto:m.staudigl@maastrichtuniversity.nl}{m.staudigl@maastrichtuniversity.nl})}
\affil[3]{\footnotesize            Mathematical Sciences, University of Southampton, Highfield Southampton SO17 1BJ, United Kingdom }
\begin{document}
\maketitle

\begin{abstract}
We consider monotone inclusions defined on a Hilbert space where the operator is given by the sum of a maximal monotone operator $T$ and  a single-valued monotone, Lipschitz continuous, and expectation-valued operator $V$. We draw
motivation from the seminal work by Attouch and Cabot~\cite{AttCab19,AttCab20} on relaxed inertial methods for monotone inclusions and present a stochastic extension of the relaxed inertial forward-backward-forward (RISFBF) method. Facilitated by an online variance reduction strategy via a mini-batch approach, we show that (RISFBF) produces a sequence that weakly converges to the solution set. Moreover, it is possible to estimate the rate at which the discrete velocity of the stochastic process vanishes. Under strong monotonicity, we demonstrate strong convergence, and give a detailed assessment of the iteration and oracle complexity of the scheme. When the mini-batch is raised at a geometric (polynomial) rate, the rate statement can be strengthened to a linear (suitable polynomial) rate while the oracle complexity of computing an $\epsilon$-solution improves to $\mathcal{O}(1/\epsilon)$.  Importantly, the latter claim allows for possibly biased oracles, a key theoretical advancement allowing for far broader applicability. By defining a restricted gap function based on  the Fitzpatrick function, we prove that the expected gap of an averaged sequence diminishes at a sublinear rate of $\scrO(1/k)$ while the oracle complexity of computing a suitably defined  $\epsilon$-solution is $\scrO(1/\epsilon^{1+a})$ where $a > 1$.  Numerical results on two-stage games and an overlapping group Lasso problem illustrate the advantages of our method compared to \ac{SFBF} and SA schemes.
\end{abstract}

\renewcommand{\sharp}{\gamma}
\acresetall
\allowdisplaybreaks

\section{Introduction}
\label{sec:intro}
%
\subsection{Problem formulation and motivation}
A wide range of problems in areas such as optimization, variational inequalities, game theory, signal processing, or traffic theory, can be reduced to solving inclusions involving set-valued operators in a Hilbert space $\setH$, i.e. to find a point $x\in\setH$ such that $0\in F(x)$, where $F:\setH\to 2^{\setH}$ is a set-valued operator. In many problems such inclusion problems display specific structure revealing that the operator $F$ can be additively decomposed. This leads us to the main problem we consider in this paper. 
\begin{problem}\label{problem}
Let $\setH$ be a real separable Hilbert space with inner product $\inner{\cdot,\cdot}$ and associated norm $\norm{\cdot}=\sqrt{\inner{\cdot,\cdot}}$. Let $T:\setH\to 2^{\setH}$ and $V:\setH\to\setH$ be maximally monotone operators, such that $V$ is $L$-Lipschitz continuous. The problem is to 
\begin{equation}\label{eq:MI}\tag{MI}
\text{ find }x\in\setH\text{ such that } 0\in F(x)\eqdef V(x)+T(x),
\end{equation}
\end{problem}

\noindent 
We assume that Problem \ref{problem} is well-posed:
\begin{assumption}\label{ass:exists}
$\setS\eqdef \Zer(F)\neq\emptyset$.
\end{assumption}

We are interested in the case where \eqref{eq:MI} is solved by an iterative algorithm based on a \emph{\ac{SO} representation} of the operator $V$. Specifically, when solving the problem, the algorithm calls to the \ac{SO}. At each call, the \ac{SO} receives as input a search point $x\in\setH$ generated by the algorithm on the basis of past information so far, and returns the output $\hat{V}(x,\xi)$, where $\xi$ is a random variable defined on some given probability space $(\Omega,\scrF,\Pr)$, taking values in a measurable set $\Xi$ with law $\measP=\Pr\circ\xi^{-1}$. In most parts of this paper, and the vast majority of contributions on stochastic variational problems in general, it is assumed that the output of the \ac{SO} is unbiased, 
\begin{equation}\label{eq:V}
    V(x) =\Ex_{\xi}[\hat{V}(x,\xi)]=\int_{\Xi}\hat{V}(x,z)\dif\measP(z)\qquad \forall x\in\setH.
\end{equation}
Such stochastic inclusion problems arise in numerous problems of fundamental importance in mathematical optimization and equilibrium problems, either directly or through an appropriate reformulation. An excellent survey on the existing techniques for solving problem \eqref{eq:MI} can be found in \cite{BauCom16} (in general Hilbert spaces) and \cite{FacPan03} (in the finite-dimensional case). 

\subsection{Motivating examples}
In what follows, we provide some motivating examples.

\begin{example}[{\bf Stochastic Convex Optimization}]
\label{ex:PDA}
Let $\setH_{1},\setH_{2}$ be separable Hilbert spaces. A large class of stochastic optimization problems, with wide range of applications in signal processing, machine learning and control, is given by  
\begin{equation}\label{eq:primal}
\min_{u\in\setH_{1}} \{f(u)+h(u)+g(Lu)\}
\end{equation}
where $h:\setH_{1}\to\R$ is a convex differentiable function with a Lipschitz continuous gradient $\nabla h$, represented as $h(u)=\Ex_{\xi}[\hat{h}(u,\xi)]$. $f:\setH_{1}\to(-\infty,\infty]$ and $g:\setH_{2}\to(-\infty,\infty]$ are
proper, convex lower semi-continuous functions, and $L:\setH_{1}\to\setH_{2}$
is a bounded linear operator. Problem \eqref{eq:primal} gains particular relevance in machine learning, where usually $h(u)$ is a convex data fidelity term (e.g. a \emph{population risk} functional), and $g(Lu)$ and $f(u)$ embody penalty or regularization terms; see e.g. total variation \cite{RudOshFat92}, hierarchical variable selection \cite{JacOboVer09,RocYuZha09}, and graph regularization \cite{TibRosSauZhu05,TibTay11}. Applications in control and engineering are given in \cite{AttBriCom10,latafat19}. We refer to \eqref{eq:primal} as the \emph{primal problem}. Using Fenchel-Rockafellar duality \cite[ch.19]{BauCom16}, the dual problem of \eqref{eq:primal} is given by 
\begin{equation}\label{eq:dual}
\min_{v\in\setH_{2}}\{(f+h)^{\ast}(-L^{\ast}v)+g^{\ast}(v)\},
\end{equation}
where $g^{\ast}$ is the Fenchel conjugate of $g$ and $(f+h)^{\ast}(w)=f^{\ast}\square h^{\ast}(w)=\inf_{u\in\setH_{1}}\{f^{\ast}(u)+h^{\ast}(w-u)\}$ represents the infimal convolution of the functions $f$ and $h$. Combining the primal problem \eqref{eq:primal} with its dual \eqref{eq:dual}, we obtain the saddle-point problem 
\begin{equation}\label{eq:SPP}
\inf_{u\in\setH_{1}}\sup_{v\in\setH_{2}}\{f(u)+h(u)-g^{\ast}(v)+\inner{Lu,v}\}.
\end{equation}
Following classical Karush-Kuhn-Tucker theory \cite{Roc74}, the primal-dual optimality conditions associated with \eqref{eq:SPP} are concisely represented by the following monotone inclusion: Find $\bar{x}=(\bar{u},\bar{v})\in\setH_{1}\times\setH_{2}\equiv\setH$ such that 
\begin{equation}\label{eq:KKT}
-L^{\ast}\bar{v}\in \partial f(\bar{u})+\nabla h(\bar{u}),\text{ and }L\bar{u}\in\partial g^{\ast}(\bar{v}).
\end{equation}
We may compactly summarize these conditions in terms of the zero-finding problem \eqref{eq:MI} using the operators $V$ and $T$, defined as 
\begin{align*}
    V(u,v)\eqdef ( \nabla h(u)+L^{\ast}v, -Lu ) \text{ and } 
    T(u,v)\eqdef  \partial f(u)\times \partial g^{\ast}(v).
\end{align*}
Note that the operator $V:\setH\to\setH$ is the sum of a maximally monotone and a skew-symmetric operator. Hence, in general, it is not cocoercive\footnote{An operator $C:\setH\to\setH$ is cocoercive if there exists $\beta>0$ such that $\inner{C(x)-C(y),x-y}\geq\beta\norm{C(x)-C(y)}^{2}$.}. Conditions on the data guaranteeing Assumption \ref{ass:exists} are stated in \cite{ComPes12}. 

Since $h(u)$ is represented as an expected value, we need to appeal to simulation based methods to evaluate its gradient. Also, significant computational speedups can be made if we are able to sample the skew-symmetric linear operator $(u,v)\mapsto (L^{\ast}u,-Lu)$ in an efficient way. Hence, we assume that there exists a \ac{SO} that can provide unbiased estimator to the gradient operators $\nabla h(u)$ and $(L^{\ast}v,-Lu)$. More specifically, given the current position $x=(u,v)\in\setH_{1}\times\setH_{2}$, the oracle will output the random estimators $\hat{H}(u,\xi),\hat{L}_{u}(u,\xi),\hat{L}_{v}(v,\xi)$ such that 
\begin{align*}
\Ex_{\xi}[\hat{H}(u,\xi)]=\nabla h(u),\; \Ex_{\xi}[\hat{L}_{u}(u,\xi)]=Lu,\text{ and }\Ex_{\xi}[\hat{L}_{v}(v,\xi)]=L^{\ast}v.
\end{align*}
This oracle feedback generates the \emph{random operator} $\hat{V}(x,\xi)=(\hat{H}(u,\xi)+\hat{L}_{v}(v,\xi),-\hat{L}_{u}(u,\xi))$, which allows us to approach the saddle-point problem \eqref{eq:SPP} via simulation-based techniques. 
\end{example}

\begin{example}[{\bf Stochastic variational inequality problems}]
\label{ex:SVI}
There are a multitude of examples of monotone inclusion problems \eqref{eq:MI} where the single-valued map $V$ is not the gradient of a convex function. An important model class where this is the case is the \emph{\ac{SVI}} problem. Due to their huge number of applications, \ac{SVI}'s received enormous interest over the last several years from various communities~\cite{sajiang08,shanbhag13stochastic,MerStaIFAC19,MerStaCDC17}. This problem emerges when $V(x)$ is represented
as an expected value as in \eqref{eq:V} and $T(x)=\partial g(x)$ for some proper lower semi-continuous function $g:\setH\to(-\infty,\infty]$. In this case, the resulting structured monotone inclusion problem can be equivalently stated as 
\begin{align}\label{eq:SVI}
\text{find }\bar{x}\in\setH\text{ s.t. }\inner{V(\bar{x}),x-\bar{x}}+g(x)-g(\bar{x})\geq 0\quad\forall x\in\setH.
\end{align}
An important and frequently studied special case of \eqref{eq:SVI} arises if $g$ is the indicator function of a given closed and convex subset $\setC\subset\setH$. In this cases the set-valued operator $T$ becomes the normal cone map 
\begin{equation}\label{eq:NC}
T(x)=\NC_{\setC}(x)\eqdef\left\{\begin{array}{ll} \left\{p\in\setH\vert \sup_{y\in\setC}\inner{y-x,p}\leq 0\right\} & \text{if }x\in\setC,\\
 \emptyset & \text{else}.
 \end{array}
 \right.
 \end{equation}
 This formulation includes many fundamental problems including fixed point problems, Nash equilibrium problems and complementarity problems \cite{FacPan03}. Consequently, the equilibrium condition \eqref{eq:SVI} reduces to
\[
\text{find }\bar{x}\in\setC\text{ s.t. }\inner{V(\bar{x}),x-\bar{x}}\geq 0\quad\forall x\in\setC.
\]
\end{example}

\subsection{Contributions} 
Despite the advances in stochastic optimization and variational inequalities, the algorithmic treatment of general monotone inclusion problems under stochastic uncertainty is a largely unexplored field. This is rather surprising given the vast amount of applications of maximally monotone inclusions in control and engineering, encompassing distributed computation
of generalized Nash equilibria~\cite{BricCom13,YiPav19,FraStaGramECC20}, traffic systems~\cite{Frie93,Fuk96,Han:2019ua}, and PDE-constrained optimization~\cite{Boergens2021}. 
The first major aim of this manuscript is to introduce and investigate a \ac{RISFBF} method, building on an operator splitting scheme originally due to Paul Tseng \cite{Tse00}. \ac{RISFBF} produces three sequences $\{(X_{k},Y_{k},Z_{k});k\in\N\}$, defined as 
\begin{align}\tag{RISFBF}
    \begin{aligned}
    Z_{k}&=X_{k}+\alpha_{k}(X_{k}-X_{k-1}),\\
    Y_{k}& =J_{\lambda_{k}T}(Z_{k}-\lambda_{k}A_k(Z_{k})),\\
    X_{k+1}& =(1-\rho_{k})Z_{k}+\rho_{k}[Y_{k}+\lambda_{k}(A_k(Z_{k})-B_k(Y_{k}))].
\end{aligned}
\end{align}
The data involved in this scheme are explained as follows:
\begin{itemize}
\item $A_k(Z_{k})$ and $B_k(Y_{k})$ are random estimators of $V$ obtained by consulting the \ac{SO} at search points $Z_k$ and $Y_k$, respectively;
\item $(\alpha_{k})_{k\in\N}$ is a sequence of non-negative numbers regulating the memory, or \emph{inertia} of the method;
\item $(\lambda_{k})_{k\in\N}$ is a positive sequence of step-sizes;
\item $(\rho_{k})_{k\in\N}$ is a non-negative \emph{relaxation} sequence.
\end{itemize} 

If $\alpha_{k}=0$ and $\rho_{k}=1$ the above scheme reduces to the stochastic forward-backward-forward method developed in \cite{BotMerStaVu21,CuiSha20}, with important applications in Gaussian communication networks \cite{MerStaIFAC19} and dynamic user equilibrium problems \cite{GibStaThoVuo20}. However, even more connections to existing methods can be made. 

\vspace{0.2cm}
\noindent {\bf Stochastic Extragradient.} 
If $T = \{0\}$, we obtain the inertial extragradient method 
\begin{align*}
Z_{k}&=X_{k}+\alpha_{k}(X_{k}-X_{k-1}),\\
Y_{k}&=Z_{k}-\lambda_{k}A_{k}(Z_{k}),\\
X_{k+1}&=Z_{k}-\rho_{k}\lambda_{k}B_{k}(Y_{k}).
\end{align*}
If  $\alpha_{k}=0$, this reduces to a generalized extragradient method 
\begin{align*}
Y_{k}&=X_{k}-\lambda_{k}A_{k}(X_{k}),\\
X_{k+1}&=X_{k}-\lambda_{k}\rho_{k}B_{k}(Y_{k}),
\end{align*}
recently introduced in \cite{Diakonikolas:2021vz}. 

\vspace{0.2cm}
\noindent{\bf Proximal Point Method. }
If $V=0$, the method reduces to the well-known deterministic proximal point algorithm~\cite{AttCab20}, overlaid by inertial and relaxation effects. The scheme reads explicitly as 
\begin{align*}
&Z_{k}=X_{k}+\alpha_{k}(X_{k}-X_{k-1}),\\
&X_{k+1}=(1-\rho_{k})Z_{k}+\rho_{k}J_{\lambda_{k}T}(Z_{k}).
\end{align*}
%
%
%
The list of our contributions reads as follows: 
\begin{itemize}
\item[(i)] \emph{Wide Applicability.} A key argument in favor of Tseng's operator splitting method is that it is provably convergent when solving structured monotone inclusions of the type \eqref{eq:MI}, without imposing \emph{cocoercivity} of the single-valued part $V$. This is a remarkable advantage relative to the perhaps more familiar and direct forward-backward splitting methods (aka projected (stochastic) gradient descent in the potential case). In particular, our scheme is applicable to the primal-dual splitting described in Example \ref{ex:PDA}. 

\item[(ii)] \emph{Asymptotic guarantees.} We show that under suitable assumptions on the relaxation sequence $(\rho_k)_{k\in\N}$, the non-decreasing inertial sequence $(\alpha_k)_{k\in\N}$, and step-length sequence $(\lambda_k)_{k\in\N}$, the generated stochastic process $(X_{k})_{k\in\N}$ weakly almost surely converges to a random variable with values in $\setS$. Assuming demiregularity of the operators yields strong convergence in the real (possibly infinite-dimensional) Hilbert space. 

\item[(iii)] \emph{Non-asymptotic linear rate under strong monotonicity of $V$.} When $V$ is strongly monotone, strong convergence of the last iterate is shown and the sequence admits a non-asymptotic linear rate of convergence {\em without a conditional unbiasedness of the \ac{SO}.} In particular, we show that the iteration and oracle complexity of computing an $\epsilon$-solution is no worse than $\scrO(\log(\tfrac{1}{\epsilon}))$ and $\scrO(\tfrac{1}{\epsilon})$, respectively.    

\item[(iv)] \emph{Non-asymptotic sublinear rate under monotonicity of $V$.}
When $V$ is monotone, by leveraging the \emph{Fitzpatrick function} \cite{Fit88,SimZal04,BauCom16} associated with the structured operator $F=T+V$, we propose a restricted gap function. We then prove that the expected gap of an averaged sequence
diminishes at the rate of $\mathcal{O}(\tfrac{1}{k})$. This allows us to derive an $\scrO(\tfrac{1}{\epsilon})$ upper bound on the iteration complexity, and an $\scrO(\tfrac{1}{\epsilon^{2+\delta}})$ upper bound (for $\delta>0)$ on the oracle complexity for computing an $\epsilon$-solution.
\end{itemize}
%
The above listed contributions shed new light on a set of urgent open questions, which we summarize below:
\begin{itemize}
\item[{(i)}] {\em Absence of rigorous asymptotics.} 
So far no aymptotic convergence guarantees have been available when considering relaxed
inertial FBF schemes when $T$ is maximally monotone and $V$ is a single-valued
monotone expectation-valued map.  

\item[{(ii)}] {\em Unavailability of rate statements.} We are not aware of any known non-asymptotic rate guarantees for algorithms solving \eqref{eq:MI} under stochastic uncertainty. A key barrier in monotone and stochastic regimes in developing such statements has been in the availability of a residual function. Some recent progress in the special \acl{SVI} case has been made by \cite{IusJofOliTho17,Iusem:2019ws,BotMerStaVu21}, but the general Hilbert-space setting involving set-valued operators seems to be largely unexplored (we will say more in Section \ref{sec:related}). 

\item[{(iii)}] {\em Bias requirements.} A standard assumption in stochastic optimization is that the \ac{SO} generates signals which are unbiased estimators of the deterministic operator $V(x)$. Of course, the requirement that the noise process is unbiased may often fail to hold in practice. In the present Hilbert space setting this is in some sense even expected to be the rule rather than the exception, since most operators are derived from complicated dynamical systems or the optimization method is applied to discretized formulations of the original problem. See the recent work \cite{GeiPfl19,Geiersbach:2020tw} for an interesting illustration in the context of PDE-constrained optimization. Some of our results go beyond the standard unbiasedness assumption. 

\end{itemize}

\subsection{Related research} 
\label{sec:related}
Understanding the role of inertial and relaxation effects in numerical schemes is a line of research which received enormous interest over the last two decades. Below, we try to give a brief overview about related algorithms.\\

\noindent {\bf Inertial, Relaxation, and Proximal schemes.} 

In the context of convex optimization, Polyak \cite{Pol64} introduced the \emph{Heavy-ball method}. This is a two-step method for minimizing a smooth convex function $f$. The algorithm reads as 

\begin{equation}\label{eq:HB}\tag{HB}
\left\{\begin{array}{l} Z_{k}=X_{k}+\alpha_{k}(X_{k}-X_{k-1}),\\
X_{k+1}=Z_{k}-\lambda_{k}\nabla f(X_{k})
\end{array}\right.
\end{equation}
The difference from the gradient method is that the base point of the gradient descent step is taken to be the extrapolated point $Z_{k}$, instead of $X_{k}$. This small difference has the surprising consequence that \eqref{eq:HB} attains optimal complexity guarantees for strongly convex functions with Lipschitz continuous gradients. Hence, \eqref{eq:HB} resembles an optimal method \cite{Nes04}. The acceleration effects can be explained by writing the process entirely in terms of a single updating equation as
\[
X_{k+1}-2X_{k}-X_{k-1}+(1-\alpha_{k})(X_{k}-X_{k-1})+\lambda_{k}\nabla f(X_{k})=0.
\]
Choosing $\alpha_{k}=1-a_{k}\delta_{k}$ and $\lambda_{k}=\gamma_{k}\delta^{2}_{k}$ for $\delta_{k}$ a small parameter, we arrive at 
\[
\frac{1}{\delta_{k}^{2}}(X_{k+1}-2X_{k}-X_{k-1})+\frac{a_{k}}{\delta_{k}}(X_{k}-X_{k-1})+\gamma_{k}\nabla f(X_{k})=0.
\]
This can be seen as a discrete-time approximation of the second-order dynamical system 
\[
\ddot{x}(t)+\frac{a}{t}\dot{x}(t)+\gamma(t)\nabla f(x(t))=0,
\]
\cite{Pol87} introduced this dynamical system in an optimization context. Since then, it has received significant attention in the potential, as well as in the non-potential case (see e.g \cite{AttMai11,BotCse16b,AttPey19} for an appetizer). As pointed out in \cite{SuBoyCan16}, if $\gamma(t)=1$, the above system reduces to a continuous version of Nesterov's fast gradient method \cite{Nes83}. Recently, \cite{GadPanSaa18} defined a stochastic version of the Heavy-ball method.\\

Motivated by the development of such fast methods for convex optimization, Attouch and Cabot \cite{AttCab19} studied a \emph{relaxed-inertial forward-backward algorithm}, reading as 
\begin{equation}\tag{RIFB}
\left\{\begin{array}{l}
    Z_{k}=X_{k}+\alpha_{k}(X_{k}-X_{k-1}),\\
    Y_k = J_{\lambda_{k}T}(Z_{k}-\lambda_{k}V(Z_{k})) \\
    X_{k+1}=(1-\rho_{k})Z_{k}+\rho_{k} Y_{k}.
    \end{array}\right.
\end{equation}
If $V=0$, this reduces to a relaxed inertial proximal point method analyzed by Attouch and Cabot~\cite{AttCab20}. If $\rho_{k}=1$, an inertial forward-backward splitting method is recovered, first studied by Lorenz and Pock \cite{LorPoc15}.\\

Convergence guarantees for the forward-backward splitting rely on the cocoercivity (inverse strong monotonicity) of the single-valued operator $V$. Example \ref{ex:PDA}, in which $V$ is given by a monotone plus a skew-symmetric linear operator, illustrates an important instance for which this assumption is not satisfied (see \cite{BricCom11} for further examples). A general-purpose operator splitting framework, relaxing the cocoercivity property, is the forward-backward-forward (FBF) method due to Tseng \cite{Tse00}. Inertial \cite{BotCse16} and relaxed-inertial~\cite{BotSedVuo00} versions of FBF have been developed. An all-encompassing numerical scheme can be compactly described as 
\begin{equation}\tag{RIFBF}
 \left\{ \begin{array}{l}
    Z_{k}=X_{k}+\alpha_{k}(X_{k}-X_{k-1}),\\
    Y_{k} =J_{\lambda_{k}T}(Z_{k}-\lambda_{k}V(Z_{k})),\\
    X_{k+1}=(1-\rho_{k})Z_{k}+\rho_{k}[Y_{k}-\lambda_{k}(V(Y_{k})-V(Z_{k}))].
\end{array}\right.
\end{equation}
Weak and strong convergence under appropriate conditions on the involved operators and parameter sequences are established in~\cite{BotSedVuo00}, but no rate statements are given. 

\vspace{0.2cm}
\noindent {\bf Related work on stochastic approximation.} Efforts in
extending stochastic approximation methods to variational
inequality problems have considered standard projection
schemes~\cite{sajiang08} for Lipschitz and strongly monotone operators. Extragradient and (more generally) mirror-prox algorithms ~\cite{JudLanNemSha09,JudNemTau11} can contend with merely monotone operators, while iterative 
smoothing~\cite{YouNedSha17} schemes can cope with
with the lack of Lipschitz continuity. It is worth noting that extragradient schemes have recently assumed relevance in the training of generative adversarial networks (GANS)~\cite{gidel2018variational,mishchenko2020revisiting}. Rate
analysis for \ac{SEG} have led
to optimal rates for Lipschitz and monotone operators~\cite{JudNemTau11}, as
well as extensions to non-Lipschitzian~\cite{YouNedSha17} and
pseudomonotone settings~\cite{IusJofOliTho17,KanSha19}. To alleviate the computational complexity single-projection schemes, such as the \acf{SFBF} method~\cite{BotMerStaVu21,CuiSha20}, as well
as subgradient-extragradient and projected reflected algorithms~\cite{CuiSha21} have been studied as well. 

 \ac{SFBF} has been shown to be nearly optimal in terms
of iteration and oracle complexity, displaying significant empirical
improvements compared to \ac{SEG}. While the role of  inertia in optimization
is well documented, in stochastic splitting problems, the only contribution
we are aware of is the work by Rosasco et
al.~\cite{RosVilVu16}. In that paper asymptotic guarantees for an inertial \ac{SFB}
algorithm are presented under the hypothesis that the operators $V$ and
$T$ are maximally monotone and the single-valued operator $V$ is cocoercive.

\vspace{0.2cm}
\noindent {\bf Variance reduction approaches.} Variance-reduction schemes address the deterioration in convergence
rate and the resulting poorer practical behavior via two commonly adopted
avenues: 
\begin{itemize}
\item[(i)] If the single-valued part $V$ appears as a finite-sum (see e.g. \cite{gidel2018variational,PalBac16}), variance-reduction ideas from machine learning \cite{Gower:2020ty} can be used.
\item[(ii)] Mini-batch schemes that employ an increasing batch-size of
gradients~\cite{friedlander12hybrid} lead to deterministic rates of convergence
for stochastic strongly convex~\cite{jalilzadeh2018smoothed},
convex~\cite{jofre2019variance}, and nonconvex
optimization~\cite{GhaLanHon16}, as well as for pseudo-monotone \ac{SVI}s via extragradient~\cite{IusJofOliTho17}, and splitting schemes~\cite{BotMerStaVu21}.
\end{itemize}

In terms of run-time, improvements in iteration complexities achieved by mini-batch approaches are significant; e.g. in strongly monotone regimes, the iteration complexity improves from $\scrO(\tfrac{1}{\epsilon})$ to $\scrO(\ln(\tfrac{1}{\epsilon}))$ \cite{CuiSha20,CuiSha21}. Beyond
run-time advantages,  such avenues provide asymptotic and rate guarantees 
under possibly weaker assumptions on the problem as well as the oracle; in particular,
mini-batch schemes allow for possibly biased oracles and state-dependency
of the noise~\cite{CuiSha21}. Concerns about the sampling burdens are, in our opinion, often
overstated since such schemes are meant to provide $\epsilon$-solutions; e.g. 
if $\epsilon=10^{-3}$ and the obtained rate is $\mathcal{O}(1/k)$, then the
batch-size $m_k = \lfloor k^a \rfloor$ where $a > 1$, implying that the
batch-sizes are $\mathcal{O}(10^{3a})$, a relatively modest requirement, given
the advances in computing.
   
\vspace{0.2cm}
\noindent {\bf Outline.} The remainder of the paper is organized in five sections. After dispensing with the preliminaries in Section~\ref{sec:prelims}, we present the (RISFBF) scheme in Section~\ref{sec:Algo}. Asymptotic and rate statements are developed in Section~\ref{sec:analysis} and preliminary numerics are presented in Section~\ref{sec:Applications}. We conclude with some brief remarks in Section~\ref{sec:conclusion}.

\section{Preliminaries}
\label{sec:prelims}
%
\subsection{Notation}
Throughout, $\setH$ is a real separable Hilbert space with scalar product
$\inner{\cdot,\cdot}$, norm $\norm{\cdot}$, and Borel $\sigma$-algebra
$\scrB$. The symbols $\to$ and $\wlim$ denote strong and weak convergence,
respectively. $\Id:\setH\to\setH$ denotes the identity operator on $\setH$. Stochastic
uncertainty is modeled on a complete probability space $(\Omega,\scrF,\Pr)$,
endowed with a filtration $\F=(\scrF_{k})_{k\geq 0}$. By means of the
Kolmogorov extension theorem, we assume that $(\Omega,\scrF,\Pr)$ is large
enough so that all random variables we work with are defined on this
space. A $\setH$-valued random variable is a measurable function
$X:(\Omega,\scrF)\to (\setH,\scrB)$.  We denote by $\ell^{0}(\F)$ the set of
sequences of real-valued random variables $(\xi_{k})_{k\in\N}$ such that, for
every $k\in\N$, $\xi_{k}$ is $\scrF_{k}$-measurable. For $p\in[1,\infty]$, we
set 
\[
\ell^{p}(\F)\eqdef \left\{(\xi_{k})_{k\in\N}\in\ell^{0}(\F)\vert \sum_{k\geq 1}\abs{\xi_{k}}^{p}<\infty\quad\Pr\text{-a.s.}\right\}. 
\]
We denote the set of summable non-negative sequences by $\ell^{1}_{+}(\N)$. We now collect some concepts from monotone operator theory. For more details, we  refer the reader to \cite{BauCom16}. Let $F:\setH\to 2^{\setH}$ be a set-valued operator. Its domain and graph are defined as 
$\dom F\eqdef \{x\in\setH\vert F(x)\neq\emptyset\},\text{ and }\gr(F)\eqdef\{(x,u)\in\setH\times\setH\vert u\in F(x)\},$ respectively. Recall that an operator $F:\setH\to 2^{\setH}$ is monotone if 
\begin{equation}
\inner{ v-w,x-y}\geq 0\qquad\forall (x,v),(y,w)\in\gr(F).
\end{equation}
The set of zeros of $F$, denoted by $\Zer(T)$, defined as $\Zer(F)\eqdef \{x\in\setH\vert 0\in T(x)\}$. The inverse of $F$ is $F^{-1}:\setH\to 2^{\setH},u\mapsto F^{-1}(u)=\{x\in\setH\vert u\in F(x)\}$. The resolvent of $F$ is $J_{F}\eqdef (\Id+ F)^{-1}.$ If $F$ is maximally monotone, then $J_{F}$ is a single-valued map. We also need the classical notion of \emph{demiregularity} of an operator. 
\begin{definition}\label{def:demiregular}
An operator $F:\setH\to 2^{\setH}$ is demiregular at $x\in\dom(F)$ if for every sequence $\{(y_{n},u_{n})\}_{n\in\N}\subset\gr(F)$ and every $u\in F(y)$, we have 
\begin{align*}
[y_{n}\wlim y,v_{n}\to v]\Rightarrow y_{n}\to y.
\end{align*}
\end{definition}
The notion of demiregularity captures various properties typically used to establish strong convergence of dynamical systems. \cite{AttBriCom10} exhibit a large class of possibly set-valued operators $F$ which are demiregular. In particular, demiregularity holds if $F$ is uniformly or strongly monotone, or when $F$ is the subdifferential of a uniformly convex lower semi-continuous function $f$.

\section{Algorithm}
\label{sec:Algo}
%
Let $(\Omega,\scrF,\Pr)$ be a complete probability space. Our aim is to solve the monotone inclusion problem \eqref{eq:MI} under the following assumption: 
\begin{assumption}\label{ass:MM}
Consider Problem \ref{problem}. The set-valued operator $T:\setH\to 2^{\setH}$ is maximally monotone with an efficiently computable resolvent. The single-valued operator $V:\setH\to\setH$ is maximally monotone and $L$-Lipschitz continuous ($L>0)$ with full domain $\dom V=\setH$.
\end{assumption}
Assumption \ref{ass:MM} guarantees that the operator $F=T+V$ is maximally monotone \cite[Corolllary 24.4]{BauCom16}.\\

For numerical tractability, we make a \emph{finite-dimensional noise} assumption, common to stochastic optimization problems
in (possibly infinite-dimensional) Hilbert spaces~\cite{GunWebZha14}.\footnote{Our analysis does not rely on this assumption. It is made here only for concreteness and because it is the most prevalent one in applications.}

\begin{assumption}[Finite-dimensional Noise]
\label{ass:fdd}
All randomness can be described via a finite dimensional random variable
$\xi:(\Omega,\scrF)\to (\Xi,\scrE)$, where $\Xi\subseteq\R^{d}$ is a
measurable set with Borel sigma algebra $\scrE$. The law of the random
variable $\xi$ is denoted by $\measP$, i.e.
$\measP(\Gamma)\eqdef \Pr(\{\omega\in\Omega\vert \xi(\omega)\in\Gamma\})$ for
all $\Gamma\in\scrE$. 
\end{assumption}

To access new information about the values of the operator $V(x)$, we adopt a
\emph{stochastic approximation} (SA) approach where samples are accessed
iteratively and online: At each iteration, we assume to have access to a \acf{SO} which generates some estimate on the value of the deterministic operator $V(x)$ when the current position is $x$. This information is obtained by drawing an iid sample form the law $\measP$. These fresh samples are then used in the numerical algorithm after an initial extrapolation step delivering the point $Z_{k}=X_{k}+\alpha_{k}(X_{k}-X_{k-1})$, for some extrapolation coefficient $\alpha_{k}\in[0,1]$. Departing from $Z_{k}$, we call the \ac{SO} to retrieve the mini-batch estimator with \emph{sample rate} $m_{k}\in\N$:
\begin{equation}\label{eq:A}
A_{k}(Z_{k},\omega)\eqdef\frac{1}{m_{k}}\sum_{t=1}^{m_{k}}\hat{V}(Z_{k},\xi_{k}^{(t)}(\omega)).
\end{equation}
$\xi_{k}\eqdef (\xi_{k}^{(1)},\ldots,\xi_{k}^{m_{k}})$ is the data sample employed by the \ac{SO} to return the estimator $A_{k}(Z_{k})$. Subsequently we perform a forward-backward update with step size $\lambda_{k}>0$:
\begin{equation}\label{eq:Y}
Y_{k}=J_{\lambda_{k}T}\left(Z_{k}-\lambda_{k}A_{k}(Z_{k})\right).
\end{equation}
In the final updates, a second independent call of the \ac{SO} is made, using the data set $\eta_{k}=(\eta^{(1)}_{k},\ldots,\eta^{(m_{k})}_{k})$, yielding the estimator 
\begin{equation}\label{eq:B}
B_{k}(Y_{k},\omega)\eqdef \frac{1}{m_{k}}\sum_{t=1}^{m_{k}}\hat{V}(Y_{k},\eta_{k}^{(t)}(\omega)),
\end{equation}
and the new state   
\begin{equation}\label{eq:Xnew}
X_{k+1}=(1-\rho_{k})Z_{k}+\rho_{k}\left[Y_{k}+\lambda_{k}(A_{k}(Z_{k})-B_{k}(Y_{k}))\right]
\end{equation}
This iterative procedure generates a stochastic process $\{(Z_{k},Y_{k},X_{k})\}_{k\in\N}$, defining the \acli{RISFBF} (RISFBF) scheme. A pseudocode is given as Algorithm \ref{alg:RISFBF} below. 

\begin{algorithm}[H]
 \caption{\ac{RISFBF}}
 \label{alg:RISFBF}
 \begin{algorithmic}
\STATE {\bf Input: } $X_{0},X_{1}\in\dom T$, Terminal time $K\in\N$, nonnegative sequences $(\alpha_{k})_{k=1}^{K},(\lambda_{k})_{k=1}^{K},(\rho_{k})_{k=1}^{K},(m_{k})_{k=1}^{K}$.
\FOR{$k=1,\ldots,K$}
\STATE Compute $Z_{k}=X_{k}+\alpha_{k}(X_{k}-X_{k-1})$\;
\STATE  Draw an iid sample $\{\xi_{k}^{(1)},\ldots,\xi_{k}^{(m_{k})}\}$ from the law $\measP$\;
\STATE Compute $A_k(Z_{k})$ and $Y_k$ as in \eqref{eq:A} and \eqref{eq:Y}\;
\STATE  Draw an iid sample $\{\eta_{k}^{(1)},\ldots,\eta_{k}^{(m_{k})}\}$ from the law $\measP$\;
\STATE Compute $B_k(Y_{k})$ and $X_{k+1}$ as in \eqref{eq:B} and \eqref{eq:Xnew}.
\ENDFOR
\STATE Report
\begin{equation}\label{eq:output}
\bar{X}_{K}=\sum_{k=1}^{K}\frac{\rho_{k}Y_{k}}{\sum_{k=1}^{K}\rho_{k}}.
\end{equation}
\end{algorithmic}
\end{algorithm}
Note that \ac{RISFBF} is still conceptual since we have not explained how the sequences $(\alpha_{k})_{k\in\N},(\lambda_{k})_{k\in\N}$ and $(\rho_{k})_{k\in\N}$ should be chosen. We will make this precise in our complexity analysis, starting in Section \ref{sec:analysis}.

\subsection{Equivalent form of RISFBF}
We can collect the sequential updates of \ac{RISFBF} as the fixed-point iteration
\begin{equation}\label{eq:IKM}
\left\{\begin{array}{l}
Z_{k}=X_{k}+\alpha_{k}(X_{k}-X_{k-1}),\\
X_{k+1}=Z_{k}-\rho_{k}\Phi_{k,\lambda_{k}}(Z_{k})
\end{array}\right.
\end{equation}
where $\Phi_{k,\lambda}:\setH\times\Omega\to\setH$ is the time-varying map given by 
\[
\Phi_{k,\lambda}(x,\omega)\eqdef x-\lambda A_{k}(x,\omega)- (\Id_{\setH}-\lambda B_{k}(\cdot,\omega))\circ J_{\lambda T}\circ(\Id_{\setH}-\lambda A_{k}(\cdot,\omega))(x).
\]
Formulating the algorithm in this specific way establishes the connection between RISFBF and the Heavy-ball system. Indeed, combining the iterations in \eqref{eq:IKM} in one, we get a second-order difference equations, closely resembling the structure present in \eqref{eq:HB}:
\[
\frac{1}{\rho_{k}}(X_{k+1}-2X_{k}-X_{k-1})+\frac{(1-\alpha_{k})}{\rho_{k}}(X_{k}-X_{k-1})+\Phi_{k,\lambda_{k}}(X_{k}+\alpha_{k}(X_{k}-X_{k-1}))=0.
\]
Also, it reveals the Markovian nature of the process $(X_{k})_{k\in\N}$; It is clear from the formulation \eqref{eq:IKM} that $X_{k}$ is Markov with respect to the sigma-algebra $\sigma(\{X_{0},\ldots,X_{k-1}\})$. 

\subsection{Assumptions on the stochastic oracle}
\label{sec:SO}
In order to tame the stochastic uncertainty in \ac{RISFBF}, we need to impose some
assumptions on the distributional properties of the random fields
$(A_{k}(x))_{k\in\N}$ and $(B_{k}(x))_{k\in\N}$. One crucial statistic we need to
control is the \ac{SO} variance. Define the \emph{oracle error} at a point $x \in \setH$ as
\begin{equation}\label{eq:eps}
\eps(x,\xi)\eqdef \hat{V}(x,\xi)-V(x).
\end{equation}
\begin{assumption}[Oracle Noise]
\label{ass:SObound}
We say that the \ac{SO} 
\begin{itemize}
\item[(i)] is conditionally unbiased if $\Ex_{\xi}[\eps(x,\xi) \vert x]=0$ for all $x\in\setH$;
\item[(ii)]  enjoys a uniform variance bound: $\Ex_{\xi}[\norm{\eps(x,\xi)}^{2}\vert x]\leq\sigma^{2}$ for some $\sigma> 0$ and all $x\in\setH$. 
\end{itemize}
\end{assumption}
Define 
\begin{align*}
U_{k}(\omega)\eqdef\frac{1}{m_{k}}\sum_{t=1}^{m_{k}}\eps(Z_{k}(\omega),\xi_{k}^{(t)}(\omega)),\text{ and }W_{k}(\omega)\eqdef \frac{1}{m_{k}}\sum_{t=1}^{m_{k}}\eps(Y_{k}(\omega),\eta^{(t)}_{k}(\omega)).
\end{align*}
The introduction of these two processes allows us to decompose the random estimator into a mean component and a residual, so that
\begin{align*}
A_{k}(Z_{k})&=V(Z_{k})+U_{k},\text{ and }B_{k}(Y_{k})=V(Y_{k})+W_{k}
\end{align*}
If Assumption \ref{ass:SObound}(i) holds true then $\Ex[W_{k}\vert\hat{\scrF}_{k}]=0=\Ex[U_{k}\vert\scrF_{k}]=0$. Hence, under conditional unbiasedness, the processes $\{(U_{k},\scrF_{k});k\in\N\}$ and $\{(W_{k},\hat{\scrF_{k}});k\in\N\}$ are martingale difference sequences, where the filtrations are defined as $\scrF_{0}\eqdef\hat{\scrF}_{0}\eqdef\scrF_{1}\eqdef \sigma(X_{0},X_{1})$, and iteratively, for $k\geq 1$, 
\begin{align*}
\hat{\scrF}_{k}\eqdef\sigma(X_{0},X_{1},\xi_{1},\eta_{1},\ldots,\eta_{k-1},\xi_{k}),\; \scrF_{k+1}\eqdef \sigma(X_{0},X_{1},\xi_{1},\eta_{1},\ldots,\xi_{k},\eta_{k}).
\end{align*}
Observe that $\scrF_{k}\subseteq\hat{\scrF}_{k}\subseteq\scrF_{k+1}$ for all $k\geq 1$. The uniform variance bound, Assumption \ref{ass:SObound}(ii), ensures that the processes $\{(U_{k},\scrF_{k});k\in\N\}, \{(W_{k},\hat{\scrF_{k}});k\in\N\}$ have finite second moment. 

\begin{remark}
For deriving the stochastic estimates in the analysis to come, it is important to emphasize that $X_{k}$ is $\scrF_{k}$-measurable for all $k\geq 0$, and $Y_{k}$ is $\hat{\scrF}_{k}$-measurable. 
\end{remark}
The mini-batch sampling technology implies an online variance reduction effect, summarized in the next Lemma, whose simple proof we omit.
\begin{lemma}[Variance of the \ac{SO}]
 Suppose Assumption \ref{ass:SObound} holds. Then for $k \geq 1$,
\begin{align}\label{e:sigma}
    \Ex[\norm{W_{k}}^{2}\vert\scrF_{k}]\leq \frac{\sigma^{2}}{m_{k}}\text{ and }\Ex[\norm{U_{k}}^{2}\vert\scrF_{k}]\leq \frac{\sigma^{2}}{m_{k}}, \qquad \Pr-\text{a.s.} 
\end{align}
\end{lemma}
We see that larger sampling rates lead to more precise point estimates of the single-valued operator. This comes at the cost of more evaluations of the stochastic operator. Hence, any mini-batch approach faces a trade-off between the \emph{oracle complexity} and the iteration complexity. We want to use mini-batch estimators to achieve an online variance reduction scheme, motivating the next assumption. 
\begin{assumption}[Batch Size]
\label{ass:batch}
The batch size sequence $(m_{k})_{k\geq 1}$ is non-decreasing and satisfies $\sum_{k=1}^{\infty}\frac{1}{m_{k}}<\infty$. 
\end{assumption}

\section{Analysis}
\label{sec:analysis}
%
This section is organized into three subsections. The first subsection derives asymptotic convergence guarantees, while the second and third subsections provides linear and sublinear rate statements in strongly monotone and monotone regimes, respectively.

\subsection{Asymptotic Convergence}

Given $\lambda>0$, we define the residual function for the monotone inclusion \eqref{eq:MI} as
\begin{equation}\label{eq:res}
\residual_{\lambda}(x)\eqdef \norm{x-J_{\lambda T}(x-\lambda V(x))}.
\end{equation}
Clearly, for every $\lambda>0$, $x\in\setS\iff \residual_{\lambda}(x)=0$.
Hence, $\residual_{\lambda}(\cdot)$ is a \emph{merit function} for the monotone
inclusion problem. To put this merit function into context, let us consider the special case where $T$ is the subdifferential of a lower semi-continuous convex function $g:\setH\to[0,\infty]$, i.e. $T=\partial g$. In this case, the resolvent $J_{\lambda T}$ reduces to the well-known \emph{proximal-operator}
$$
\prox_{\lambda g}(x)\eqdef \argmin_{u\in\setH}\{\lambda g(u)+\frac{1}{2}\norm{u-x}^{2}\}.
$$
In the potential case, where $V(x)=\nabla f(x)$ for some smooth convex function $f:\setH\to\R$, the residual function is thus seen to be a constant multiple of the norm of the so-called \emph{gradient mapping} $\norm{x-\prox_{\lambda g}(x-\lambda V(x))}$, which is a standard merit function in convex \cite{NesConvex} and stochastic \cite{DavDru19,Lan:2020tp} optimization. We use this function to quantify the per-iteration progress of \ac{RISFBF}. The main result of this subsection is the following. 

\begin{theorem}[{\bf Asymptotic Convergence}] \label{th:convergence}
Let $\bar{\alpha},\bar{\eps}\in(0,1)$ be fixed parameters. Suppose that Assumption \ref{ass:exists}-\ref{ass:batch} holds true. Let $(\alpha_{k})_{k\in\N}$ be a non-decreasing sequence such that $\lim_{k\to\infty}\alpha_{k}=\bar{\alpha}$. Let $(\lambda_{k})_{k\in\N}$ be a converging sequence in $(0,\frac{1}{4L})$ such that $\lim_{k\to\infty}\lambda_{k}=\lambda\in(0,\frac{1}{4L})$. If $\rho_{k}=\frac{5(1-\bar{\eps})(1-\bar{\alpha})^{2}}{4(2\alpha_{k}^{2}-\alpha_{k}+1)(1+L\lambda_{k})}$ for all $k\geq 1$, then
\begin{itemize}
\item[(i)] $\lim_{k\to\infty}\residual_{\lambda_{k}}(Z_{k})=0$ in $L^{2}(\Pr)$;
\item[(ii)]  the stochastic process $(X_{k})_{k\in\N}$ generated by algorithm \ac{RISFBF} weakly converges to a $\setS$-valued limiting random variable $X$;
\item[(iii)] $\sum_{k=1}^{\infty}\left[(1-\alpha_{k})\left(\frac{5(1-\alpha_{k}}{4\rho_{k}(1+L\lambda_{k})}-1\right)-2\alpha_{k}^{2}\right]\norm{X_{k}-X_{k-1}}^{2}<\infty\quad \Pr$-a.s.
\end{itemize}
\end{theorem}
We prove this Theorem via a sequence of technical Lemmas. Our proof strategy for Theorem \ref{th:convergence} follows \cite{BotMerStaVu21}, where a stochastic quasi-Fej\'{e}r principle is established, including the residual function. We start with a basic relation.
\begin{lemma}\label{lem:YZG}
For all $k\geq 1$, we have 
\begin{equation}\label{eq:YZG}
    -\norm{Z_{k}-Y_{k}}^{2}\leq \lambda^{2}_{k}\norm{U_{k}}^{2}-\frac{1}{2}\residual^{2}_{\lambda_{k}}(Z_{k}).
\end{equation}
\end{lemma}
\begin{proof}
By definition, 
\begin{align*}
\frac{1}{2}\residual_{\lambda_{k}}^{2}(Z_{k})&=\frac{1}{2}\norm{Z_{k}-J_{\lambda_{k}T}(Z_{k}-\lambda_{k}V(Z_{k}))}^{2}
=\frac{1}{2}\norm{Z_{k}-Y_{k}+J_{\lambda_{k}T}(Z_{k}-\lambda_{k}A_{k}(Z_{k}))-J_{\lambda_{k}T}(Z_{k}-\lambda_{k}V(Z_{k}))}^{2}\\
&\leq \norm{Z_{k}-Y_{k}}^{2}+\norm{J_{\lambda_{k}T}(Z_{k}-\lambda_{k}A_{k}(Z_{k}))-J_{\lambda_{k}T}(Z_{k}-\lambda_{k}V(Z_{k}))}^{2}
\leq \norm{Z_{k}-Y_{k}}^{2}+\lambda^{2}_{k}\norm{U_{k}}^{2},
\end{align*}
where the last inequality uses the non-expansivity property of the resolvent operator. Rearranging terms gives the claimed result. 
\qed
\end{proof}
Next, for a given pair $(p,p^{\ast})\in\gr(F)$, we define the stochastic processes $(\Delta M_{k})_{k\in\N}, (\Delta N_{k}(p,p^{\ast}))_{k\in\N}$, and $(\ce_k)_{k\in\N}$ as
\begin{align}
    &\Delta M_{k}\eqdef \frac{5\rho_{k}\lambda_{k}^{2}}{2(1+L\lambda_{k})} \norm{\ce_{k}}^{2}+\frac{\rho_{k}\lambda^{2}_{k}}{2}\norm{U_{k}}^{2},\label{eq:M}\\
    &\Delta N_{k}(p,p^{\ast})\eqdef 2\rho_{k}\lambda_{k}\inner{W_{k}+p^{\ast},p-Y_{k}},\text{ and }
\label{eq:N}\\
 &\ce_{k}\eqdef W_{k}-U_{k}.
\label{eq:e}
\end{align}
Key to our analysis is the following energy bound on the evolution of the anchor sequence $\left(\norm{X_{k}-p}^{2}\right)_{k\in\N}$.
\begin{lemma}[{\bf Fundamental Recursion}]
\label{lem:Recursion}
Let $(X_{k})_{k\in\N}$ be the stochastic process generated by \ac{RISFBF} with $\alpha_{k}\in(0,1)$, $0\leq \rho_{k}<\frac{5}{4(1+L\lambda_{k})}$, and $\lambda_{k}\in(0,1/4L)$. For all $k\geq 1$ and $(p,p^{\ast})\in\gr(F)$, we have 
\begin{align*}
\norm{X_{k+1}-p}^{2}&\leq (1+\alpha_{k})\norm{X_{k}-p}^{2}-\alpha_{k}\norm{X_{k-1}-p}^{2}-\frac{\rho_{k}}{4}\residual^{2}_{\lambda_{k}}(Z_{k})\\
&+\Delta M_{k}+\Delta N_{k}(p,p^{\ast})-2\rho_{k}\lambda_{k}\inner{V(Y_{k})-V(p),Y_{k}-p}\\
&+\alpha_{k}\norm{X_{k}-X_{k-1}}^{2}\left(2\alpha_{k}+\frac{5(1-\alpha_{k})}{4\rho_{k}(1+L\lambda_{k})}\right)\\
&-(1-\alpha_{k})\left(\frac{5}{4\rho_{k}(1+L\lambda_{k})}-1\right)\norm{X_{k+1}-X_{k}}^{2}.
\end{align*}
\end{lemma}
\begin{proof}
To simplify the notation, let us call $A_{k}\equiv A_{k}(Z_{k})$ and $B_{k}\equiv B_{k}(Y_{k})$. We also introduce the intermediate update $R_{k}\eqdef Y_{k}+\lambda_{k}(A_{k}-B_{k})$. For all $k\geq 0$, it holds true that 
\begin{align*}
\norm{Z_{k}-p}^{2}&=\norm{Z_{k}-Y_{k}+Y_{k}-R_{k}+R_{k}-p}^{2}\\
&=\norm{Z_{k}-Y_{k}}^{2}+\norm{Y_{k}-R_{k}}^{2}+\norm{R_{k}-p}^{2}+2\inner{Z_{k}-Y_{k},Y_{k}-p}+2\inner{Y_{k}-R_{k},R_{k}-p}\\
&=\norm{Z_{k}-Y_{k}}^{2}+\norm{Y_{k}-R_{k}}^{2}+\norm{R_{k}-p}^{2}+2\inner{Z_{k}-Y_{k},Y_{k}-p}+2\inner{Y_{k}-R_{k},R_{k}-p}\\
&=\norm{Z_{k}-Y_{k}}^{2}+\norm{Y_{k}-R_{k}}^{2}+\norm{R_{k}-p}^{2}+2\inner{Z_{k}-Y_{k},Y_{k}-p}+2\inner{Y_{k}-R_{k},Y_{k}-p}\\
&+2\inner{Y_{k}-R_{k},R_k-Y_{k}}\\
&=\norm{Z_{k}-Y_{k}}^{2}+\norm{Y_{k}-R_{k}}^{2}+\norm{R_{k}-p}^{2}+2\inner{Z_{k}-R_{k},Y_{k}-p}+2\inner{Y_{k}-R_{k},R_k-Y_{k}}\\
&=\norm{Z_{k}-Y_{k}}^{2}-\norm{Y_{k}-R_{k}}^{2}+\norm{R_{k}-p}^{2}+2\inner{Z_{k}-R_{k},Y_{k}-p}.
\end{align*}
Since 
\begin{align*}
\norm{Y_{k}-R_{k}}^{2}&=\lambda^{2}_{k}\norm{B_{k}(Y_{k})-Y_{k}(Z_{k})}^{2}\\
&\leq \lambda_{k}^{2}\norm{V(Y_{k})-V(Z_{k})+W_{k+1}-U_{k+1}}^{2}\\
&\leq \lambda_{k}^{2}\norm{V(Y_{k})-V(Z_{k})}^{2}+\lambda_{k}^{2}\norm{W_{k}-U_{k}}^{2}+2\lambda_{k}^{2}\inner{V(Y_{k})-V(Z_{k}),W_{k}-U_{k}}\\
&\leq L^{2}\lambda_{k}^{2}\norm{Y_{k}-Z_{k}}^{2}+\lambda_{k}^{2}\norm{W_{k}-U_{k}}^{2}+2\lambda_{k}^{2}\inner{V(Y_{k})-V(Z_{k}),W_{k}-U_{k}}\\
&\leq 2L^{2}\lambda_{k}^{2}\norm{Y_{k}-Z_{k}}^{2}+2\lambda_{k}^{2}\norm{W_{k}-U_{k}}^{2}.
\end{align*}
Introducing the process $(\ce_{k})_{k\in\N}$ from eq. \eqref{eq:e}, the aforementioned set of inequalities reduces to
\[
\norm{Y_{k}-R_{k}}^{2}\leq 2L^{2}\lambda_{k}^{2}\norm{Y_{k}-Z_{k}}^{2}+2\lambda_{k}^{2}\norm{\ce_{k}}^{2}.
\]
Hence, 
\begin{align*}
\norm{Z_{k}-p}^{2}\geq (1-2L^{2}\lambda_{k}^{2})\norm{Z_{k}-Y_{k}}^{2}-2\lambda_{k}^{2}\norm{\ce_{k}}^{2}+\norm{R_{k}-p}^{2}+2\inner{Z_{k}-R_{k},Y_{k}-p}.
\end{align*}
But $Y_{k}+\lambda_{k}T(Y_{k})\ni Z_{k}-\lambda_{k}A_{k}$, implying that 
\[
\frac{1}{\lambda_{k}}(Z_{k}-Y_{k}-\lambda_{k}A_{k})\in T(Y_{k}).
\]
Pick $(p,p^{\ast})\in\gr(F)$, so that $p^{\ast}-V(p)\in T(p)$. Then, the monotonicity of $T$ yields the estimate 
\[
\inner{\frac{1}{\lambda_{k}}(Z_{k}-Y_{k}-\lambda_{k}A_{k})-p^{\ast}+V(p),Y_{k}-p}\geq 0.
\]
This is equivalent to 
\begin{align} \notag 
& \inner{\frac{1}{\lambda_{k}}(Z_{k}-R_{k}-\lambda_{k}B_{k})-p^{\ast}+V(p),Y_{k}-p}\geq 0, \\
\label{eq:MT}
\mbox{ or } & \inner{Z_{k}-R_{k},Y_{k}-p}\geq \lambda_{k}\inner{W_{k}+p^{\ast},Y_{k}-p}+\lambda_{k}\inner{V(Y_{k})-V(p),Y_{k}-p}.
\end{align}
This implies that
\begin{align*}
\inner{Z_{k}-R_{k},Y_{k}-x^{\ast}}\geq\lambda_{k}\inner{W_{k},Y_{k}-x^{\ast}}.
\end{align*}
Hence, we obtain the following, 
\begin{align*}
    \norm{Z_{k}-p}^{2}& \geq (1-2L^{2}\lambda^{2}_{k})\norm{Y_{k}-Z_{k}}^{2}+\norm{R_{k}-p}^{2}-2\lambda^{2}_{k}\norm{\ce_{k}}^{2}\\
&+2\lambda_{k}\inner{W_{k}+p^{\ast},Y_{k}-p}+2\lambda_{k}\inner{V(Y_{k})-V(p),Y_{k}-p}.
\end{align*}
Rearranging terms, we arrive at the following bound on $\norm{R_{k}-p}^{2}$:
\begin{align}\label{eq:Rk}
\norm{R_{k}-p}^{2}\leq& \norm{Z_{k}-p}^{2}-(1-2L^{2}\lambda_{k}^{2})\norm{Y_{k}-Z_{k}}^{2}+2\lambda^{2}_{k}\norm{\ce_{k}}^{2}+2\lambda_{k}\inner{W_{k}+p^{\ast},p-Y_{k}}\\
&+2\lambda_{k}\inner{V(Y_{k})-V(p),p-Y_{k}}\nonumber
\end{align}
Next, we observe that $\norm{X_{k+1}-p}^{2}$ may be bounded as follows.
\begin{align}
\norm{X_{k+1}-p}^{2}&=\norm{(1-\rho_{k})Z_{k}+\rho_{k}R_{k}-p}^{2}\notag\\
&=\norm{(1-\rho_{k})(Z_{k}-p)-\rho_{k}(R_{k}-p)}^{2}\notag\\
&=(1-\rho_{k})^2\norm{Z_{k}-p}^{2}+\rho_{k}^2\norm{R_{k}-p}^{2}- 2\rho_k(1-\rho_k) \inner{Z_k-p,R_k-p} \notag\\
&=(1-\rho_{k})\norm{Z_{k}-p}^{2}-\rho_{k}(1-\rho_{k})\norm{Z_{k}-p}^{2}+\rho_{k}\norm{R_{k}-p}^{2}-\rho_{k}(1-\rho_{k})\norm{R_{k}-p}^2 \notag\\
&  + 2\rho_k(1-\rho_k) \inner{Z_k-p,R_k-p} \notag\\
&=(1-\rho_{k})\norm{Z_{k}-p}^{2}+\rho_{k}\norm{R_{k}-p}^{2}-\rho_{k}(1-\rho_{k})\norm{R_{k}-Z_{k}}^{2}\notag \\
&=(1-\rho_{k})\norm{Z_{k}-p}^{2}+\rho_{k}\norm{R_{k}-p}^{2}-\frac{1-\rho_{k}}{\rho_{k}}\norm{X_{k+1}-Z_{k}}^{2}. \label{bd-eq1}
\end{align}
We may then derive a bound on the expression in~\eqref{bd-eq1}, 
\begin{align}
\notag& \quad (1-\rho_{k})\norm{Z_{k}-p}^{2}+\rho_{k}\norm{R_{k}-p}^{2}-\frac{1-\rho_{k}}{\rho_{k}}\norm{X_{k+1}-Z_{k}}^{2}\\
\notag&\leq \norm{Z_{k}-p}^{2}-\frac{1-\rho_{k}}{\rho_{k}}\norm{X_{k+1}-Z_{k}}^{2}-\rho_{k}(1-2L^{2}\lambda^{2}_{k})\norm{Z_{k}-Y_{k}}^{2}\\
&+2\lambda^{2}_{k}\rho_{k}\norm{\ce_{k}}^{2}-2\rho_{k}\lambda_{k}\inner{W_{k}+p^{\ast},Y_{k}-p}+2\rho_{k}\lambda_{k}\inner{V(Y_{k})-V(p),p-Y_{k}}  \label{bd-eq1sc} \\
\notag&=\norm{Z_{k}-p}^{2}-\frac{1-\rho_{k}}{\rho_{k}}\norm{X_{k+1}-Z_{k}}^{2}-\rho_{k}(1/2-2L^{2}\lambda^{2}_{k})\norm{Z_{k}-Y_{k}}^{2}\\
\notag&+2\lambda^{2}_{k}\rho_{k}\norm{\ce_{k}}^{2}-2\rho_{k}\lambda_{k}\inner{W_{k}+p^{\ast},Y_{k}-p}\\
& -\frac{\rho_{k}}{2}\norm{Y_{k}-Z_{k}}^{2}+2\rho_{k}\lambda_{k}\inner{V(Y_{k})-V(p),p-Y_{k}}. \label{l-lambda}
\end{align}
By invoking \eqref{eq:YZG}, we arrive at the estimate  
\begin{align*}
\norm{X_{k+1}-p}^{2}&\leq \norm{Z_{k}-p}^{2}-\frac{1-\rho_{k}}{\rho_{k}}\norm{X_{k+1}-Z_{k}}^{2}+2\lambda^{2}\rho_{k}\norm{\ce_{k}}^{2}-2\rho_{k}\lambda_{k}\inner{W_{k}+p^{\ast},Y_{k}-p} \nonumber\\
&-\rho_{k}(1/2-2L^{2}\lambda^{2}_{k})\norm{Y_{k}-Z_{k}}^{2} -\frac{\rho_{k}}{4}\residual^{2}_{\lambda_{k}}(Z_{k})\\
    & +\frac{\rho_{k}\lambda^{2}_{k}}{2}\norm{U_{k}}^{2}+2\rho_{k}\lambda_{k}\inner{V(Y_{k})-V(p),p-Y_{k}}.
\end{align*}
Furthermore,
\begin{align*}
\frac{1}{\rho_{k}}\norm{X_{k+1}-Z_{k}} &=\norm{R_{k}-Z_{k}}\leq\norm{R_{k}-Y_{k}}+\norm{Y_{k}-Z_{k}}\\
&\leq \lambda_{k}\norm{B_{k}-A_{k}}+\norm{Y_{k}-Z_{k}}\\
&\leq(1+ L\lambda_{k})\norm{Y_{k}-Z_{k}}+\lambda_{k}\norm{\ce_{k}},
\end{align*}
which implies that
\begin{align}
\frac{1}{2\rho^{2}_{k}}\norm{X_{k+1}-Z_{k}}^{2}\leq (1+L \lambda_{k})^{2}\norm{Y_{k}-Z_{k}}^{2}+\lambda^{2}_{k}\norm{\ce_{k}}^{2}. \label{bd-eq2sc}
\end{align}
Multiplying both sides by $\frac{\rho_{k}(1/2-2L\lambda_{k})}{1+L\lambda_{k}}$, a positive scalar since $\lambda_k \in (0,\tfrac{1}{4L})$, we obtain 
\begin{align}
\notag\frac{1/2-2L\lambda_{k}}{2\rho_{k}(1+L\lambda_{k})}\norm{X_{k+1}-Z_{k}}^{2}&\leq\rho_{k}(1/2-2L\lambda_{k})(1+L\lambda_{k})\norm{Y_{k}-Z_{k}}^{2}\\
&+\frac{\lambda^{2}_{k}\rho_{k}(1/2-2L\lambda_{k})}{1+L\lambda_{k}}\norm{\ce_{k}}^{2}. \label{l-lambda-2}
\end{align}
Rearranging terms, and noting that $(1/2-2L\lambda_{k})(1+L\lambda_{k})\leq 1/2-2L^{2}\lambda^{2}_{k}$, the above estimate becomes
\begin{align} \label{bd-eq2} \notag
-\rho_{k}(1/2-2L^{2}\lambda^{2}_{k})\norm{Y_{k}-Z_{k}}^{2}&\leq -\frac{1/2-2L\lambda_{k}}{2\rho_{k}(1+L\lambda_{k})}\norm{X_{k+1}-Z_{k}}^{2}\\
	& +\frac{\rho_{k}\lambda^{2}_{k}(1/2-2L\lambda_{k})}{1+L\lambda_{k}}\norm{\ce_{k}}^{2}.
\end{align}
Substituting this bound into the first majorization of the anchor process $\norm{X_{k+1}-p}^{2}$, we see
\begin{align*}
\norm{X_{k+1}-p}^{2}&\leq \norm{Z_{k}-p}^{2}-\left(\frac{1-\rho_{k}}{\rho_{k}}+\frac{1/2-2L\lambda_{k}}{2\rho_{k}(1+L\lambda_{k})}\right)\norm{X_{k+1}-Z_{k}}^{2}\\
&+\rho_{k}\lambda_{k}^{2}\norm{\ce_{k}}^{2}\left(2+\frac{1/2-2L\lambda_{k}}{1+L\lambda_{k}}\right)+2\rho_{k}\lambda_{k}\inner{V(Y_{k})-V(p),p-Y_{k}}\\
&-2\rho_{k}\lambda_{k}\inner{W_{k}+p^{\ast},Y_{k}-p}-\frac{\rho_{k}}{4}\residual^{2}_{\lambda_{k}}(Z_{k})+\frac{\rho_{k}\lambda^{2}_{k}}{2}\norm{U_{k}}^{2}\\
&=\norm{Z_{k}-p}^{2}-\frac{\rho_{k}}{4}\residual^{2}_{\lambda_{k}}(Z_{k})+\frac{\rho_{k}\lambda^{2}_{k}}{2}\norm{U_{k}}^{2}-2\rho_{k}\lambda_{k}\inner{W_{k}+p^{\ast},Y_{k}-p}\\
&-\frac{5/2-2\rho_{k}(1+L\lambda_{k})}{2\rho_{k}(1+L\lambda_{k})}\norm{X_{k+1}-Z_{k}}^{2}\\
& +\frac{5\rho_{k}\lambda_{k}^{2}}{2(1+L\lambda_{k})} \norm{\ce_{k}}^{2}+2\rho_{k}\lambda_{k}\inner{V(Y_{k})-V(p),p-Y_{k}}.
\end{align*}
Observe that 
\begin{align}
\norm{X_{k+1}-Z_{k}}^{2}&=\norm{(X_{k+1}-X_{k})-\alpha_{k}(X_{k}-X_{k-1})}^{2} \nonumber\\
&\geq (1-\alpha_{k})\norm{X_{k+1}-X_{k}}^{2}+(\alpha^{2}_{k}-\alpha_{k})\norm{X_{k}-X_{k-1}}^{2},
\label{eq:Xk1}
\end{align}
and Lemma \ref{lem:ab} gives
\begin{equation}\label{eq:Z1}
\norm{Z_{k}-p}^{2}=(1+\alpha_{k})\norm{X_{k}-p}^{2}-\alpha_{k}\norm{X_{k-1}-p}^{2}+\alpha_{k}(1+\alpha_{k})\norm{X_{k}-X_{k-1}}^{2}.
\end{equation}
By hypothesis, $\alpha_{k},\rho_{k},\lambda_{k}$ are defined such that $\frac{5/2-2\rho_{k}(1+L\lambda_{k})}{2\rho_{k}(1+L\lambda_{k})}>0$. Then, using both of these relations in the last estimate for $\norm{X_{k+1}-p}^{2}$, we arrive at 
\begin{align*}
\norm{X_{k+1}-p}^{2}&\leq (1+\alpha_{k})\norm{X_{k}-p}^{2}-\alpha_{k}\norm{X_{k-1}-p}^{2}+\alpha_{k}(1+\alpha_{k})\norm{X_{k}-X_{k-1}}^{2}-2\rho_{k}\lambda_{k}\inner{W_{k+1}+p^{\ast},Y_{k}-p}\\
&-\frac{\rho_{k}}{4}\residual^{2}_{\lambda_{k}}(Z_{k})+\frac{5\rho_{k}\lambda_{k}^{2}}{2(1+L\lambda_{k})}\norm{\ce_{k}}^{2}+\frac{\rho_{k}\lambda^{2}_{k}}{2}\norm{U_{k}}^{2}+2\rho_{k}\lambda_{k}\inner{V(Y_{k})-V(p),p-Y_{k}}\\
&-\left(\frac{5}{4\rho_{k}(1+L\lambda_{k})}-1\right)\left[(1-\alpha_{k})\norm{X_{k+1}-X_{k}}^{2}+(\alpha^{2}_{k}-\alpha_{k})\norm{X_{k}-X_{k-1}}^{2}\right].
\end{align*}
Using the respective definitions of the stochastic increments $\Delta M_{k},\Delta N_{k}(p,p^{\ast})$ in \eqref{eq:M} and \eqref{eq:N}, we arrive at
\begin{align}
\notag
\norm{X_{k+1}-p}^{2}&\leq (1+\alpha_{k})\norm{X_{k}-p}^{2}-\alpha_{k}\norm{X_{k-1}-p}^{2}-\frac{\rho_{k}}{4}\residual^{2}_{\lambda_{k}}(Z_{k})\\
&\notag+\Delta M_{k}+\Delta N_{k}(p,p^{\ast})-2\rho_{k}\lambda_{k}\inner{V(Y_{k})-V(p),Y_{k}-p}\\
&\notag+\alpha_{k}\norm{X_{k}-X_{k-1}}^{2}\left(2\alpha_{k}+\frac{5(1-\alpha_{k})}{4\rho_{k}(1+L\lambda_{k})}\right)\\
\label{recur-Xk} &-(1-\alpha_{k})\left(\frac{5}{4\rho_{k}(1+L\lambda_{k})}-1\right)\norm{X_{k+1}-X_{k}}^{2}.
\end{align}
\qed
\end{proof}

Recall that $Y_{k}$ is $\hat{\scrF}_{k}$-measurable. By the law of iterated expectations, we therefore see 
\begin{align*}
\Ex[\Delta N_{k}(p,p^{\ast})\vert\scrF_{k}]=\Ex\left\{\Ex[\Delta N_{k}(p,p^{\ast})\vert\hat{\scrF}_{k}]\vert\scrF_{k}\right\}=2\rho_{k}\lambda_{k} \Ex[\inner{p^{\ast},p-Y_{k}}\vert\scrF_{k}], 
\end{align*}
for all $(p,p^{\ast})\in\gr(F)$. Observe that if we choose $(p,0)\in\gr(F)$, meaning that $p\in\setS$, then $\Delta N_{k}(p,0)\equiv\Delta N_{k}(p)$ is a martingale difference sequence. Furthermore, for all $k\geq 1$, 
\begin{equation}\label{eq:boundM}
\Ex[\Delta M_{k}\vert\scrF_{k}]\leq \frac{5\rho_{k}\lambda_{k}^2}{1+L\lambda_{k}}\Ex[\norm{W_{k}}^{2}\vert\scrF_{k}]+\lambda^{2}_{k}\left(\frac{5\rho_{k}}{1+L\lambda_{k}}+\frac{\rho_{k}}{2}\right)\Ex[\norm{U_{k}}^{2}\vert\scrF_{k}]\leq \frac{\ca_{k}\sigma^{2}}{m_{k}},  
\end{equation}
where $\ca_{k}\eqdef\lambda^{2}_{k}\left( \frac{10\rho_{k}}{1+L\lambda_{k}}+\frac{\rho_{k}}{2}\right)$.\\

To prove the a.s. convergence of the stochastic process $(X_{k})_{k\in\N}$, we rely on the following preparations. Motivated by the analysis of deterministic inertial schemes, we are interested in a regime under which $\alpha_{k}$ is non-decreasing.

For a fixed reference point $p\in\setH$, define the anchor sequences $\phi_{k}(p)\eqdef \frac{1}{2}\norm{X_{k}-p}^{2}$, and the energy sequence $\Delta_{k}\eqdef \frac{1}{2}\norm{X_{k}-X_{k-1}}^{2}.$ In terms of these sequences, we can rearrange the fundamental recursion from Lemma \ref{lem:Recursion} to obtain 
\begin{align*}
\phi_{k+1}(p)-&\alpha_{k}\phi_{k}(p)-(1-\alpha_{k})\left(\frac{5}{4\rho_{k}(1+L\lambda_{k})}-1\right)\Delta_{k+1}\leq \phi_{k}(p)-\alpha_{k}\phi_{k-1}(p)\\
&-(1-\alpha_{k})\left(\frac{5}{4\rho_{k}(1+L\lambda_{k})}-1\right)\Delta_{k}+\frac{1}{2}\Delta M_{k}+\frac{1}{2}\Delta N_{k}(p,p^{\ast})\\
&-\rho_{k}\lambda_{k}\inner{V(Y_{k})-V(p),Y_{k}-p}+\Delta_{k}\left[2\alpha^{2}_{k}+(1-\alpha_{k})\left(1-\frac{5(1-\alpha_{k})}{4\rho_{k}(1+L\lambda_{k})}\right)\right]\\
& -\frac{\rho_{k}}{8}\residual^{2}_{\lambda_{k}}(Z_{k}).
\end{align*} 
For a given pair $(p,p^{\ast})\in\gr(F)$, define 
 \begin{equation}
Q_{k}(p)\eqdef \phi_{k}(p)-\alpha_{k}\phi_{k-1}(p)+(1-\alpha_{k})\left(\frac{5}{4\rho_{k}(1+L\lambda_{k})}-1\right)\Delta_{k}.
 \end{equation}
Then, in terms of the sequence 
\begin{equation}
\beta_{k+1}\eqdef (1-\alpha_{k})\left(\frac{5}{4\rho_{k}(1+L\lambda_{k})}-1\right)-(1-\alpha_{k+1})\left(\frac{5}{4\rho_{k+1}(1+L\lambda_{k+1})}-1\right), \label{beta}
\end{equation}
and using the monotonicity of $V$, guaranteeing that $\inner{V(Y_{k})-V(p),Y_{k}-p}\geq 0$, we get 
\begin{align*}
Q_{k+1}(p)&\leq Q_{k}(p)-\frac{\rho_{k}}{8}\residual^{2}_{\lambda_{k}}(Z_{k})+\frac{1}{2}\Delta M_{k}+\frac{1}{2}\Delta N_{k}(p,p^{\ast})+(\alpha_{k}-\alpha_{k+1})\phi_{k}(p)\\
&+\left[2\alpha^{2}_{k}+(1-\alpha_{k})\left(1-\frac{5(1-\alpha_{k})}{4\rho_{k}(1+L\lambda_{k})}\right)\right]\Delta_{k}-\beta_{k+1}\Delta_{k+1}.
\end{align*}
Defining
\[
\theta_{k}\eqdef \frac{\rho_{k}}{8}\residual^{2}_{\lambda_{k}}(Z_{k})-\left[2\alpha^{2}_{k}+(1-\alpha_{k})\left(1-\frac{5(1-\alpha_{k})}{4\rho_{k}(1+L\lambda_{k})}\right)\right]\Delta_{k},
\]
 we arrive at 
\begin{equation}\label{eq:Q1}
Q_{k+1}(p)\leq Q_{k}(p)-\theta_{k}+\frac{1}{2}\Delta M_{k}+\frac{1}{2}\Delta N_{k}(p,p^{\ast})+(\alpha_{k}-\alpha_{k+1})\phi_{k}(p)-\beta_{k}\Delta_{k+1}.
\end{equation}
Our aim is to use $Q_{k}(p)$ as a suitable energy function for RISFBF. For that to work, we need to identify a specific parameter sequence pair $(\rho_{k},\alpha_{k})$ so that $\beta_{k}\geq 0$ and $\theta_{k} \geq 0$, taking the following design criteria into account:
\begin{enumerate}
\item $\alpha_{k}\in(0,\bar{\alpha}]\subset(0,1)$ for all $k\geq 1$;
\item $\alpha_{k}$ is non-decreasing with 
\begin{equation}
\sup_{k\geq 1}\alpha_{k}=\bar{\alpha},\text{ and} \inf_{k\geq 1}\alpha_{k}>0.
\end{equation}
\end{enumerate}
Incorporating these two restrictions on the inertia parameter $\alpha_{k}$, we are left with the following constraints: 
\begin{equation}\label{eq:signs}
\beta_{k}\geq 0\text{ and }2\alpha_{k}^{2}+(1-\alpha_{k})\left(1-\frac{5(1-\alpha_{k})}{4\rho_{k}(1+L\lambda_{k})}\right)\leq 0.
\end{equation}
To identify a constellation of parameters $(\alpha_{k},\rho_{k})$ satisfying these two conditions, define 
\begin{equation}\label{def-hk}
h_{k}(x,y)\eqdef (1-x)\left(\frac{5}{4y(1+L\lambda_{k})}-1\right).
\end{equation}
Then, 
\begin{align*}
0&\geq 2\alpha_{k}^{2}-(1-\alpha_{k})\left(h_{k}(\alpha_{k},\rho_{k})+(1-\alpha_{k})-1\right)\\
&=2\alpha^{2}_{k}+\alpha_{k}(1-\alpha_{k})-(1-\alpha_{k})h_{k}(\alpha_{k},\rho_{k})\\
&=\alpha_{k}(1+\alpha_{k})-(1-\alpha_{k})h_{k}(\alpha_{k},\rho_{k}),
\end{align*}
which gives 
\begin{equation}\label{eq:h1}
h_{k}(\alpha_{k},\rho_{k})\geq\frac{\alpha_{k}(1+\alpha_{k})}{1-\alpha_{k}}.
\end{equation}
Solving this condition for $\rho_{k}$ reveals that $\frac{1}{\rho_{k}}\geq \frac{4(2\alpha^{2}_{k}-\alpha_{k}+1)(1+L\lambda_{k})}{5(1-\alpha_{k})^{2}}.$ Using the design condition $\alpha_{k}\leq\bar{\alpha}<1$, we need to choose the relaxation parameter $\rho_{k}$ so that $\rho_{k}\leq\frac{5(1-\alpha_k)^{2}}{4(1+L\lambda_{k})(2\alpha^{2}_{k}-\alpha_{k}+1)}$. This suggests to use the relaxation sequence $\rho_{k}=\rho_k(\alpha_k,\lambda_k)\eqdef \frac{5(1-\bar{\eps})(1-\bar{\alpha})^{2}}{4(1+L\lambda_{k})(2\alpha^{2}_{k}-\alpha_{k}+1)}$. It remains to verify that with this choice we can guarantee $\beta_{k}\geq 0.$ This can be deduced as follows: Recalling \eqref{def-hk}, we get
\[
    h_k(\alpha_k,\rho_k) = (1-\alpha_{k})\left(\tfrac{5}{4\rho_{k}(1+L\lambda)}-1\right)=\tfrac{(1-\alpha_k)(2\alpha_k^2-\alpha_k+1)}{(1-\bar{\eps})(1-\bar{\alpha})^{2}}+\alpha_{k}-1.
\]
In particular, we note that if $f(\alpha) \triangleq  {\tfrac{(1-\alpha)(2\alpha^2-\alpha+1)}{(1-\bar{\eps})(1-\bar{\alpha})^2}}+\alpha-1$, then
\begin{align*}
  f'(\alpha) &= \tfrac{(1-\alpha)(4\alpha-1)-(2\alpha^2-\alpha+1)}{(1-\bar{\eps})(1-\bar{\alpha})^2}+1= \tfrac{-6\alpha^2+6\alpha-2+(1-\bar{\eps})(1-\bar{\alpha})^2}{(1-\bar{\eps})(1-\bar{\alpha})^2} = \tfrac{-6(\alpha-\frac{1}{2})^2-\frac{1}{2}+(1-\bar{\eps})(1-\bar{\alpha})^2}{(1-\bar{\eps})(1-\bar{\alpha})^2}
  \end{align*}
  We consider two cases:\\
Case 1: $\bar{\alpha}\leq 1/2$. In this case 
\begin{align*}
f'(\alpha)\leq \tfrac{-6(\bar{\alpha}-\frac{1}{2})^2-\frac{1}{2}+(1-\bar{\eps})(1-\bar{\alpha})^2}{(1-\bar{\eps})(1-\bar{\alpha})^{2}}\leq \frac{-5\bar{\alpha}^{2}+4\bar{\alpha}-1}{(1-\bar{\eps})(1-\bar{\alpha})^2}<0.
\end{align*}
Case 2: $1/2<\bar{\alpha}<1$. In this case 
\[
f'(\alpha)\leq \frac{-6(1/2-1/2)^{2}-1/2+(1-\bar{\eps})(1-\bar{\alpha})^{2}}{(1-\bar{\eps})(1-\bar{\alpha})^{2}}\leq \frac{-1/2+(1-\bar{\eps})(1-1/2)^{2}}{(1-\bar{\eps})(1-\bar{\alpha})^{2}}<0.
\]
Thus, $f(\alpha)$ is decreasing in $\alpha \in (0,\bar{\alpha}]$, where $0<\bar{\alpha}<1$. \\

Using these relations, we see that \eqref{eq:Q1} reduces to 
\begin{equation}\label{eq:Q2}
Q_{k+1}(p)\leq Q_{k}(p)-\theta_{k}+\frac{1}{2}\Delta M_{k}+\frac{1}{2}\Delta N_{k}(p,p^{\ast}),
\end{equation}
where $\theta_{k}\geq 0$. This is the basis for our proof of Theorem \ref{th:convergence}. 

\begin{proof}[Proof of Theorem \ref{th:convergence}]
We start with (i). Consider \eqref{eq:Q2}, with the special choice $p^{\ast}=0$, so that $p\in\setS$. Taking conditional expectations on both sides of this inequality, we arrive at 
\[
\Ex[Q_{k+1}(p)\vert\scrF_{k}]\leq Q_{k}(p)-\theta_{k}+\psi_{k},
\]
where $\psi_{k}\eqdef \frac{\ca_{k}\sigma^{2}}{2m_{k}}$. By design of the relaxation sequence $\rho_{k}$, we see that
\begin{align*}
\ca_{k}&=\lambda^{2}_{k}\left( \frac{10\rho_{k}}{1+L\lambda_{k}}+\frac{\rho_{k}}{2}\right)=\lambda^{2}_{k}\left( \frac{10}{1+L\lambda_{k}}+\frac{1}{2}\right)\frac{5(1-\bar{\eps})}{4(2\alpha^{2}_{k}-\alpha_{k}+1)(1+L\lambda_{k})}.
\end{align*}
Since $\lim_{k\to\infty}\lambda_{k}=\lambda\in(0,1/4L)$, and $\lim_{k\to\infty}\alpha_{k}=\bar{\alpha}\in(0,1)$, we conclude that the sequence $(\ca_{k})_{k\in\N}$ is bounded. Consequently, thanks to Assumption \ref{ass:batch}, the sequence $(\psi_{k})_{k\in\N}$ is in $\ell^{1}_{+}(\N)$. We next claim that $Q_{k}(p)\geq 0$. To verify this, note that 
\begin{align*}
Q_{k}(p)&=\frac{1}{2}\norm{X_{k}-p}^{2}-\frac{\alpha_{k}}{2}\norm{X_{k-1}-p}^{2}+\frac{(1-\alpha_{k})}{2}\left(\frac{5}{4\rho_{k}(1+L\lambda_{k})}-1\right)\norm{X_{k}-X_{k-1}}^{2}\\
&=\frac{1}{2}\norm{X_{k}-p}^{2}+\left(\frac{(1-\alpha_{k})(2\alpha_{k}^{2}+1-\alpha_{k})}{(1-\bar{\eps})(1-\bar{\alpha})^{2}}-1+\alpha_{k}\right)\frac{1}{2}\norm{X_{k}-X_{k-1}}^{2}-\frac{\alpha_{k}}{2}\norm{X_{k-1}-p}^{2}\\
&\geq \frac{1}{2}\norm{X_{k}-p}^{2}+\left(\frac{(1-\alpha_{k})(\alpha_{k}^{2}+1-\alpha_{k})}{(1-\bar{\eps})(1-\bar{\alpha})^{2}}-1+\alpha_{k}\right)\frac{1}{2}\norm{X_{k}-X_{k-1}}^{2}-\frac{\alpha_{k}}{2}\norm{X_{k-1}-p}^{2}\\
&\geq \frac{1}{2}\norm{X_{k}-p}^{2}+\left(\frac{(1-\alpha_{k})(\alpha_{k}^{2}+1-\alpha_{k})}{(1-\alpha_{k})^{2}}-1+\alpha_{k}\right)\frac{1}{2}\norm{X_{k}-X_{k-1}}^{2}-\frac{\alpha_{k}}{2}\norm{X_{k-1}-p}^{2}\\
&=(\alpha_{k}+(1-\alpha_{k}))\norm{X_{k}-p}^{2}+\left(\alpha_{k}+\frac{\alpha^{2}_{k}}{1-\alpha_{k}}\right)\norm{X_{k}-X_{k-1}}^{2}-\alpha_{k}\norm{X_{k-1}-p}^{2}\\
&\geq \frac{\alpha_{k}}{2}\left(\norm{X_{k}-p}^{2}+\norm{X_{k}-X_{k-1}}^{2}\right)-\frac{\alpha_{k}}{2}\norm{X_{k-1}-p}^{2}+\alpha_{k}\norm{X_{k}-p}\cdot\norm{X_{k}-X_{k-1}}\\
&\geq \frac{\alpha_{k}}{2}\left(\norm{X_{k}-p}+\norm{X_{k}-X_{k-1}}\right)^{2}-\frac{\alpha_{k}}{2}\norm{X_{k-1}-p}^{2}\geq 0.
\end{align*}
where the first and second inequality uses $\bar{\eps}<1$ and $\alpha_{k}\leq\bar{\alpha}\in(0,1)$, the third inequality makes use of the Young inequality: $\frac{1-a}{2a}\norm{X_{k}-p}^{2}+\frac{a}{2(1-a)}\norm{X_{k}-X_{k-1}}^{2}\geq \norm{X_{k}-p}\cdot\norm{X_{k}-X_{k-1}}$. Finally, the fourth inequality uses the triangle inequality $\norm{X_{k-1}-p}\leq \norm{X_{k}-X_{k-1}}+\norm{X_{k}-p}$. Lemma \ref{lem:RS} readily yields the existence of an a.s. finite limiting random variable $Q_{\infty}(p)$ such that $Q_{k}(p)\to Q_{\infty}(p)$, $\Pr$-a.s., and $(\theta_{k})_{k\in\N}\in\ell^{1}_{+}(\F)$. Since $\lambda_{k}\to \lambda$, we get $\lim_{k\to\infty}\rho_{k}=\frac{5(1-\bar{\eps})(1-\bar{\alpha})^{2}}{4(1+L\lambda)(2\bar{\alpha}^{2}+1-\bar{\alpha})}$. Hence,
\begin{align*}
&\lim_{k\to\infty}\left(2\alpha_{k}^2-(1-\alpha_{k})\left(1-\tfrac{5(1-\alpha_{k})}{4\rho_{k}(1+L\lambda_{k})}\right)\right)\norm{X_{k}(\omega)-X_{k-1}(\omega)}^2=0,\text{ and }\\
&\lim_{k\to\infty}\frac{\rho_{k}}{4}\residual^{2}_{\lambda_{k}}(Z_{k}(\omega))=0.
\end{align*}
$\Pr$-a.s. We conclude that $\lim_{k\to\infty}\residual^{2}_{\lambda_{k}}(Z_{k})=0$, $\Pr$-a.s..

To prove (ii) observe that, since $\bar{\eps}\in(0,1)$ and $\lim_{k\to\infty}\alpha_{k}=\bar{\alpha}$, it follows 
\[
\left[2\alpha_{k}^{2}+(1-\alpha_{k})\left(1-\frac{5(1-\alpha_{k})}{4\rho_{k}(1+L\lambda_{k})}\right)\right]\leq \frac{-\bar{\eps}}{1-\bar{\eps}}(2\bar{\alpha}^{2}+1-\bar{\alpha})<0.
\]
Consequently,  $\lim_{k\to\infty}\norm{X_{k}-X_{k-1}}^2 =0$, $\Pr$-a.s., and  $\left(\phi_{k}(p)-\alpha_{k}\phi_{k-1}(p)\right)_{k\in\N}$ is almost surely bounded. Hence, for each $\omega\in\Omega$, there exists a bounded random variable $C_{1}(\omega)\in[0,\infty)$ such that 
\[
\phi_{k}(p,\omega)\leq C_{1}(\omega)+\alpha_{k}\phi_{k-1}(p,\omega)\leq C_{1}(\omega)+\bar{\alpha}\phi_{k-1}(p,\omega)\qquad\forall k\geq 1.
\]
Iterating this relation, using the fact that $\bar{\alpha}\in[0,1)$, we easily derive 
\[
\phi_{k}(p,\omega)\leq \frac{C_{1}(\omega)}{1-\bar{\alpha}}+\bar{\alpha}^{k}\phi_{1}(p,\omega).
\]
Hence, $(\phi_{k}(p))_{k\in\N}$ is a.s. bounded, which implies that $(X_{k})_{k\in\N}$ is bounded $\Pr$-a.s. We next claim that $(\norm{X_{k}-p})_{k\in\N}$ converges to a $[0,\infty)$-valued random variable $\Pr$-a.s. Indeed, take $\omega\in\Omega$ such that $\phi_{k}(p,\omega)\equiv\phi_{k}(\omega)$ is bounded. Suppose there exists $\ct_{1}(\omega)\in[0,\infty),\ct_{2}(\omega)\in[0,\infty)$, and subsequences $(\phi_{k_{j}}(\omega))_{j\in\N}$ and $(\phi_{l_{j}}(\omega))_{j\in\N}$ such that $\phi_{k_{j}}(\omega)\to\ct_{1}(\omega)$ and $\phi_{l_{j}}(\omega)\to\ct_{2}(\omega)>\ct_{1}(\omega)$. Then, $\lim_{j\to\infty}Q_{k_{j}}(p)(\omega)=Q_{\infty}(p)(\omega)=(1-\bar{\alpha})\ct_{1}(\omega)<(1-\bar{\alpha})\ct_{2}(\omega)=\lim_{j\to\infty}Q_{l_{j}}(\omega)=Q_{\infty}(p)(\omega)$, a contradiction. It follows that $\ct_{1}(\omega)=\ct_{2}(\omega)$ and, in turn, $\phi_{k}(\omega)\to\ct(\omega)$. Thus, for each $p\in\setS$, $\phi_{k}(p)\to\ct$ $\Pr$-a.s.\\
Since we assume that $\setH$ is separable, \cite[Prop 2.3(iii)]{ComPes15} guarantees that there exists a set $\Omega_{0}\in\scrF$ with $\Pr(\Omega_{0})=1$, and, for every $\omega\in\Omega_{0}$ and every $p\in\setS$, the sequence $(\norm{X_{k}(\omega)-p})_{k\in\N}$ converges. 

We next show that all weak limit points of $(X_{k})_{k\in\N}$ are contained in $\setS$. Let $\omega\in\Omega$ such that $(X_{k}(\omega))_{k\in\N}$ is bounded. Thanks to \cite[Lemma 2.45]{BauCom16}, we can find a weakly convergent subsequence $(X_{k_{j}}(\omega))_{j\in\N}$ with limit $\chi(\omega)$, i.e. for all
$u\in\setH$ we have
$\lim_{j\to\infty}\inner{X_{k_{j}}(\omega),u}=\inner{\chi(\omega),u}$.
This implies 
\[
\lim_{j\to\infty}\inner{Z_{k_{j}}(\omega),u}=\lim_{j\to\infty}\inner{X_{k_{j}}(\omega),u}+\lim_{j\to\infty}\alpha_{k_{j}}\inner{X_{k_{j}}(\omega)-X_{k_{j-1}}(\omega),u}=\inner{\chi(\omega),u},
\]
showing that $Z_{k_{j}}(\omega)\wlim \chi(\omega)$. Along this weakly converging subsequence, define 
\[
r_{k_{j}}(\omega)\eqdef Z_{k_{j}}(\omega)-J_{\lambda_{k_{j}}T}(Z_{k_{j}}(\omega)-\lambda_{k_{j}}V(Z_{k_{j}}(\omega))).
\]
Clearly, $\residual_{\lambda_{k_{j}}}(Z_{k_{j}}(\omega))=\norm{r_{k_{j}}(\omega)}$, so that $\lim_{j\to\infty}r_{k_{j}}(\omega)=0$. By definition 
\[
\frac{1}{\lambda_{k_{j}}}r_{k_{j}}(\omega)-V(Z_{k_{j}}(\omega))+V\left(Z_{k_{j}}(\omega)-r_{k_{j}}(\omega)\right)\in F(Z_{k_{j}}(\omega)-r_{k_{j}}(\omega)).
\]
Since $V$ and $F=T+V$ are maximally monotone, their graphs are sequentially closed in the weak-strong topology $\setH^{\text{weak}}\times\setH^{\text{strong}}$ \cite[Prop. 20.33(ii)]{BauCom16}. Therefore, by the strong convergence of the sequence $(r_{k_{j}}(\omega))_{j\in\N}$, we deduce weak convergence of the sequence 
$(Z_{k_{j}}(\omega)-r_{k_{j}}(\omega),Z_{k_{j}}(\omega))_{j\in\N}\wlim (\chi(\omega),\chi(\omega))$. Therefore $\frac{1}{\lambda}r_{k_{j}}(\omega)-V(Z_{k_{j}}(\omega))+V\left(Z_{k_{j}}(\omega)-r_{k_{j}}(\omega)\right)\to 0$. Hence, $0\in (T+V)(\chi(\omega))$, showing that $\chi(\omega)\in\setS$. Invoking \cite[Prop 2.3(iv)]{ComPes15}, we conclude that $(X_{k})_{k\in\N}$ converges weakly $\Pr$-a.s to an $\setS$-valued random variable. 

\noindent
We now establish (iii). Let $q_{k}\eqdef\Ex[Q_{k}(p)]$, so that \eqref{eq:Q2} yields the recursion
\[
q_{k}\leq q_{k-1}-\Ex[\theta_{k}]+\psi_{k}.
\]
By Assumption \ref{ass:batch}, and the definition of all sequences involved, we see that $\sum_{k=1}^{\infty}\psi_{k}<\infty$. Hence, a telescopian argument gives 
\[
q_{k}-q_{0}=\sum_{i=1}^{k}(q_{i}-q_{i-1})\leq -\sum_{i=1}^{k}\Ex[\theta_{i}]+\sum_{i=1}^{k}\psi_{i}\leq -\sum_{i=1}^{k}\Ex[\theta_{i}]+\sum_{i=1}^{\infty}\psi_{i}.
\]
Hence, for all $k\geq 1$, rearranging the above reveals 
\[
\sum_{i=1}^{k}\Ex[\theta_{i}]\leq q_{0}+\sum_{i=1}^{\infty}\psi_{i}<\infty.
\]
Letting $k\to\infty$, we conclude $\left(\Ex[\theta_{k}]\right)_{k\in\N}\in\ell^{1}_{+}(\N)$. Classically, this implies $\theta_{k}\to 0$ $\Pr$-a.s. By a simple majorization argument, we deduce that $\Pr$-a.s.
\begin{align*}
\infty&>\sum_{k=1}^{\infty}\left\{\frac{\rho_{k}}{8}\residual^{2}_{\lambda_{k}}(Z_{k})-\left[2\alpha^{2}_{k}+(1-\alpha_{k})\left(1-\frac{5(1-\alpha_{k})}{4\rho_{k}(1+L\lambda_{k})}\right)\right]\Delta_{k}\right\}\\
&\geq \sum_{k=1}^{\infty}\left[(1-\alpha_{k})\left(\frac{5(1-\alpha_{k})}{4\rho_{k}(1+L\lambda_{k})}-1\right)-2\alpha^{2}_{k}\right]\Delta_{k}.
\end{align*}
\qed
\end{proof}

\begin{remark}\label{rem:1}
The above result gives some indication of the balance between the inertial effect and the relaxation effect. Our analysis revealed that the maximal value of the relaxation parameter is $\rho\leq \frac{5(1-\bar{\eps})(1-\alpha)^{2}}{4(1+L\lambda)(2\alpha^{2}-\alpha+1)}$. This is closely aligned with the maximal relaxation value exhibited in Remark 2.13 of \cite{AttCab20}. Specifically, the function $\rho_{m}(\alpha,\eps)=\frac{5(1-\eps)(1-\alpha)^{2}}{4(1+L\lambda)(2\alpha^{2}-\alpha+1)}$. This function is decreasing in $\alpha$.  For this choice of parameters, one observes that for $\alpha\to 0$ we get $\rho\to\frac{5(1-\eps)}{4(1+L\lambda)}$ and for $\alpha\to 1$ it is observed $\rho\to 0$.
\end{remark}

As an immediate corollary of Theorem \ref{th:convergence}, we obtain a convergence result when all parameter sequences are constant. 

\begin{corollary}[{\bf Asymptotic convergence under constant inertia and relaxation}]
Let the same Assumptions as in Theorem \ref{th:convergence} hold. Consider Algorithm \ac{RISFBF} with the constant parameter sequences $\alpha_{k}\equiv \alpha\in(0,1),\lambda_{k}\equiv\lambda\in(0,\tfrac{1}{4L})$ and $\rho_{k}=\rho<\frac{5(1-\alpha)^{2}}{4(1+L\lambda)(2\alpha^{2}+1-\alpha)}$. Then $(X_{k})_{k\in\N}$ converges weakly $\Pr$-a.s. to a limiting random variable with values in $\setS$.
\end{corollary}

In fact, the a.s. convergence with a larger $\lambda_k$ is allowed as shown in the following corollary.
\begin{corollary}[{\bf Asymptotic convergence under larger steplength}]
Let the same Assumptions as in Theorem \ref{th:convergence} hold. Consider Algorithm \ac{RISFBF} with the constant parameter sequences $\alpha_{k}\equiv \alpha\in(0,1),\lambda_{k}\equiv\lambda\in(0,\tfrac{1-\nu}{2L})$ and $\rho_{k}=\rho<\frac{(3-\nu)(1-\alpha)^{2}}{2(1+L\lambda)(2\alpha^{2}+1-\alpha)}$, where $0<\nu<1$. Then $(X_{k})_{k\in\N}$ converges weakly $\Pr$-a.s. to a limiting random variable with values in $\setS$.
\end{corollary}
\begin{proof}
First we make a slight modification to \eqref{l-lambda} that the following relation holds for $0<\nu<1$
\begin{align}
\notag& \quad (1-\rho_{k})\norm{Z_{k}-p}^{2}+\rho_{k}\norm{R_{k}-p}^{2}-\frac{1-\rho_{k}}{\rho_{k}}\norm{X_{k+1}-Z_{k}}^{2}\\
\notag&\leq\norm{Z_{k}-p}^{2}-\frac{1-\rho_{k}}{\rho_{k}}\norm{X_{k+1}-Z_{k}}^{2}-\rho_{k}((1-\nu)-2L^{2}\lambda^{2}_{k})\norm{Z_{k}-Y_{k}}^{2}\\
\notag&+2\lambda^{2}\rho_{k}\norm{\ce_{k}}^{2}-2\rho_{k}\lambda_{k}\inner{W_{k}+p^{\ast},Y_{k}-p}-\rho_{k}\nu\norm{Y_{k}-Z_{k}}^{2}+2\rho_{k}\lambda_{k}\inner{V(Y_{k})-V(p),p-Y_{k}}.
\end{align}
Then similarly with \eqref{l-lambda-2}, we multiply both sides of \eqref{bd-eq2sc} by $\frac{\rho_{k}((1-\nu)-2L\lambda_{k})}{1+L\lambda_{k}}$, which is positive since $\lambda_k \in (0,\tfrac{1-\nu}{2L})$. The convergence follows in a similar fashion to Theorem~\ref{th:convergence}.
\end{proof}

Another corollary of Theorem \ref{th:convergence} is a strong convergence result, assuming that $F$ is demiregular (cf. Definition \ref{def:demiregular}).
\begin{corollary}[{\bf Strong Convergence under demiregularity}]
Let the same Assumptions as in Theorem \ref{th:convergence} hold. If $F=T+V$ is demiregular, then $(X_{k})_{k\in\N}$ converges strongly $\Pr$-a.s. to a $\setS$-valued random variable. 
\end{corollary}
\begin{proof}
Set $y_{k_{j}}(\omega)\eqdef Z_{k_{j}}(\omega)-r_{k_{j}}(\omega)$, and $u_{k_{j}}(\omega)\eqdef \frac{1}{\lambda}r_{k_{j}}(\omega)-V(Z_{k_{j}}(\omega))+V(Z_{k_{j}}(\omega)-r_{k_{j}}(\omega))$. We know from the proof of Theorem \ref{th:convergence} that $y_{k_{j}}(\omega)\wlim \chi(\omega)$ and $u_{k_{j}}(\omega)\to 0$. If $F=T+V$ is demiregular then $y_{k_{j}}(\omega)\to \chi(\omega)$. Since we know $r_{k_{j}}(\omega)\to 0$, we conclude $Z_{k_{j}}(\omega)\to \chi(\omega)$. Since $Z_{k}$ and $X_{k}$ have the same limit points, it follows $X_{k}\to \chi$. \qed
\end{proof}

\subsection{Linear Convergence}
\label{sec:Linear}
%

In this section, we derive a linear convergence rate and prove strong convergence of
the last iterate in the case where the single-valued operator $V$ is
\emph{strongly monotone}. Various linear convergence results in the context of
stochastic approximation algorithms for solving fixed-point problems are
reported in \cite{ComPes19} in the context of the random sweeping processes. In a
general structured monotone inclusion setting \cite{RosVilVuSBF16} derive rate
statements for cocoercive mean operators in the context of forward-backward
splitting methods. More recently, Cui and Shanbhag~\cite{CuiSha20} provide
linear and sublinear rates of convergence for a variance-reduced inexact proximal-point 
scheme for both strongly monotone and monotone inclusion problems. However, to the best of our knowledge, our results \us{are the
first published} for a stochastic operator splitting algorithm, featuring
\emph{relaxation and inertial} effects. Notably, this result does not require imposing Assumption~\ref{ass:SObound}(i) (i.e. the noise process be conditionally unbiased.) \ms{Instead our derivations hold true under a weaker notion of an asymptotically unbiased \ac{SO}.}

\begin{assumption}[\textbf{Asymptotically unbiased \ac{SO}}]
\label{ass:SObiased}
There exists a constant $\cs>0$ such that
\begin{align}\label{e:sigma2}
    \Ex[\norm{U_{k}}^{2}\vert\scrF_{k}]\leq \frac{\cs^{2}}{m_{k}}\text{ and }\Ex[\norm{W_{k}}^{2}\vert\scrF_{k}]\leq \frac{\cs^{2}}{m_{k}}, \qquad \Pr-\text{a.s.} 
\end{align}
for all $k\geq 1$.
\end{assumption}

This definition is rather mild and is \us{imposed} in many simulation-based optimization schemes in finite dimensions. \us{Amongst the more important} ones is the simultaneous perturbation stochastic approximation (SPSA) method pioneered by Spall~\cite{Spall92,Spa97}. \us{In this scheme, it is required that the gradient estimator satisfies an asymptotic unbiasedness requirement; in particular, the bias in the gradient estimator needs to diminish at a suitable rate to ensure asymptotic convergence}. In \us{fact}, this setting has been investigated in detail in the \us{context of stochastic Nash games}~\cite{DuvMerStaVer18}. Further examples for stochastic approximation schemes in a Hilbert-space setting obeying Assumption \ref{ass:SObiased} are \cite{BarRoyStr07,BarRoyStru09} and \cite{Geiersbach:2020tw}. \us{We now discuss an example that further clarifies the requirements on the estimator.} 

\begin{example}
Let $\{\hat{V}_{k}(x,\xi)\}_{k\in\N}$ be a collection of independent random $\setH$-valued vector fields of the form $\hat{V}_{k}(x,\xi)=V(x)+\eps_{k}(x,\xi)$ such that 
\begin{align*}
    \Ex_{\xi}[\eps_{k}(x,\xi)\vert x]=\us{\frac{B_{k}}{\sqrt{m_k}}}\text{ and }\Ex_{\xi}[\norm{\eps_{k}(x,\xi)}^{2}\vert x]\leq \hat{\sigma}^{2}\qquad \Pr-\text{a.s.},
\end{align*}
where $\hat{\sigma}>0$ and \us{$\tilde{b} > 0$ such that $(B_{k})_{k\in\N}$ is an $\setH$-valued sequence satisfying $\|B_k\|^2 \leq \hat{b}^2$ in an a.s. sense}. These statistics \us{can be obtained as } 
\begin{align*}
\Ex[\norm{U_{k}}^{2}\vert\scrF_{k}]&=\Ex\left[\left\| \frac{1}{m_{k}}\sum_{t=1}^{m_{k}}\eps_{t}(Z_{k})\right\|^{2}\vert\scrF_{k}\right]\\
&=\frac{1}{m_{k}^{2}}\sum_{t=1}^{m_{k}}\Ex[\norm{\eps_{t}(Z_{k})}^{2}\vert\scrF_{k}]+\frac{2}{m_{k}^{2}}\sum_{t=1}^{m_{k}}\sum_{l>t}\Ex[\inner{\eps_{t}(Z_{k}),\eps_{l}(Z_{k})}\vert\scrF_{k}]\\
&\leq \frac{\hat{\sigma}^{2}}{m_{k}}+\us{\frac{(m_{k}-1)}{m_{k}}\frac{\norm{B_{k}}^{2}}{m_k}} \leq \us{\frac{\hat{\sigma}^2 +  \hat{b}^2}{m_k}} \qquad \us{\Pr-\text{a.s.} }
\end{align*}
\ms{Setting} $\cs^2 \eqdef \hat{\sigma}^{2}+\us{\hat{b}^2}$, we see that condition \eqref{e:sigma2} holds. A similar estimate holds for the random noise $\norm{W_{k}}^{2}$.
\end{example}

\begin{assumption}\label{ass:strongV}
$V:\setH\to\setH$ is $\mu$-strongly monotone ($\mu>0$), i.e. 
\begin{equation}
\inner{V(x)-V(y),x-y}\geq \mu\norm{x-y}^{2}\qquad\forall x,y\in\dom V=\setH.
\end{equation}
\end{assumption}
Combined with Assumption \ref{ass:exists}, strong monotonicity implies that $\setS=\{\bar{x}\}$ for some $\bar{x}\in\setH$.
\begin{remark} 
In the context of a structured operator $F=T+V$, the assumption that the single-valued part $V$ is strongly monotone can be done without loss of generality. Indeed, if instead $T$ is assumed to be $\mu$-strongly monotone, then $(V+\mu\Id)+(T-\mu\Id)$ is maximally monotone and Lipschitz continuous while $\tilde V \triangleq  V+\mu \Id$ may be seen to be $\mu$-strongly monotone operator. 
\end{remark}
Our first result establishes a ``perturbed linear convergence'' rate on the anchor sequence $\left(\norm{X_{k}-\bar{x}}^{2}\right)_{k\in\N}$, similar to the one derived in \cite[Corollary 3.2]{ComPes19} in the context of randomized fixed point iterations. 

\begin{theorem}[{\bf Perturbed linear convergence}] \label{th:linear1}
Consider \ac{RISFBF} with $X_{0}=X_{1}$. Suppose Assumptions~\ref{ass:exists}-\ref{ass:fdd}, \ms{Assumption \ref{ass:SObiased}} 
and Assumption \ref{ass:strongV} hold. Let $\setS=\{\bar{x}\}$ denotes the unique solution of \eqref{eq:MI}. Suppose $\lambda_k\equiv \lambda\le\min\left\{\tfrac{a}{2\mu},b\mu,\tfrac{1-a}{2\tilde{L}}\right\}$, where $0<a,b<1$, $\tilde{L}^2\triangleq L^2+\tfrac{1}{2}$, $\eta_k\equiv\eta\triangleq(1-b)\lambda\mu$. Define $\Delta M_{k}\eqdef 2\rho_{k}\norm{W_{k}}^{2}+\frac{(3-a)\rho_{k}\lambda_{k}^{2}}{1+\tilde{L}\lambda_{k}}\norm{\ce_{k}}^{2}$. Let $(\alpha_k)_{k\in\N}$ be a non-decreasing sequence such that $0<\alpha_k\le\bar{\alpha}<1$, and define $\rho_k\triangleq\tfrac{(3-a)(1-\alpha_k)^2}{2(2\alpha_k^2-0.5\alpha_k+1)(1+\tilde{L}\lambda)}$ for every $k \in \N$. Set
\begin{align}
&H_{k}\eqdef \norm{X_{k}-\bar{x}}^2+(1-\alpha_k)\left(\tfrac{3-a}{2\rho_k(1+\tilde{L}\lambda)}-1\right)\norm{X_{k}-X_{k-1}}^2-\alpha_k\norm{X_{k-1}-\bar{x}}^2,
\label{eq:Hk}\\
&c_{k}\eqdef \Ex[\Delta M_{k}\vert\scrF_{k}],\text{ and }\bar{c}_{k}\eqdef \sum_{i=1}^{k}q^{k-i}\Ex[c_{i}\vert\scrF_{1}],\nonumber
\end{align}
where $q=1-\rho\eta\in(0,1)$, $\rho=\frac{16(3-a)(1-\bar{\alpha})^{2}}{31(1+\tilde{L}\lambda)}$. Then the following hold: 
\begin{itemize}
\item[(i)]  $(\bar{c}_{k})_{k\in\N}\in\ell^{1}_{+}(\N)$.
\item[(ii)] For all $k\geq 1$ 
\begin{equation}\label{eq:LinearH}
\Ex[H_{k+1}\vert\scrF_{1}]\leq q^{k}H_{1}+\bar{c}_{k}.
\end{equation}
In particular, this implies a perturbed linear rate of the sequence $(\|X_{k}-\bar{x}\|^2)_{k\in\N}$ as 
\begin{equation}\label{eq:LinearX}
\Ex[\norm{X_{k+1}-\bar{x}}^{2}\vert\scrF_{0}]\leq q^{k}\left(\frac{2(1-\alpha_1)}{1-\bar{\alpha}}\norm{X_{1}-\bar{x}}^{2}\right)+\tfrac{2}{1-\bar{\alpha}}\bar{c}_{k}.
\end{equation}
\item[(iii)] $\sum_{k=1}^{\infty}(1-\alpha_k)\left(\tfrac{3-a}{2\rho_k(1+\tilde{L}\lambda)}-1\right)\norm{X_{k}-X_{k-1}}^2<\infty$ $\Pr\text{-a.s.}$.
\end{itemize}
\end{theorem}
\begin{proof}
Our point of departure for the analysis under the stronger Assumption \ref{ass:strongV} is eq. \eqref{eq:MT}, which becomes 
\begin{align*}
    \inner{Z_{k}-R_{k},Y_{k}-p}\geq \lambda_{k}\inner{W_{k}+p^{\ast},Y_{k}-p}+\lambda_{k}\mu\norm{Y_{k}-p}^{2}\quad\forall (p,p^{\ast})\in\gr(F). 
\end{align*}
Repeating the analysis of the previous section with reference point $p=\bar{x}$ and $p^{\ast}=0$, the unique solution of \eqref{eq:MI}, yields the bound 
\begin{align*}
\norm{R_{k}-\bar{x}}^{2} & \leq \norm{Z_{k}-\bar{x}}^{2}-(1-2L^{2}\lambda_{k}^{2})\norm{Y_{k}-Z_{k}}^{2}+2\lambda^{2}_{k}\norm{\ce_{k}}^{2}\\
&+2\lambda_{k}\inner{W_{k},\bar{x}-Y_{k}}-2\lambda_{k}\mu\norm{\bar{x}-Y_{k}}^{2}.
\end{align*}
The triangle inequality $\norm{Z_{k}-\bar{x}}^{2}\leq 2\norm{Y_{k}-Z_{k}}^{2}+2\norm{Y_{k}-\bar{x}}^{2}$ gives 
\begin{align*}
\norm{R_{k}-\bar{x}}^{2}& \leq \norm{Z_{k}-\bar{x}}^{2}-(1-2L^{2}\lambda_{k}^{2})\norm{Y_{k}-Z_{k}}^{2}+2\lambda^{2}_{k}\norm{\ce_{k}}^{2}+2\lambda_{k}\inner{W_{k},\bar{x}-Y_{k}}\\
&+2\lambda_{k}\mu\norm{Y_{k}-Z_{k}}^{2}-\lambda_{k}\mu\norm{Z_{k}-\bar{x}}^{2}.
\end{align*}
By Young's inequality, we have for all $c>0$,
\begin{align} 
    \notag
\inner{W_{k},\bar{x}-Y_{k}}&\leq \frac{1}{2c}\norm{W_{k}}^{2}+\frac{c}{2}\norm{Y_{k}-\bar{x}}^{2}\\
\label{bias-weak} &\leq \frac{1}{2c}\norm{W_{k}}^{2}+c\left(\norm{Z_{k}-\bar{x}}^{2}+\norm{Z_{k}-Y_{k}}^{2}\right).
\end{align}
Observe that this estimate is crucial in weakening the requirement of conditional unbiasedness. Choose $c=\frac{\lambda_{k}}{2}$ to get 
\begin{align*}
\norm{R_{k}-\bar{x}}^{2}& \leq \norm{Z_{k}-\bar{x}}^{2}-(1-2L^{2}\lambda_{k}^{2})\norm{Y_{k}-Z_{k}}^{2}+2\lambda^{2}_{k}\norm{\ce_{k}}^{2}+2\norm{W_{k}}^{2}+\lambda^{2}_{k}\norm{\bar{x}-Z_{k}}^{2}\\
&+2\lambda_{k}\mu\norm{Y_{k}-Z_{k}}^{2}-\lambda_{k}\mu\norm{Z_{k}-\bar{x}}^{2}+\lambda_{k}^{2}\norm{Z_{k}-Y_{k}}^{2}\\
&=(1+\lambda_{k}^{2}-\lambda_{k}\mu)\norm{Z_{k}-\bar{x}}^{2}+2\lambda^{2}_{k}\norm{\ce_{k}}^{2}+2\norm{W_{k}}^{2}\\
&-(1-2L^{2}\lambda^{2}_{k}-2\lambda_{k}\mu-\lambda^{2}_{k})\norm{Y_{k}-Z_{k}}^{2}.
\end{align*}
Assume that $\lambda_{k}\mu\leq \frac{a}{2}<1$. Then, 
\[
1-2L^{2}\lambda^{2}_{k}-2\lambda_{k}\mu-\lambda_{k}^{2}\geq (1-a)-2L^{2}\lambda_{k}^{2}-\lambda_{k}^{2}=(1-a)-2\tilde{L}^{2}\lambda^{2}_{k},
\]
where $\tilde{L}^2\eqdef L^2+1/2$. Moreover, choosing $\lambda_{k}\leq b\mu$, we see 
\[
1+\lambda^{2}_{k}-\lambda_{k}\mu\leq 1-(1-b)\lambda_k\mu.
\]
Using these bounds, we readily deduce for $0<\lambda_{k}\leq\min\{\frac{a}{2\mu},b\mu\}$, that 
\begin{align}
\notag
\norm{R_{k}-\bar{x}}^{2}&\leq \left(1-(1-b)\lambda_k\mu\right)\norm{Z_{k}-\bar{x}}^{2}-\left((1-a)-2\tilde{L}^{2}\lambda^{2}_{k}\right)\norm{Y_{k}-Z_{k}}^{2} \\
&+2\lambda^{2}_{k}\norm{\ce_{k}}^{2}+2\norm{W_{k}}^{2}.
\label{eq:Rnew}
\end{align}
Proceeding as in the derivation of eq. \eqref{bd-eq2}, one sees first that
\[
\frac{1}{2\rho_{k}^{2}}\norm{X_{k+1}-Z_{k}}^{2}\leq (1+\tilde{L}\lambda_{k})^{2}\norm{Y_{k}-Z_{k}}^{2}+\lambda^{2}_{k}\norm{\ce_{k}}^{2},
\]
and therefore, 
\begin{align}\notag
-\rho_{k}((1-a)-2\tilde{L}^{2}\lambda^{2}_{k})\norm{Y_{k}-Z_{k}}^{2}& \leq -\frac{(1-a)-2\tilde{L}\lambda_{k}}{2\rho_{k}(1+\tilde{L}\lambda_{k})}\norm{X_{k+1}-Z_{k}}^{2}\\
	& +\frac{\rho_{k}\lambda^{2}_{k}((1-a)-2\tilde{L}\lambda_{k})}{1+\tilde{L}\lambda_{k}}\norm{\ce_{k}}^{2}.
\label{bd-eq3}
\end{align}
Define $\eta_{k}=(1-b)\lambda_k\mu$. Using the equality \eqref{bd-eq1},
\begin{align*}
\norm{X_{k+1}&-\bar{x}}^{2}=(1-\rho_{k})\norm{Z_{k}-\bar{x}}^{2}+\rho_{k}\norm{R_{k}-\bar{x}}^{2}-\tfrac{1-\rho_{k}}{\rho_{k}}\norm{X_{k+1}-Z_{k}}^{2}\\
&\overset{\eqref{eq:Rnew}}{\leq} (1-\rho_k\eta_k)\norm{Z_{k}-\bar{x}}^{2}-\tfrac{1-\rho_{k}}{\rho_{k}}\norm{X_{k+1}-Z_{k}}^{2}-\rho_{k}((1-a)-2\tilde{L}^{2}\lambda^{2}_{k})\norm{Z_{k}-Y_{k}}^{2}\\
&+2\lambda^{2}_k\rho_{k}\norm{\ce_{k}}^{2}+2\rho_k\norm{W_{k}}^2\\
&\overset{\eqref{bd-eq3}}{\leq} (1-\rho_k\eta_k)\norm{Z_{k}-\bar{x}}^{2}-\tfrac{(3-a)-2\rho_{k}(1+\tilde{L}\lambda_{k})}{2\rho_{k}(1+\tilde{L}\lambda_{k})}\norm{X_{k+1}-Z_{k}}^{2}+2\rho_{k}\norm{W_{k}}^{2}+\tfrac{(3-a)\rho_{k}\lambda_{k}^{2}}{1+\tilde{L}\lambda_{k}} \norm{\ce_{k}}^{2} \\
&\overset{\eqref{eq:Z1},\eqref{eq:Xk1}}{\leq} (1-\rho_k\eta_k)[(1+\alpha_{k})\norm{X_{k}-\bar{x}}^{2}-\alpha_{k}\norm{X_{k-1}-\bar{x}}^{2}+\alpha_{k}(1+\alpha_{k})\norm{X_{k}-X_{k-1}}^{2}] \\
&-\tfrac{(3-a)-2\rho_{k}(1+\tilde{L}\lambda_{k})}{2\rho_{k}(1+\tilde{L}\lambda_{k})}[(1-\alpha_{k})\norm{X_{k+1}-X_{k}}^{2}+(\alpha^{2}_{k}-\alpha_{k})\norm{X_{k}-X_{k-1}}^{2}] \\
&+2\rho_{k}\norm{W_{k}}^{2}+\tfrac{(3-a)\rho_{k}\lambda_{k}^{2}}{(1+\tilde{L}\lambda_{k})} \norm{\ce_{k}}^{2} \\
&\leq(1+\alpha_k)(1-\rho_k\eta_k)\|X_k-\bar{x}\|^2-\alpha_k(1-\rho_k\eta_k)\|X_{k-1}-\bar{x}\|^2+\Delta M_{k} \\
&+\alpha_k\norm{X_k-X_{k-1}}^2\left[(1+\alpha_k)(1-\rho_k\eta_k)+(\alpha_k-1)+\tfrac{(3-a)(1-\alpha_k)}{2\rho_k(1+\tilde{L}\lambda_k)}\right] \\
&-(1-\alpha_k)\left(\tfrac{3-a}{2\rho_k(1+\tilde{L}\lambda_k)}-1\right)\norm{X_{k+1}-X_k}^2,  
\end{align*}
with stochastic error term $\Delta M_{k}\eqdef 2\rho_{k}\norm{W_{k}}^{2}+\frac{(3-a)\rho_{k}\lambda_{k}^{2}}{1+\tilde{L}\lambda_{k}}\norm{\ce_{k}}^{2}$. From here, it follows that
\begin{align}
\notag &\norm{X_{k+1}-\bar{x}}^2+(1-\alpha_k)\left(\tfrac{3-a}{2\rho_k(1+\tilde{L}\lambda_k)}-1\right)\norm{X_{k+1}-X_k}^2-\alpha_k\|X_k-\bar{x}\|^2 \\
\notag &\le (1-\rho_k\eta_k)\left[\norm{X_{k}-\bar{x}}^2+(1-\alpha_k)\left(\tfrac{3-a}{2\rho_k(1+\tilde{L}\lambda_k)}-1\right)\norm{X_{k}-X_{k-1}}^2-\alpha_k\|X_{k-1}-\bar{x}\|^2\right] \\
\notag &-\left[(1-\rho_k\eta_k)(1-\alpha_k)\left(\tfrac{3-a}{2\rho_k(1+\tilde{L}\lambda_k)}-1\right) \right. \\
\notag &\left.-\alpha_k\left((1+\alpha_k)(1-\rho_k\eta_k)+(\alpha_k-1)+\tfrac{(3-a)(1-\alpha_k)}{2\rho_k(1+\tilde{L}\lambda_k)}\right) \right] \norm{X_{k}-X_{k-1}}^2 \\
\notag &-\alpha_k\rho_k\eta_k\norm{X_{k}-\bar{x}}^2+\Delta M_{k} \\
\notag &= (1-\rho_k\eta_k)\left[\norm{X_{k}-\bar{x}}^2+(1-\alpha_k)\left(\tfrac{3-a}{2\rho_k(1+\tilde{L}\lambda_k)}-1\right)\norm{X_{k}-X_{k-1}}^2-\alpha_k\|X_{k-1}-\bar{x}\|^2\right] \\
&-\underbrace{\left[(1-\alpha_k-\rho_k\eta_k)\left(\tfrac{(3-a)(1-\alpha_k)}{2\rho_k(1+\tilde{L}\lambda_k)}-1 \right)-\alpha_k^2(2-\rho_k\eta_k) \right]}_{\eqdef \tilde I}\norm{X_{k}-X_{k-1}}^2-\alpha_k\rho_k\eta_k\norm{X_{k}-\bar{x}}^2+\Delta M_{k}. \label{rec-Hk}
\end{align}
Since $\lambda_k = \lambda$, and $\rho_k=\tfrac{(3-a)(1-\alpha_k)^2}{2(2\alpha_k^2-0.5\alpha_k+1)(1+\tilde{L}\lambda)}$, we claim that $\rho_k \le \tfrac{1-\alpha_k}{(1+4\alpha_k)\eta}$ for $\eta\equiv(1-b)\lambda\mu$. Indeed,\footnote{To wit, the function $x\mapsto 2x^{2}-0.5x+1$ is attains a global minumum at $x=1/8$, which gives the global lower bound $31/32$. Furthermore, the function $x\mapsto (1-x)(1+4x)$ attains a global maximum at $x=3/8$, with corresponding value $25/16$.}
\begin{align*}
\tfrac{\tfrac{1-\alpha_k}{(1+4\alpha_k)\eta}}{\rho_k}=\tfrac{2(2\alpha_k^2-0.5\alpha_k+1)(1+\tilde{L}\lambda)}{(3-a)(1-\alpha_k)(1+4\alpha_k)\eta} \ge \tfrac{2(2\alpha_k^2-0.5\alpha_k+1)(1+\tilde{L}\lambda)}{(3-a)(1-\alpha_k)(1+4\alpha_k)\tfrac{a}{2}(1-b)} \ge \tfrac{2\cdot \tfrac{31}{32}\cdot 1}{\tfrac{25}{16}\cdot 1} =\frac{31}{25}> 1.
\end{align*} 
In particular, this implies $\eta\rho_{k}\in(0,1)$ for all $k\in\N$. We then have
\begin{align}
    \notag
   \tilde I&=(1-\alpha_k-\rho_k\eta)\left(\tfrac{(3-a)(1-\alpha_k)}{2\rho_k(1+\tilde{L}\lambda)}-1 \right)-\alpha_k^2(2-\rho_k\eta) \ge\left(1-\alpha_k-\tfrac{1-\alpha_k}{1+4\alpha_k}\right)\left(\tfrac{2\alpha_k^2-0.5\alpha_k+1}{1-\alpha_k}-1 \right)-2\alpha_k^2 \\
\notag&=\tfrac{(1-\alpha_k)4\alpha_k}{1+4\alpha_k}\cdot\tfrac{2\alpha_k^2+0.5\alpha_k}{1-\alpha_k}-\tfrac{2\alpha_k^2(1+4\alpha_k)}{1+4\alpha_k}\\
&=0. \label{def-I}
\end{align}
Next, we show that $H_k\ge \tfrac{1-\bar{\alpha}}{2}\|X_k-\bar{x}\|^2$, for $H_{k}$ defined in \eqref{eq:Hk}. This can be seen from the next string of inequalities:
\begin{align*}
H_{k}&=\norm{X_{k}-\bar{x}}^{2}-\alpha_k\norm{X_{k-1}-\bar{x}}^{2}+(1-\alpha_k)\left(\frac{3-a}{2\rho_k(1+\tilde{L}\lambda)}-1\right)\|X_k-X_{k-1}\|^2\\
&\geq \norm{X_{k}-\bar{x}}^{2}+\left(\frac{(1-\alpha_k)(2\alpha_k^{2}+1-0.5\alpha_k)}{(1-\alpha_k)^{2}}-1+\alpha_k\right)\norm{X_{k}-X_{k-1}}^{2}-\alpha_k\norm{X_{k-1}-\bar{x}}^{2}\\
&\ge\norm{X_{k}-\bar{x}}^{2}+\left(\frac{(1-\alpha_k)(2\alpha_k^{2}+1-\alpha_k)}{(1-\alpha_k)^{2}}-1+\alpha_k\right)\norm{X_{k}-X_{k-1}}^{2}-\alpha_k\norm{X_{k-1}-\bar{x}}^{2}\\
&=\left(\alpha_k+\frac{1-\alpha_k}{2}\right)\norm{X_{k}-\bar{x}}^{2}+\left(\alpha_k+\frac{2\alpha_k^{2}}{1-\alpha_k}\right)\norm{X_{k}-X_{k-1}}^{2}-\alpha_k\norm{X_{k-1}-\bar{x}}^{2}+\frac{1-\alpha_k}{2}\norm{X_{k}-\bar{x}}^{2}\\
&\geq \alpha_k\left(\norm{X_{k}-\bar{x}}^{2}+\norm{X_{k}-X_{k-1}}^{2}\right)-\alpha_k\norm{X_{k-1}-\bar{x}}^{2}+2\alpha_k\norm{X_{k}-\bar{x}}\cdot\norm{X_{k}-X_{k-1}}+\frac{1-\alpha_k}{2}\norm{X_{k}-\bar{x}}^{2}\\
&\geq \alpha_k\left(\norm{X_{k}-\bar{x}}+\norm{X_{k}-X_{k-1}}\right)^{2}-\alpha_k\norm{X_{k-1}-\bar{x}}^{2}+\frac{1-\alpha_k}{2}\norm{X_{k}-\bar{x}}^{2} \geq \frac{1-\alpha_k}{2}\norm{X_{k}-\bar{x}}^{2}\\
& \ge \frac{1-\bar{\alpha}}{2}\norm{X_{k}-\bar{x}}^{2}.
\end{align*}
In this derivation we have used the Young inequality$\frac{1-\alpha_k}{2}\norm{X_{k}-\bar{x}}^{2}+\frac{2\alpha_k^{2}}{1-\alpha_k}\norm{X_{k}-X_{k-1}}^{2}\geq 2\alpha_k\norm{X_{k}-\bar{x}}\cdot\norm{X_{k}-X_{k-1}},$ and the specific choice $\rho_k = \frac{(3-a)(1-\alpha_k)^{2}}{2(2\alpha_k^{2}-\frac{1}{2}\alpha_k+1)(1+\tilde{L}\lambda)}$.

By recalling \eqref{rec-Hk} and invoking \eqref{def-I}, we are left with the stochastic recursion
\begin{equation}\label{eq:Hk2}
    H_{k+1}\leq q_k  H_{k}-\tilde b_{k}+\Delta M_{k}.
\end{equation}
where $q_k \eqdef 1-\rho_k \eta$ and $\tilde{b}_k \eqdef \alpha_k \rho_k \eta_k\|X_k-\bar{x}\|^2.$ Since $\rho_k = \frac{(3-a)(1-\alpha_k)^{2}}{2(2\alpha_k^{2}-\frac{1}{2}\alpha_k+1)(1+\tilde{L}\lambda)} \ge \rho =\frac{16(3-a)(1-\bar{\alpha})^{2}}{31(1+\tilde{L}\lambda)}$ for every $k$, we have that $q_k \le q =1-\eta \rho$ for every $k$. Furthermore, $1>\eta\rho_{k}\geq\eta\rho$, so that $q \in (0,1)$. Taking conditional expectations on both sides on \eqref{eq:Hk2}, we get 
\begin{align*}
    \Ex[H_{k+1}\vert \scrF_{k}]+\tilde b_{k}\leq q H_{k}+c_{k}\quad\Pr\text{-a.s.}
\end{align*}
using the notation $c_{k}\eqdef\Ex[\Delta M_{k}\vert\scrF_{k}]$. Applying the operator $\Ex[\cdot\vert\scrF_{k-1}]$ and using the tower property of conditional expectations, this gives 
\[
\Ex[H_{k+1}\vert\scrF_{k-1}]\leq q^{2}H_{k-1}-q\Ex[\tilde{b}_{k-1}\vert\scrF_{k-1}]-\Ex[\tilde{b}_{k}\vert\scrF_{k-1}]+q\Ex[c_{k-1}\vert\scrF_{k-1}]+\Ex[c_{k}\vert\scrF_{k-1}].
\]
Proceeding inductively, we see that
\[
\Ex[H_{k+1}\vert\scrF_{1}]\leq q^{k}H_{1}+\sum_{i=1}^{k-1}q^{k-i}\Ex[c_{i}\vert\scrF_{1}]=q^{k}H_{1}+\bar{c}_{k}.
\]
This establishes eq. \eqref{eq:LinearH}. To validate eq. \eqref{eq:LinearX}, recall that we assume $X_{1}=X_{0}$, so that $H_{1}=(1-\alpha_{1})\norm{X_{1}-\bar{x}}^{2}$. Furthermore, $H_{k+1}\geq\frac{1-\bar{\alpha}}{2}\norm{X_{k+1}-\bar{x}}^{2}$, so that 
\[
\Ex[\norm{X_{k+1}-\bar{x}}^{2}\vert\scrF_{1}]\leq q^{k}\left(\frac{2(1-\alpha_{1})}{1-\bar{\alpha}}\norm{X_{1}-\bar{x}}^{2}\right)+\frac{2}{1-\bar{\alpha}}\bar{c}_{k}.
\]
We now show that $(\bar{c}_{k})_{k\in\N}\in\ell^{1}_{+}(\N)$. Simple algebra, \ms{combined with Assumption \ref{ass:SObiased}}, gives 
\begin{align}
 c_{k}&=\Ex[\Delta M_{k}\vert\scrF_{k}]\leq 2\rho_{k}\left(1+\frac{(3-a)\lambda^{2}}{1+\tilde{L}\lambda}\right)\Ex[\norm{W_{k}}^{2}\vert\scrF_{k}]+\frac{2(3-a)\rho_k\lambda^{2}}{1+\tilde{L}\lambda}\Ex[\norm{U_{k}}^{2}\vert\scrF_{1}]\nonumber\\
&\leq \frac{2\cs^{2}\rho_{k}}{m_{k}}\left(1+\frac{2(3-a)\lambda^{2}}{1+\tilde{L}\lambda}\right)\equiv\frac{\rho_{k}\cs^{2}}{m_{k}}\kappa.
\label{eq:boundc}
\end{align}
Hence, since $(\rho_{k})_{k\in\N}$ is bounded, Assumption \ref{ass:batch} gives $\lim_{k\to\infty}c_{k}=0$ a.s. Using again the tower property, we see $\Ex[c_{k}\vert\scrF_{1}]=\Ex\left[\Ex(c_{k}\vert\scrF_{k})\vert\scrF_{1}\right]\leq \kappa\frac{\rho_{k}\cs^{2}}{m_{k}} \le \kappa\frac{\bar{\rho}\cs^{2}}{m_{k}}$, where $\rho_k = \frac{(3-a)(1-\alpha_k)^{2}}{2(2\alpha_k^{2}-\frac{1}{2}\alpha_k+1)(1+\tilde{L}\lambda)} \le \bar{\rho} =\frac{3-a}{2(1+\tilde{L}\lambda)} $ for every $k$. Consequently, the discrete convolution $\left(\sum_{i=1}^{k-1}q^{k-i}\Ex[c_{i}\vert\scrF_{1}]\right)_{k\in\N}$ is summable. Therefore $\sum_{k\geq 1}\Ex[H_{k}]<\infty$ and $\sum_{k\geq 1}\Ex[\tilde{b}_{k}]<\infty$. Clearly, this implies $\lim_{k\to\infty}\Ex[\tilde b_{k}]=0,$ and consequently the subsequently stated two implication follow as well: 
\begin{align*}
&\lim_{k\to\infty}\norm{X_{k}-\bar{x}}=0\quad\Pr\text{-a.s.},\quad\text{ and }\\
&\sum_{k=1}^{\infty}(1-\alpha_k)\left(\tfrac{3-a}{2\rho_k(1+\tilde{L}\lambda)}-1\right)\norm{X_{k}-X_{k-1}}^2<\infty\quad\Pr\text{-a.s.}.
\end{align*}
\qed
\end{proof}
\begin{remark}
    It is worth remarking that the above proof does not rely on unbiasedness of the random estimators. The reason why we can lift this rather typical assumption lies in our application Young's inequality in the estimate \eqref{bias-weak}. The only assumption needed is a summable oracle variance as formulated in Assumption \ref{ass:SObiased} to get the above result working. 
\end{remark}
\begin{remark}
    The above result illustrates again nicely the well-known trade-off between relaxation and inertial effects (cf. Remark~\ref{rem:1}). 
    Indeed, up to constant factors, the coupling between inertia and relaxation is expressed by the function $\alpha\mapsto \frac{(1-\alpha)^{2}}{2\alpha^{2}-\frac{1}{2}\alpha+1}$. Basic calculus reveals that this function is decreasing for $\alpha$ increasing. In the extreme case when $\alpha\uparrow 1$, it is necessary to let $\rho\downarrow 0$, and vice versa. When $\alpha\to 0$ then the limiting value of our specific relaxation policy is $\frac{3-a}{1+\tilde{L}\lambda}$. In practical applications, it is advisable to choose $b$ small in order to make $q$ large. The value $a$ must be calibrated in a disciplined way in order to allow for a sufficiently large step size $\lambda$. This requires some knowledge of the condition number of the problem $\mu/L$. As a heuristic argument, a good strategy, anticipating that $b$ should be close to $0$, is to set $\frac{a}{2\mu}=\frac{1-a}{2\tilde{L}}$. This means $a=\frac{\mu}{\tilde{L}+\mu}$. 
\end{remark}
We obtain a full linear rate of convergence when a more aggressive sample rate is employed in the \ac{SO}. We achieve such global linear rates, together with tuneable iteration and oracle complexity estimates in two settings: First, we consider an aggressive simulation strategy, where the sample size grows over time geometrically. Such a sampling frequency can be quite demanding in some applications. As an alternative, we then move on and consider a more modest simulation strategy under which only polynomial growth of the batch size is required. Whatever simulation strategy is adopted, key to the assessment of the iteration and oracle complexity is to bound the stopping time 
\begin{equation}\label{eq:K}
K_{\epsilon} \eqdef \inf\{ k \in\N\vert\; \Ex\left(\norm{X_{k+1} - \bar{x}}^2\right) \leq \epsilon \}.
\end{equation}
In order to understand the definition of this stopping time, recall that \ac{RISFBF} computes the last iterate $X_{K+1}$ by extrapolating between the current base point $Z_{k}$ and the correction step involving $Y_{k}+\lambda_{K}(A_{k}-B_{k})$, which requires $2 m_{k}$ iid realizations from the law $\measP$. In total, when executing the algorithm until the terminal time $K_{\epsilon}$, where therefore need to simulate $2\sum_{k=1}^{K_{\epsilon}}m_{k}$ random variables. We now estimate the integer $K_{\epsilon}$ under a geometric sampling strategy.

\begin{proposition} [Non-asymptotic linear convergence under geometric sampling]
\label{prop:rate_geom}
Suppose the conditions of Theorem \ref{th:linear1} hold. Let $p\in(0,1),\cB= 2\bar{\rho}\cs^{2}\left(1+\frac{2(3-a)\lambda^{2}}{1+\tilde{L}\lambda}\right),$ and choose the sampling rate $m_{k}=\lfloor p^{-k}\rfloor$. Let $\hat{p}\in(p,1)$, and define 
\begin{align}\label{eq:cpq}
    C(p,q)&\eqdef \ic{\tfrac{2(1-\alpha_1)}{1-\bar{\alpha}}}\norm{X_{1}-\bar{x}}^{2}+\tfrac{4\cB}{\uss{(1-\bar{\alpha})}(1-\min\{p/q,q/p\})}\quad \text{if }p\neq q,\text{ and }\\
    \hat{C}&\eqdef \ic{\tfrac{2(1-\alpha_1)}{1-\bar{\alpha}}}\norm{X_{1}-\bar{x}}^{2}+\tfrac{4\cB}{\uss{(1-\bar{\alpha})}\exp(1)\ln(\hat{p}/q)}\quad \text{if }p=q.
\label{eq:hatc}
\end{align}
Then, whenever $p\neq q$, we see that 
\[
\Ex\left(\norm{X_{k+1}-\bar{x}}^{2}\right)\leq C(p,q)\max\{p,q\}^{k},
\]
and whenever $p=q$, 
\[
\Ex\left(\norm{X_{k+1}-\bar{x}}^{2}\right)\leq \hat{C}\hat{p}^{k}.
\]
In particular, the stochastic process $(X_{k})_{k\in\N}$ converges strongly and $\Pr$-a.s. to the unique solution $\bar{x}$ at a linear rate.
\end{proposition}
\begin{proof}
Departing from \eqref{eq:Hk2}, ignoring the positive term $\tilde{b}_{k}$ from the right-hand side, and taking expectations on both sides leads to
\begin{align}\label{cond-exp}
  \frac{1-\ic{\bar{\alpha}}}{2}\Ex(\norm{X_{k+1}-\bar{x}}^{2})\leq h_{k+1}\eqdef \Ex(H_{k+1})\leq q\Ex(H_{k})+c_{k}=q h_{k}+c_{k},
\end{align}
where the equality follows from $c_k$ being deterministic. \ms{The sequence $(c_{k})_{k\in\N}$ is further upper bounded by the following considerations: First, the relaxation sequence is bounded by $\rho_{k}\leq\bar{\rho}=\frac{3-a}{2(1+\tilde{L}\lambda)}$; Second, the sample rate is bounded by $m_{k}= \lfloor p^{-k} \rfloor \ge \left\lceil \tfrac{1}{2}p^{-k} \right\rceil \ge \tfrac{1}{2}p^{-k}$. Using these facts, eq. \eqref{eq:boundc} yields
\begin{equation}\label{eq:boundc2}
c_{k}\leq \frac{\rho_{k}\cs^{2}\kappa}{m_{k}}\leq 2\cB p^{k}\qquad\forall k\geq 1,
\end{equation}
where $\cB=2\bar{\rho}\cs^{2}\left(1+\frac{2(3-a)\lambda^{2}}{1+\tilde{L}\lambda}\right)$.} Iterating the recursion above, one readily sees that 
\begin{equation}\label{eq:h}
h_{k+1}\leq q^{k}h_{1}+\sum_{i=1}^{k}q^{k-i}c_{i}\quad \forall k\geq 1.
\end{equation}
Consequently, by recalling that $h_1 = (1-\alpha_1)\|X_1-\bar{x}\|^2$ and $h_{k}\geq\frac{1-\bar{\alpha}}{2}\Ex(\norm{X_{k}-\bar{x}}^{2})$, the bound \eqref{eq:boundc2} allows us to derive the recursion
\begin{equation}\label{eq:geometric}
\Ex\left(\norm{X_{k+1}-\bar{x}}^{2}\right)\leq q^{k}\left(\frac{2(1-\alpha_1)}{1-\bar{\alpha}}\norm{X_{1}-\bar{x}}^{2}\right)+\frac{4\cB}{1-\bar{\alpha}}\sum_{i=1}^{k}q^{k-i}p^{i}.
\end{equation}
%
We consider three cases.

\noindent (i) $0<q<p<1$: Defining $\const_{1} \eqdef \ic{\frac{2(1-\alpha_1)}{1-\bar{\alpha}}}\norm{X_{1}-\bar{x}}^{2}+\tfrac{4\cB}{(1-\ic{\bar{\alpha}})(1-q/p)}$, we obtain from \eqref{eq:geometric}
\[
   \Ex(\norm{X_{k+1}-\bar{x}}^{2})\leq q^{k}\left(\ic{\frac{2(1-\alpha_1)}{1-\bar{\alpha}}}\norm{X_{1}-\bar{x}}^{2}\right)+\frac{4\cB}{1-\ic{\bar{\alpha}}}\sum_{i=1}^{k}(q/p)^{k-i}p^{k}\leq \const_{1}p^{k}.
   \]

\noindent (ii) $0<p<q<1$. Akin to (i) and defining $\const_{2}\eqdef \ic{\frac{2(1-\alpha_1)}{1-\bar{\alpha}}}\norm{X_{1}-\bar{x}}^{2}+\tfrac{4\cB}{(1-\ic{\bar{\alpha}})(1-p/q)}$, we arrive as above at the bound $\Ex(\norm{X_{k}-\bar{x}}^{2})\leq q ^{k}\const_{2}$. 
\\
\noindent (iii) $p=q<1$. Choose $\hat{p} \in (q,1)$ and $\const_{3}\eqdef \tfrac{1}{\exp(1)\ln(\hat{p}/q)}$, so that Lemma \ref{lem:geometric} yields $kq^{k}\leq \const_{3}\hat{p}^{k}$ for all $k\geq 1$. Therefore, plugging this estimate in eq. \eqref{eq:geometric}, we see
\begin{align*}
\Ex(\norm{X_{k}-\bar{x}}^{2})&\leq q^{k}\left(\ic{\frac{2(1-\alpha_1)}{1-\bar{\alpha}}}\norm{X_{1}-\bar{x}}^{2}\right)+\frac{4\cB}{1-\ic{\bar{\alpha}}}\sum_{i=1}^{k}q^{k}\\
&\leq \hat{p}^{k}\left(\ic{\frac{2(1-\alpha_1)}{1-\bar{\alpha}}}\norm{X_{1}-\bar{x}}^{2}\right)+\frac{4\cB}{1-\ic{\bar{\alpha}}}\const_{3}\hat{p}^{k}\\
&=\const_{4}\hat{p}^{k},
\end{align*}
after setting $\const_{4}\eqdef \ic{\frac{2(1-\alpha_1)}{1-\bar{\alpha}}}\norm{X_{1}-\bar{x}}^{2}+\frac{4\cB\const_{3}}{1-\ic{\bar{\alpha}}}$. Collecting these three cases together, verifies the first part of the proposition. 
\qed
\end{proof}

\begin{proposition} [Oracle and Iteration Complexity under geometric sampling]
\label{prop:comgeom}
Given $\epsilon > 0$, define the stopping time $K_{\epsilon}$ as in eq. \eqref{eq:K}. Define 
\begin{equation}\label{eq:taueps}
\tau_{\epsilon}(p,q)\eqdef \left\{\begin{array}{ll} 
\lceil \frac{\ln(C(p,q)\epsilon^{-1})}{\ln(1/\max\{p,q\})}\rceil & \text{if }p\neq q,\\
\lceil\frac{\us{\ln(\hat{C}\epsilon^{-1})}}{\ln(1/\hat{p})}\rceil & \text{if }p=q
 \end{array}\right.
 \end{equation}
 and the same hypothesis as in Theorem \ref{th:linear1} apply. Then, $K_{\eps}\leq \tau_{\epsilon}(p,q)=\scrO(\us{\ln(\eps^{-1})})$. The corresponding oracle complexity of \ac{RISFBF} is upper bounded as $2\sum_{i=1}^{\tau_{\epsilon}(p,q)} m_i = \scrO\left((1/\epsilon)^{1+\delta(p,q)}\right)$, where 
\begin{align*}
\delta(p,q)\eqdef\left\{\begin{array}{ll}
0 & \text{if }p>q,\\
\frac{\ln(p)}{\ln(q)}-1 & \text{if }p\in(0,q),\\
\frac{\ln(p)}{\ln(\hat{p})}-1 & \text{if }p=q.
\end{array}\right.
\end{align*} 
\end{proposition}
\begin{proof} 
First, let us recall that the total oracle complexity of the method is assessed by 
\begin{align*}
2\sum_{i=1}^{K_{\epsilon}}m_{i}=2\sum_{i=1}^{K_{\epsilon}}\lfloor p^{-i}\rfloor\leq  2\sum_{i=1}^{K_{\epsilon}}p^{-i}.
\end{align*}
If $p\neq q$ define $\tau_{\epsilon}\equiv\tau_{\epsilon}(p,q)=\lceil \frac{\ln(C(p,q)\epsilon^{-1})}{\ln(1/\max\{p,q\})}\rceil$. Then, $\Ex(\norm{X_{\tau_{\epsilon}+1}-\bar{x}}^{2})\leq\epsilon$, and hence $K_{\epsilon}\leq \tau_{\epsilon}$. We now compute
\begin{align*}
\sum_{i=1}^{\tau_{\epsilon}}(1/p)^{i}&=\frac{1}{p}\frac{(1/p)^{\lceil \frac{\ln(C(p,q)\epsilon^{-1})}{\ln(1/\max\{p,q\})}\rceil}-1}{1/p-1}\leq\frac{1}{p^{2}} \frac{(1/p)^{\frac{\ln(C(p,q)\epsilon^{-1})}{\ln(1/\max\{p,q\})}}}{1/p-1}\\
&=\frac{\left(\epsilon^{-1}C(p,q)\right)^{\ln(1/p)/\ln(1/\max\{p,q\})}}{p(1-p)}.
\end{align*}
This gives the oracle complexity bound 
\[
2\sum_{i=1}^{\tau_{\epsilon}}m_{i}\leq 2\frac{\left(\epsilon^{-1}C(p,q)\right)^{\ln(1/p)/\ln(1/\max\{p,q\})}}{p(1-p)}.
\]
If $p=q$, we can replicate this calculation, after setting $\tau_{\epsilon}=\lceil\frac{\us{\ln(\epsilon^{-1}\hat{C})}}{\ln(1/\hat{p})}\rceil$. After so many iterations, we can be ensured that $\Ex(\norm{X_{\tau_{\epsilon}+1}-\bar{x}}^{2})\leq\epsilon$, with an oracle complexity
\[
2\sum_{i=1}^{\tau_{\epsilon}}m_{i}\leq \frac{2}{\hat{p}(1-\hat{p})}\left(\frac{\hat{C}}{\epsilon}\right)^{\ln(p)/\ln(\hat{p})}.
\]
\qed
\end{proof}

To the best of our knowledge, the provided non-asymptotic linear convergence guarantee appears to be amongst the first in relaxed and inertial splitting algorithms. In particular, by leveraging the increasing nature of mini-batches, this result no longer requires the unbiasedness assumption on the \ac{SO}, a crucial benefit of the proposed scheme. 

There may be settings where geometric growth of $m_k$ is challenging to adopt. To this end, we provide a result where the sampling rate is polynomial rather than geometric. A polynomial sampling rate arises if $m_{k}=\lceil{a_{k}(k+k_{0})^{\theta}+b_{k}\rceil}$ for some parameters $a_{k},b_{k},\theta>0$. Such a regime has been adopted in related mini-batch approaches~\cite{lei18game,lei2020asynchronous}.   
This allows for modulating the growth rate by changing the exponent in the sampling rate. We begin by providing a supporting result. We make the specific choice $a_{k}=b_{k}=1$ for all $k\geq 1$, and $k_{0}=0$, leaving essentially the exponent $\theta>0$ as a free parameter in the design of the stochastic oracle. 

\begin{proposition}[{\bf Polynomial rate of convergence under polynomially increasing $m_k$}]
\label{prop:poly_rate}
 Suppose the conditions of Theorem \ref{th:linear1} hold. Choose the sampling rate $m_{k}=\lfloor k^{\theta}\rfloor$ where $\theta > 0$. Then, for any $k\geq 1$,  
 \begin{equation}\label{eq:poly-rate}
     \Ex(\norm{X_{k+1}-\bar{x}}^2 )   \leq q^{k} \left(\ic{\frac{2(1-\alpha_1)}{1-\bar{\alpha}}}\norm{X_{1}-\bar{x}}^{2}+\uss{\frac{2}{1-\bar{\alpha}}}\frac{q^{-1}\exp(2\theta)-1}{1-q}\right)+\frac{4\cB}{\uss{(1-\bar{\alpha})}q\ln(1/q)}(k+1)^{-\theta}
\end{equation}
\end{proposition}
\begin{proof}
From the relation \eqref{eq:h}, we obtain 
\begin{align*}
h_{k+1}&\leq q^{k}h_{1}+\sum_{i=1}^{k}q^{k-i}c_{i}\leq q^{k}h_{1}+\cB \sum_{i=1}^{k}q^{k-i}i^{-\theta}\\
&= q^{k}\left(h_{1}+\cB\sum_{i=1}^{k}q^{-i}i^{-\theta}\right)\\
&= q^{k}\left(h_{1}+\cB\sum_{i=1}^{\lceil 2\theta/\ln(1/q)\rceil}q^{-i}i^{-\theta}+\cB \sum_{i=\lceil 2\theta/\ln(1/q)\rceil+1}^{k}q^{-i}i^{-\theta}\right).
\end{align*}

A standard bound based on the integral criterion for series with non-negative summands gives 
\[
\sum_{i=\lceil 2\theta/\ln(1/q)\rceil+1}^{k}q^{-i}i^{-\theta}\leq \int_{\lceil 2\theta/\ln(1/q)\rceil}^{k+1}\frac{(1/q)^{t}}{t^{\theta}}\dif t.
\]
The upper bounding integral can be evaluated using integration-by-parts, as follows: 
$$
\int_{\lceil 2\theta/\ln(1/q)\rceil}^{k+1}\frac{(1/q)^{t}}{t^{\theta}}\dif t=t^{\theta}\frac{e^{t\ln(1/q)}}{\ln(1/q)}\Big|_{t=\lceil 2\theta/\ln(1/q)\rceil}^{t=k+1}+\int_{\lceil 2\theta/\ln(1/q)\rceil}^{k+1}\theta t^{-(\theta+1)}\frac{e^{t\ln(1/q)}}{\ln(1/q)}\dif t.
$$
Note that $\frac{\theta}{t\ln(1/q)}\leq\frac{1}{2}$ when $t\geq \lceil 2\theta/\ln(1/q)\rceil$. Therefore, we can attain a simpler bound from the above by 
\begin{align*}
\int_{\lceil 2\theta/\ln(1/q)\rceil}^{k+1}\frac{(1/q)^{t}}{t^{\theta}}\dif t\leq \frac{(1/q)^{k+1}}{\ln(1/q)(k+1)^{\theta}}+\frac{1}{2}\int_{\lceil 2\theta/\ln(1/q)\rceil}^{k+1}\frac{(1/q)^{t}}{t^{\theta}}\dif t
\end{align*}
Consequently, 
$$
\int_{\lceil 2\theta/\ln(1/q)\rceil}^{k+1}\frac{(1/q)^{t}}{t^{\theta}}\dif t\leq \frac{2(1/q)^{k+1}(k+1)^{-\theta}}{\ln(1/q)}.
$$
Furthermore, 
$$
\sum_{i=1}^{\lceil 2\theta/\ln(1/q)\rceil}q^{-i}i^{-\theta}\leq \sum_{i=1}^{\lceil 2\theta/\ln(1/q)\rceil}q^{-i}=\frac{1}{q}\frac{(1/q)^{\lceil 2\theta/\ln(1/q)\rceil}-1}{1/q-1}\leq \frac{1}{q}\frac{(1/q)^{2\theta/\ln(1/q)+1}-1}{1/q-1}.
$$
Note that  $(1/q)^{2\theta/ \ln(1/q)}=\left(\exp(\ln(1/q))\right)^{ 2\theta/ \ln(1/q)}=\exp(2\theta)$. Hence, 
$$
\sum_{i=1}^{\lceil 2\theta/\ln(1/q)\rceil}q^{-i}i^{-\theta}\leq \frac{1}{q}\frac{q^{-1}\exp(2\theta)-1}{1/q-1}=\frac{q^{-1}\exp(2\theta)-1}{1-q}.
$$
Plugging this into the opening string of inequalities shows 
\begin{align*}
h_{k+1}&\leq q^{k}\left(h_{1}+\cB \sum_{i=1}^{\lceil 2\theta/\ln(1/q)\rceil}q^{-i}+\frac{2\cB(1/q)^{k+1}(k+1)^{-\theta}}{\ln(1/q)}\right)\\
&\leq q ^{k}\left(h_{1}+\cB \frac{q^{-1}\exp(2\theta)-1}{1-q}+\frac{2\cB(1/q)^{k+1}(k+1)^{-\theta}}{\ln(1/q)}\right)\\
&=q^{k}\left(h_{1}+\cB\frac{q^{-1}\exp(2\theta)-1}{1-q}\right)+\frac{2\cB/q}{\ln(1/q)}(k+1)^{-\theta}.
\end{align*}
Since $h_{1}=(1-\ic{\alpha_1})\norm{X_{1}-\bar{x}}^{2}$ and $h_{k+1}\geq\frac{1-\ic{\bar{\alpha}}}{2}\Ex\left(\norm{X_{k+1}-\bar{x}}^{2}\right)$, we finally arrive at the desired expression \eqref{eq:poly-rate}.
\qed\end{proof}

\begin{proposition}[Oracle and Iteration complexity under polynomial sampling]
Let all Assumptions as in Theorem \ref{th:linear1} hold. Given $\epsilon > 0$, define $K_{\epsilon}$ as in \eqref{eq:K}. Then
 the  iteration and oracle complexity to obtain an $\epsilon$-solution are $\scrO(\theta\eps^{-1/\theta})$ and $\scrO(\exp(\theta)\theta^{\theta}(1/\epsilon)^{1+1/\theta})$, respectively.
 \end{proposition}
\begin{proof}

We first note that $(k+1)^{-\theta}\leq k^{-\theta}$ for all $k\geq 1$. Hence, the bound established in Proposition \ref{prop:poly_rate} yields
\[
    \Ex(\norm{X_{k+1}-\bar{x}}^{2})\leq q^{k}\left(\frac{2(1-\alpha_1)}{1-\bar{\alpha}}\norm{X_{1}-\bar{x}}^{2}+\frac{2}{1-\bar{\alpha}}\frac{q^{-1}\exp(2\theta)-1}{1-q}\right)+\frac{4\cB}{(1-\bar{\alpha})q\ln(1/q)}k^{-\theta}
\]
Consider the function $\psi(t)\eqdef t^{\theta}q^{t}$ for $t>0$. Then, straightforward calculus shows that $\psi(t)$ is unimodal on $(0,\infty)$, with unique maximum $t^{\ast}=\frac{\theta}{\ln(1/q)}$ and associated value $\psi(t^{\ast})=\exp(-\theta)\left(\frac{\theta}{\ln(1/q)}\right)^{\theta}$. Hence, for all $t>0$, we have $t^{\theta}q^{t}\leq \exp(-\theta)\left(\frac{\theta}{\ln(1/q)}\right)^{\theta}$, and consequently, $q^{k}\leq \exp(-\theta)\left(\frac{\theta}{\ln(1/q)}\right)^{\theta}k^{-\theta}$ for all $k\geq 1$. This allows us to conclude 
\begin{align*}
\Ex(\norm{X_{k+1}-\bar{x}}^{2})&\leq \exp(-\theta)\left(\frac{\theta}{\ln(1/q)}\right)^{\theta}k^{-\theta}\left(\ic{\frac{2(1-\alpha_1)}{1-\bar{\alpha}}}\norm{X_{1}-\bar{x}}^{2}+\uss{\frac{2}{1-\bar{\alpha}}}\frac{q^{-1}\exp(2\theta)-1}{1-q}\right)\\
&+\frac{4\cB}{(1-\bar{\alpha})q\ln(1/q)}k^{-\theta}= \const_{q,\theta}k^{-\theta},
\end{align*}
where
\begin{equation}
    \const_{q,\theta}\eqdef \exp(-\theta)\left(\frac{\theta}{\ln(1/q)}\right)^{\theta}\left(\ic{\frac{2(1-\alpha_1)}{1-\bar{\alpha}}}\norm{X_{1}-\bar{x}}^{2}+\uss{\frac{2}{1-\bar{\alpha}}}\frac{q^{-1}\exp(2\theta)-1}{1-q}\right)+\frac{4\cB}{(1-\bar{\alpha})q\ln(1/q)}
\end{equation}
Then, for any $k\geq K_{\epsilon}\eqdef \lceil (\const_{q,\theta}/\epsilon)^{1/\theta}\rceil$, we are ensured that $\Ex(\norm{X_{k+1}-\bar{x}}^{2})\leq \eps$. Since $(\const_{q,\theta})^{1/\theta}=\scrO(\exp(-1)\theta)$, we conclude that $K_{\epsilon}=\scrO(\theta\epsilon^{-1/\theta})$. The corresponding oracle complexity is bounded as follows: 
$$
2\sum_{i=1}^{K_{\epsilon}}m_{i}\leq 2\sum_{i=1}^{K_{\epsilon}}i^{\theta}\leq2\int_{1}^{K_{\epsilon}+1}t^{\theta}\dif t\leq \frac{2}{1+\theta}\left(\lceil\frac{\const_{q,\theta}}{\epsilon}\rceil^{1/\theta}+1\right)^{1+\theta}=\scrO(\exp(\theta)\theta^{\theta}(1/\epsilon)^{1+1/\theta}).
$$
\qed
\end{proof}

\begin{remark} {It may be observed that if the $\theta= 1$ or $m_k = k$, there is a worsening of the rate and complexity statements from their counterparts when the sampling rate is geometric; in particular, the iteration complexity worsens from $\mathcal{O}(\ln(\tfrac{1}{\epsilon}))$ to $\mathcal{O}(\tfrac{1}{\epsilon})$ while the oracle complexity degenerates from the optimal level of $\mathcal{O}(\tfrac{1}{\epsilon})$ to $\mathcal{O}(\tfrac{1}{\epsilon^2})$. But this deterioration comes with the advantage that the sampling rate is far slower and this may be of signficant consequence in some applications.}
\end{remark}

\subsection{Rates in terms of merit functions}
\label{sec:rates}
%
In this subsection we estimate the iteration and oracle complexity of
\ac{RISFBF} with the help of a suitably defined \emph{gap function}. Generally,
a gap function associated with the monotone inclusion problem \eqref{eq:MI} is
a function $\gap:\setH\to\R$ such that  (i) $\gap$ is sign restricted on
$\setH$; and (ii) $\gap(x) = 0$ if and only if $x\in\setS$. The
\emph{Fitzpatrick function} \cite{Fit88,SimZal04,BorDut16,BauCom16} is a useful tool to
construct gap functions associated with a set-valued operator $F:\setH\to
2^{\setH}$. It is defined as the extended-valued function
$G_{F}:\setH\times\setH\to[-\infty,\infty]$ given by \begin{equation}
G_{F}(x,x^{\ast})=\inner{x,x^{\ast}}-\inf_{(y,y^{\ast})\in\gr(F)}\inner{x-y,x^{\ast}-y^{\ast}}.
\end{equation}
This function allows us to recover the operator $F$, by means of the following result (cf. \cite[Prop. 20.58]{BauCom16}): If $F:\setH\to 2^{\setH}$ is maximally monotone, then $G_{F}(x,x^{\ast})\geq\inner{x,x^{\ast}}$ for all $(x,x^{\ast})\in\setH\times\setH$, with equality if and only if $(x,x^{\ast})\in\gr(F)$. In particular, $\gr(F)=\{(x,x^{\ast})\in\setH\times\setH\vert \;G_{F}(x,x^{\ast})\geq\inner{x,x^{\ast}}\}$. In fact, it can be shown that the Fitzpatrick function is minimal in the family of convex functions $f:\setH\times\setH\to(-\infty,\infty]$ such that $f(x,x^{\ast})\geq\inner{x,x^{\ast}}$ for all $(x,x^{\ast})\in\setH\times\setH$, with equality if $(x,x^{\ast})\in\gr(F)$ \cite{BorDut16}.

Our gap function for the structured monotone operator $F=V+T$ is derived from its Fitzpatrick function by setting $\gap(x)\eqdef G_{F}(x,0)$ for $x\in\setH$. This reads explicitly as 
\begin{equation}\label{eq:gap}
\gap(x) \eqdef  \sup_{(y,y^{\ast})\in\gr(F)}\inner{y^{\ast},x-y}=\sup_{p\in\dom T}\sup_{p^{\ast}\in T(p)}\inner{V(p)+p^{\ast},x-p}\qquad \forall x\in\setH.
\end{equation}
It immediately follows from the definition that $\gap(x)\geq 0$ for all $x\in\setH$. It is also clear, that $x\mapsto\gap(x)$ is convex and lower semi-continuous and $\gap(x)=0$ if and only if $x\in\setS=\Zer(F)$. Let us give some concrete formulae for the gap function.

\begin{example}[Variational Inequalities]
We reconsider the problem described in Example \ref{ex:SVI}. Let $V:\setH\to\setH$ be a maximally monotone and $L$-Lipschitz continuous map, and $T(x)=\NC_{\setC}(x)$ the normal cone of a given closed convex set $\setC\subset\setH$. Then, by \cite[Prop. 3.3]{BorDut16}, the gap function \eqref{eq:gap} reduces to the well-known \emph{dual gap function}, due to \cite{Aus74}, 
\[
\gap(x)= \sup_{p\in\setC} \ \inner{V(p),x-p}.
\]
\end{example}

\begin{example}[Convex Optimization]
Reconsider the general non-smooth convex optimization problem in Example \ref{ex:PDA}, with primal objective function $\setH_{1}\ni u\mapsto f(u)+g(Lu)+h(u)$. Let us introduce the convex-concave function
\[
\scrL(u,v)\eqdef f(u)+h(u)-g^{\ast}(v)+\inner{Lu,v}\qquad\forall (u,v)\in\setH_{1}\times\setH_{2}.
\]
 Define
\begin{equation}\label{eq:Gamma}
\Gamma(x')\eqdef \sup_{u\in\setH_{1},v\in\setH_{2}}\left(\scrL(u',v)-\scrL(u,v')\right)\qquad \forall x'=(u',v')\in\setH=\setH_{1}\times\setH_{2}.
\end{equation}
It is easy to check that $\Gamma(x')\geq 0$, and equality holds only for a primal-dual pair (saddle-point) $\bar{x}\in\setS$. Hence, $\Gamma(\cdot)$ is a gap function for the monotone inclusion derived from the Karush-Kuhn-Tucker conditions \eqref{eq:KKT}. In fact, the function \eqref{eq:Gamma} is a standard merit function for saddle-point problems (see e.g. \cite{ChenLanOuy2014}). To relate this gap function to the Fitzpatrick function, we exploit the maximally monotone operators $V$ and $T$ introduced Example \ref{ex:PDA}. In terms of these mappings, first observe that for $p=(\tilde{u},\tilde{v}),x=(u,v)$ we have
\[
\inner{V(p),x-p}=\inner{\nabla h(\tilde{u}),u-\tilde{u}}+\inner{\tilde{v},Lu}-\inner{L\tilde{u},v}
\]
Since $h$ is convex differentiable, the classical gradient inequality reads as $h(u)-h(\tilde{u})\geq\inner{\nabla h(\tilde{u}),u-\tilde{u}}$. Using this estimate in the previous display shows 
\[
\inner{V(p),x-p}\leq h(u)-h(\tilde{u})-\inner{L\tilde{u},v}+\inner{\tilde{v},Lu}.
\]
For $p^{\ast}=(\tilde{u}^{\ast},\tilde{v}^{\ast})\in T(p)$, we again employ convexity to get 
\begin{align*}
f(u)&\geq f(\tilde{u})+\inner{\tilde{u}^{\ast},u-\tilde{u}}\qquad\forall u\in\setH_{1},\\
g^{\ast}(v)&\geq g^{\ast}(\tilde{v})+\inner{\tilde{v}^{\ast},v-\tilde{v}}\qquad\forall v\in\setH_{2}.
\end{align*}
Hence, 
\[
\inner{\tilde{u}^{\ast},u-\tilde{u}}+\inner{\tilde{v}^{\ast},v-\tilde{v}}\leq (f(u)-f(\tilde{u}))+(g^{\ast}(v)-g^{\ast}(\tilde{v})).
\]
Therefore, we see
\begin{align*}
\inner{V(p)+p^{\ast},x-p}&\leq \left(f(u)+h(u)-g^{\ast}(\tilde{v})+\inner{\tilde{v},Lu}\right)-\left(f(\tilde{u})+h(\tilde{u})-g^{\ast}(v)+\inner{v,L\tilde{u}}\right)\\
&=\scrL(u,\tilde{v})-\scrL(\tilde{u},v).
\end{align*}
Hence, 
\begin{align*}
\gap(x)=\sup_{(p,p^{\ast})\in\gr(T)}\inner{V(p)+p^{\ast},x-p}&\leq\sup_{(\tilde{u},\tilde{v})\in\setH_{1}\times\setH_{2}}\left(\scrL(u,\tilde{v})-\scrL(\tilde{u},v)\right)=\Gamma(x).
\end{align*}
\end{example}

It is clear from the definition that a convex gap function can be extended-valued and its domain is contingent on the boundedness properties of $\dom T$. In the setting where $T(x)$ is bounded for all $x\in\setH$, the gap function is clearly
globally defined. However, the case where $\dom T$ is unbounded has to be handled with more care. There are potentially two approaches to cope with such a
situation: One would be to introduce a perturbation-based termination criterion as defined in \cite{MonSva10}, and recently used in \cite{CheLanOuy17} to solve
a class of structured stochastic variational inequality problems. The other solution strategy is based on the notion of \emph{restricted merit functions}, first introduced in \cite{Nes07}, and later on adopted in \cite{Mal20}. We follow the latter strategy. 

Let $x^{s}\in\dom T$ denote an arbitrary reference point and $D>0$ a suitable constant. Define the closed set $\setC\eqdef\dom T\cap\{x\in\setH\vert\; \norm{x-x^{s}}\leq D\}$, and the \emph{restricted gap function}
\begin{equation}\label{eq:restrictedgap}
\gap(x\vert\setC)\eqdef\sup\{\inner{y^{\ast},x-y}\vert y\in\setC,y^{\ast}\in F(y)\}.
\end{equation}

Clearly, $\gap(x\vert\dom T)=\gap(x)$. The following result explains in a precise way the meaning of the restricted gap function. It extends the variational case in \cite[Lemma 1]{Nes07} and \cite[Lemma 3]{Mal20} to the general monotone inclusion case.

\begin{lemma}
Let $\setC\subset\setH$ be nonempty closed and convex. The function $\setH\ni x\mapsto \gap(x\vert\setC)$ is well-defined and convex on $\setH$. For any $x\in\setC$ we have $\gap(x\vert\setC)\geq 0$. Moreover, if $\bar{x}\in\setC$ is a solution to \eqref{eq:MI}, then $\gap(\bar{x}\vert\setC)=0$. Moreover, if $\gap(\bar{x}\vert\setC)=0$ for some $\bar{x}\in\dom T$ such that $\norm{\bar{x}-x^{s}}<D$, then $\bar{x}\in\setS$. 
\end{lemma}
\begin{proof}
The convexity and non-negativity for $x\in\setC$ of the restricted function is clear. Since $\gap(x\vert C)\leq \gap(x)$ for all $x\in\setH$, we see 
\[
\bar{x}\in\setS\iff \gap(\bar{x})=0\Rightarrow \gap(\bar{x}\vert\setC)=0.
\]
To show the converse implication, suppose $\gap(\bar{x}\vert\setC)=0$ for some $\bar{x}\in\setC$ with $\norm{\bar{x}-x^{s}}<D$. Without loss of generality we can choose $\bar{x}\in\setC$ in this particular way, since we may choose the radius of the ball as large as desired. It follows that $\inner{y^{\ast},\bar{x}-y}\leq 0$ for all $y\in\setC,y^{\ast}\in F(y)$. Hence, $\bar{x}\in\setC$ is a Minty solution to the Generalized Variational inequality with maximally monotone operator $F(x)+\NC_{\setC}(x)$. Since $F$ is upper semi-continuous and monotone, Minty solutions coincide with Stampacchia solutions, implying that there exists $\bar{x}^{\ast}\in F(\bar{x})$ such that $\inner{\bar{x}^{\ast},y-\bar{x}}\geq 0$ for all $y\in\setC$ (see e.g. \cite{BurMil20}). Consider now the gap program 
\begin{align*}
g_{\setC}(\bar{x},\bar{x}^{\ast})\eqdef \inf\{\inner{\bar{x}^{\ast},y-\bar{x}}\vert y\in\setC\}.
\end{align*}
This program is solved at $y=\bar{x}$, which is a point for which $\norm{x-x^{s}}<D$. Hence, the constraint can be removed, and we conclude $\inner{\bar{x}^{\ast},y-\bar{x}}\geq 0$ for all $y\in\dom(F)$. By monotonicity of $F$, it follows
\[
\inner{y^{\ast},y-\bar{x}}\geq \inner{\bar{x}^{\ast},y-\bar{x}}\geq 0 \quad\forall (y,y^{\ast})\in\gr(F).
\]
Hence, $\gap(\bar{x})=0$ and we conclude $\bar{x}\in\setS$.
\qed
\end{proof}

We start with the first preliminary result. 

\begin{lemma}\label{lem:gap}
Consider the sequence $(X_{k})_{k\in\N}$ generated by \ac{RISFBF} with the initial condition $X_{0}=X_{1}$. Suppose $\lambda_k=\lambda \in (0,1/(2L))$ for every $k \in \N$. Moreover, suppose $(\alpha_k)_{k\in\N}$ is a non-decreasing sequence such that $0<\alpha_k\le\bar{\alpha}<1$, $\rho_k=\tfrac{3(1-\bar{\alpha})^2}{2(2\alpha_k^2-\alpha_k+1)(1+L\lambda)}$ for every $k \in\N$. Define 
 \begin{equation}
 \Delta M_{k}\eqdef \frac{3\rho_{k}\lambda_{k}^{2}}{1+L\lambda_{k}} \norm{\ce_{k}}^{2}+\frac{\rho_{k}\lambda^{2}_{k}}{2}\norm{U_{k}}^{2}
 \end{equation}
 and for $(p,p^{\ast})\in\gr(F)$, we define $\Delta N_{k}(p,p^{\ast})$ as in \eqref{eq:N}. Then, for all $(p,p^{\ast})\in\gr(F)$, we have
\begin{equation}\label{eq:gapbound1}
\sum_{k=1}^{K}2\rho_{k}\lambda\inner{p^{\ast},Y_k-p}\leq (1-\alpha_1)\norm{X_{1}-p}^2+\sum_{k=1}^{K}\Delta M_{k}+\sum_{k=1}^{K}\Delta N_{k}(p,0).
\end{equation}
\end{lemma}
\begin{proof}
For $(p,p^{\ast})\in\gr(V+T)$, we know from eq. \eqref{eq:MT} 
\[
\inner{Z_{k}-R_{k},Y_{k}-p}\geq \lambda_{k}\inner{W_{k}+p^{\ast},Y_{k}-p}+\lambda_{k}\inner{V(Y_{k})-V(p),Y_{k}-p}\geq \inner{p^{\ast},Y_{k}-p}+\lambda_{k}\inner{W_{k},Y_{k}-p},
\]
where the last inequality uses the monotonicity of $V$. We first derive a recursion which is similar to the fundamental recursion in Lemma \ref{lem:Recursion}. Invoking \eqref{bd-eq1} and \eqref{bd-eq1sc}, we get
\begin{align}
\notag\norm{X_{k+1}-p}^{2}&\leq \norm{Z_{k}-p}^{2}-\frac{1-\rho_{k}}{\rho_{k}}\norm{X_{k+1}-Z_{k}}^{2}+2\lambda^{2}\rho_{k}\norm{\ce_{k}}^{2}-2\rho_{k}\lambda_{k}\inner{W_{k}+p^{\ast},Y_{k}-p} \nonumber\\
&-\rho_{k}(1-2L^{2}\lambda^{2}_{k})\norm{Y_{k}-Z_{k}}^{2}+\frac{\rho_{k}\lambda^{2}_{k}}{2}\norm{U_{k}}^{2}+2\rho_{k}\lambda_{k}\inner{V(Y_{k})-V(p),p-Y_{k}}. \label{g-1}
\end{align} 
Multiplying both sides of \eqref{bd-eq2sc} and noting that $(1-2L\lambda_{k})(1+L\lambda_{k})\leq 1-2L^{2}\lambda^{2}_{k}$, we obtain the following inequality
\begin{align*}
-\rho_{k}(1-2L^{2}\lambda^{2}_{k})\norm{Y_{k}-Z_{k}}^{2}\leq -\frac{1-2L\lambda_{k}}{2\rho_{k}(1+L\lambda_{k})}\norm{X_{k+1}-Z_{k}}^{2}+\frac{\rho_{k}\lambda^{2}_{k}(1-2L\lambda_{k})}{1+L\lambda_{k}}\norm{\ce_{k}}^{2}.
\end{align*}
Inserting the above inequality to \eqref{g-1} and using the same fashion in deriving \eqref{recur-Xk}, we arrive at 
\begin{align}
\notag
\norm{X_{k+1}-p}^{2}&\leq (1+\alpha_{k})\norm{X_{k}-p}^{2}-\alpha_{k}\norm{X_{k-1}-p}^{2}\\
&\notag+\Delta M_{k}+\Delta N_{k}(p,p^{\ast})-2\rho_{k}\lambda_{k}\inner{V(Y_{k})-V(p),Y_{k}-p}\\
&\notag+\alpha_{k}\norm{X_{k}-X_{k-1}}^{2}\left(2\alpha_{k}+\frac{3(1-\alpha_{k})}{2\rho_{k}(1+L\lambda_{k})}\right)\\
&-(1-\alpha_{k})\left(\frac{3}{2\rho_{k}(1+L\lambda_{k})}-1\right)\norm{X_{k+1}-X_{k}}^{2}. \label{g-2}
\end{align}
Invoking the monotonicity of $V$ and rearranging \eqref{g-2}, it follows that
\begin{align}
\notag &\|X_{k+1}-p\|^2+(1-\alpha_{k})\left(\tfrac{3}{2\rho_{k}(1+L\lambda_k)}-1\right)\|X_{k+1}-X_k\|^2-\alpha_{k}\|X_k-p\|^2 \\
\notag &\le \|X_k-p\|^2+(1-\alpha_k)\left(\tfrac{3}{2\rho_k(1+L\lambda_k)}-1\right)\|X_{k}-X_{k-1}\|^2-\alpha_k\|X_{k-1}-p\|^2+\Delta M_{k}+\Delta N_{k}(p,p^*) \\
\notag&+\underbrace{\left(2\alpha_k^2+(1-\alpha_k)\left(1-\tfrac{3(1-\alpha_k)}{2\rho_k(1+L\lambda_k)}\right)\right)}_{\tiny \le0}\|X_k-X_{k-1}\|^2 \\
\notag &\le \|X_k-p\|^2+(1-\alpha_k)\left(\tfrac{3}{2\rho_k(1+L\lambda_k)}-1\right)\|X_{k}-X_{k-1}\|^2-\alpha_k\|X_{k-1}-p\|^2+\Delta M_{k}+\Delta N_{k}(p,p^*).
\end{align}
We define $\beta_{k+1}$ as 
\begin{equation}
\notag\beta_{k+1}\eqdef (1-\alpha_{k})\left(\frac{3}{2\rho_{k}(1+L\lambda_{k})}-1\right)-(1-\alpha_{k+1})\left(\frac{3}{2\rho_{k+1}(1+L\lambda_{k+1})}-1\right), 
\end{equation}
and similarly with \eqref{beta}, we can show $\{\beta_k\}$ is non-increasing by choosing $\rho_k=\tfrac{3(1-\bar{\alpha})^2}{2(2\alpha_k^2-\alpha_k+1)(1+L\lambda_k)}$ and $\lambda_k \equiv \lambda$. Thus, $(1-\alpha_{k+1})\left(\tfrac{3}{2\rho_{k+1}(1+L\lambda_{k+1})}-1\right)\le(1-\alpha_{k})\left(\tfrac{3}{2\rho_{k}(1+L\lambda_k)}-1\right)$. Together with $\alpha_{k+1}\geq\alpha_{k}$, the last inequality gives
\begin{align}
\notag &\|X_{k+1}-p\|^2+(1-\alpha_{k+1})\left(\tfrac{3}{2\rho_{k+1}(1+L\lambda)}-1\right)\|X_{k+1}-X_k\|^2-\alpha_{k+1}\|X_k-p\|^2 \\
\notag &\le \|X_k-p\|^2+(1-\alpha_k)\left(\tfrac{3}{2\rho_k(1+L\lambda)}-1\right)\|X_{k}-X_{k-1}\|^2-\alpha_k\|X_{k-1}-p\|^2+\Delta M_{k}+\Delta N_{k}(p,p^*).
\end{align}
Recall that $\Delta N_{k}(p,p^{\ast})=\Delta N_{k}(p,0)+2\rho_{k}\lambda\inner{p^{\ast},p-Y_{k}}$. Hence, after setting $\Delta N_{k}(p,0)=\Delta N_{k}(p)$, rearranging the expression given in the previous display shows that
\begin{align*}
2\rho_{k}\lambda\inner{p^{\ast},Y_{k}-p}&\leq \left(\|X_k-p\|^2+(1-\alpha_k)\left(\tfrac{3}{2\rho_k(1+L\lambda)}-1\right)\|X_{k}-X_{k-1}\|^2-\alpha_k\|X_{k-1}-p\|^2\right)  \\
&- \left(\|X_{k+1}-p\|^2+(1-\alpha_{k+1})\left(\tfrac{3}{2\rho_{k+1}(1+L\lambda)}-1\right)\|X_{k+1}-X_k\|^2-\alpha_{k+1}\|X_k-p\|^2\right) \\
&+\Delta M_{k}+\Delta N_{k}(p).
\end{align*}
Summing over $k=1,\ldots,K$, we obtain
\begin{align*}
\sum_{k=1}^{K}2\rho_{k}\lambda\inner{p^{\ast},Y_k-p}&\leq \sum_{k=1}^{K}\left[\left(\|X_k-p\|^2+(1-\alpha_k)\left(\tfrac{3}{2\rho_k(1+L\lambda)}-1\right)\|X_{k}-X_{k-1}\|^2-\alpha_k\|X_{k-1}-p\|^2\right)\right.\\
&\left. - \left(\|X_{k+1}-p\|^2+(1-\alpha_{k+1})\left(\tfrac{3}{2\rho_{k+1}(1+L\lambda)}-1\right)\|X_{k+1}-X_k\|^2-\alpha_{k+1}\|X_k-p\|^2\right)\right] \\
    &+\sum_{k=1}^{K}\Delta M_{k}+\sum_{k=1}^{K}\Delta N_{k}(p)\\ 
    & \leq \|X_1-p\|^2+(1-\alpha_1)\left(\tfrac{3}{2\rho_1(1+L\lambda)}-1\right)\|X_{1}-X_{0}\|^2-\alpha_1\|X_{0}-p\|^2 \\
    &+\sum_{k=1}^{K}\Delta M_{k}+\sum_{k=1}^{K}\Delta N_{k}(p) \\
    &= (1-\alpha_1) \|X_1-p\|^2+\sum_{k=1}^{K}\Delta M_{k}+\sum_{k=1}^{K}\Delta N_{k}(p),
 \end{align*}
 where we notice $X_1=X_0$ in the last inequality.
 \qed
\end{proof}

Next, we derive a rate statement in terms of the gap function, using the averaged sequence 
\begin{equation}\label{eq:Xbar}
\bar{X}_K\eqdef \tfrac{\sum_{k=1}^{K}\rho_kY_{k}}{\sum_{k=1}^{K}\rho_k}.
\end{equation}

\begin{theorem}[{\bf Rate and oracle complexity under monotonicity of $V$}] 
\label{th:Oracle}
Consider the sequence $(X_{k})_{k\in\N}$ generated \ac{RISFBF}. Suppose Assumptions \ref{ass:exists}-\ref{ass:batch} hold. Suppose $m_k \triangleq \lfloor k^a\rfloor$ and $\lambda_k=\lambda \in (0,1/(2L))$ for every $k \in \N$ where $a > 1$. Suppose $(\alpha_k)_{k\in\N}$ is a non-decreasing sequence such that $0<\alpha_k\le\bar{\alpha}<1$, $\rho_k=\tfrac{3(1-\bar{\alpha})^2}{2(2\alpha_k^2-\alpha_k+1)(1+L\lambda)}$ for every $k \in\N$. Then the following hold for any $K\in\N$: 
\begin{itemize}
\item[{(i)}] $\Ex[\gap(\bar{X}_{K}\vert\setC)] \leq \scrO\left(\tfrac{1}{K}\right).$ 
\item[{(ii)}] Given $\eps>0$, define $K_{\eps}\eqdef\{k\in\N\vert\Ex[\gap(\bar{X}_{k}\vert\setC)]\leq\eps\}$, then $\sum_{k=1}^{K_{\eps}} m_k \leq \scrO\left(\tfrac{1}{\eps^{1+a}}\right).$    
\end{itemize}
\end{theorem}
The proof of this Theorem builds on an idea which is frequently used in the analysis of stochastic approximation algorithms, and can at least be traced back to the robust stochastic approximation approach of \cite{JudLanNemSha09}. In order to bound the expectation of the gap function, we construct an auxiliary process which allows us to majorize the gap via a quantity which is independent of the reference points. Once this is achieved, a simple variance bound completes the result. 

\begin{proof}[Proof of Theorem \ref{th:Oracle}]
We define an auxiliary process $(\Psi_{k})_{k\in\N}$ such that 
\begin{equation}
\Psi_{k+1}\eqdef \Psi_{k}+\rho_k\lambda_k W_{k},\quad \Psi_{1}\in\dom(T).
\end{equation}
Then, 
\begin{align*}
\norm{\Psi_{k+1}-p}^{2}&=\norm{(\Psi_{k}-p)+\rho_{k}\lambda_{k}W_{k}}^{2}=\norm{\Psi_{k}-p}^{2}+\rho^{2}_{k}\lambda_{k}^{2}\norm{W_{k}}^{2}+2\rho_{k}\lambda_{k}\inner{\Psi_{k}-p,W_{k}},
\end{align*}
so that
\[
2\rho_{k}\lambda_{k}\inner{W_{k},p-\Psi_{k}}=\norm{\Psi_{k}-p}^{2}-\norm{\Psi_{k+1}-p}^{2}+\rho^{2}_{k}\lambda_{k}^{2}\norm{W_{k}}^{2}.
\]
Introducing the iterate $Y_{k}$, the above implies
\begin{align*}
2\rho_{k}\lambda_{k}\inner{W_{k},p-Y_{k}}&=2\rho_{k}\lambda_{k}\inner{W_{k},p-\Psi_{k}}+2\rho_{k}\lambda_{k}\inner{W_{k},\Psi_{k}-Y_{k}}\\
&=\norm{\Psi_{k}-p}^{2}-\norm{\Psi_{k+1}-p}^{2}+\rho^{2}_{k}\lambda_{k}^{2}\norm{W_{k}}^{2}+2\rho_{k}\lambda_{k}\inner{W_{k},\Psi_{k}-Y_{k}}.
\end{align*}
As $\Delta N_{k}(p)=2\rho_{k}\lambda_{k}\inner{W_{k},p-Y_{k}}$, this implies via a telescopian sum argument 
\begin{equation}\label{eq:boundN}
\sum_{k=1}^{K}\Delta N_{k}(p)\leq \norm{\Psi_{1}-p}^{2}+\sum_{k=1}^{K}\rho^{2}_{k}\lambda^{2}_{k}\norm{W_{k}}^{2}+\sum_{k=1}^{K}2\rho_{k}\lambda_{k}\inner{W_{k},\Psi_{k}-Y_{k}}.
\end{equation}
Using Lemma \ref{lem:gap} and setting $\lambda_k \equiv \lambda$, for any $(p,p^{\ast})\in\gr(F)$ it holds true that 
\begin{align*}
\sum_{k=1}^{K}2\rho_{k}\lambda\inner{p^{\ast},Y_{k}-p}&\leq(1-\alpha_{1})\norm{X_{1}-p}^{2}+\sum_{k=1}^{K}\Delta M_{k}+\sum_{k=1}^{K}\Delta N_{k}(p).
\end{align*}
Define $\const_{1}\eqdef(1-\alpha_{1})\norm{X_{1}-p}^{2}$, divide both sides by $\sum_{k=1}^{K}\rho_{k}$ and using our definition of an ergodic average \eqref{eq:Xbar}, this gives
\[
2\lambda\inner{p^{\ast},\bar{X}_{K}-p}\leq \frac{1}{\sum_{k=1}^{K}\rho_{k}}\left\{\const_{1}+\sum_{k=1}^{K}\Delta M_{k}+\sum_{k=1}^{K}\Delta N_{k}(p)\right\}.
\]
Using the bound established in eq. \eqref{eq:boundN}, it follows 
\begin{align*}
2\lambda\inner{p^{\ast},\bar{X}_{K}-p}\leq \frac{1}{\sum_{k=1}^{K}\rho_{k}}\left\{\const_{1}+\sum_{k=1}^{K}\Delta M_{k}+\norm{\Psi_{1}-p}^{2}+\sum_{k=1}^{K}\rho^{2}_{k}\lambda^{2}\norm{W_{k}}^{2}+\sum_{k=1}^{K}2\rho_{k}\lambda\inner{W_{k},\Psi_{k}-Y_{k}}\right\}.
\end{align*}
Choosing $\Psi_{1},p\in\setC$ and introducing $\const_{2}\eqdef \const_{1}+4D^{2}$, we see that the above can be bounded by a random quantity which is independent of $p$: 
\begin{align*}
2\lambda\inner{p^{\ast},\bar{X}_{K}-p}\leq \frac{1}{\sum_{k=1}^{K}\rho_{k}}\left\{\const_{2}+\sum_{k=1}^{K}\Delta M_{k}+\sum_{k=1}^{K}\rho^{2}_{k}\lambda^{2}\norm{W_{k}}^{2}+\sum_{k=1}^{K}2\rho_{k}\lambda_{k}\inner{W_{k},\Psi_{k}-Y_{k}}\right\}.
\end{align*}
Taking the supremum over pairs $(p,p^{\ast})$ such that $p\in\scrC$ and $p^{\ast}\in F(y)$, it follows 
\begin{equation}\label{eq:gap1}
2\lambda\gap(\bar{X}_{K}\vert\setC)\leq \frac{\const_{2}}{\sum_{k=1}^{K}\rho_{k}}+\frac{\sum_{k=1}^{K}\Delta M_{k}+\sum_{k=1}^{K}\rho^{2}_{k}\lambda^{2}\norm{W_{k}}^{2}+\sum_{k=1}^{K}2\rho_{k}\lambda_{k}\inner{W_{k},\Psi_{k}-Y_{k}}}{\sum_{k=1}^{K}\rho_{k}}
\end{equation}
In order to proceed, we bound the first moment of the process $\Delta M_{k}$ in the same way as in \eqref{eq:boundM}, in order to get 
\begin{align*}
\Ex[\Delta M_{k}\vert\scrF_{k}]&\leq \frac{6\rho_{k}\lambda_{k}^2}{1+L\lambda_{k}}\Ex[\norm{W_{k}}^{2}\vert\scrF_{k}]+\lambda^{2}_{k}\left(\frac{6\rho_{k}}{1+L\lambda_{k}}+\frac{\rho_{k}\lambda^{2}_{k}}{2}\right)\Ex[\norm{U_{k}}^{2}\vert\scrF_{k}]\\
&= \frac{\left(\frac{12\rho_{k}\lambda^{2}_{k}}{1+L\lambda_{k}}\sigma^{2}+\frac{\rho_{k}\lambda^{2}_{k}}{2}\sigma^{2}\right)}{m_{k}}\eqdef\frac{\ca_{k}\sigma^{2}}{m_{k}}.
\end{align*}
Next, we take expectations on both sides of inequality \eqref{eq:gap1}, and use the bound \eqref{e:sigma}, and $\Ex[\inner{W_{k},\Psi_{k}-Y_{k}}]=\Ex\left[\Ex\left(\inner{W_{k},\Psi_{k}-Y_{k}}\vert\hat{\scrF}_{k}\right)\right]=0.$ This yields
\begin{align*}
2\lambda\Ex\left[\gap(\bar{X}_{K}\vert\setC)\right]\leq \frac{\const_{2}}{\sum_{k=1}^{K}\rho_{k}}+\frac{1}{\sum_{k=1}^{K}\rho_{k}}\left(\sum_{k=1}^{K}\frac{\ca_{k}\sigma^{2}}{m_{k}}+\sum_{k=1}^{K}\rho^{2}_{k}\lambda^{2}\frac{\sigma^{2}}{m_{k}}\right).
\end{align*}
Since $\alpha_{k}\uparrow\bar{\alpha}\in(0,1)$, we know that $\rho_{k}\geq \tilde{\rho}\eqdef\frac{3(1-\bar{\alpha}^{2})}{2(1+L\lambda)(2\bar{\alpha}^{2}+1)}$. Similarly, since $2\alpha^{2}_{k}-\alpha_{k}+1\geq 7/8$ for all $k$, it follows $\rho_{k}\leq\bar{\rho}\eqdef\frac{12(1-\bar{\alpha})^{2}}{7}$. Using this upper and lower bound on the relaxation sequence, we also see that $\ca_{k}\leq \lambda^{2}\left(\frac{12\bar{\rho}}{1+L\lambda}+\frac{\bar{\rho}}{2}\right)\equiv\bar{\ca}$, so that 
\begin{align*}
2\lambda\Ex\left[\gap(\bar{X}_{K}\vert\setC)\right]\leq \frac{\const_{2}}{\tilde{\rho}K}+\frac{1}{\tilde{\rho}K}\left(\bar{\ca}\sigma^{2}+\bar{\rho}^2\lambda^{2}\sigma^{2}\right)\sum_{k=1}^{K}\frac{1}{m_{k}}\leq\frac{\const_{3}}{K}
\end{align*}
where $\const_{3}\eqdef \frac{\const_{2}}{\tilde{\rho}}+\frac{1}{\tilde{\rho}}\left(\bar{\ca}\sigma^{2}+\bar{\rho}\lambda^{2}\sigma^{2}\right)\sum_{k=1}^{\infty}\frac{1}{m_{k}}$. Hence, defining the deterministic stopping time $K_{\eps}=\{k\in\N\vert\Ex[\gap(\bar{X}_{k}\vert\setC)]\leq\eps\}$, we see $K_{\eps}\geq\frac{\const_{3}}{2\lambda \eps}=\frac{\const_{4}}{\eps}$. 

\noindent (ii). Suppose $m_k=\lfloor k^a\rfloor$, for $a>1$. Then the oracle complexity to compute an $\bar{X}_K$ such that $\Ex[\gap(\bar{X}_{k}\vert\setC)] \leq \epsilon$ is bounded as 
\begin{align*}
\sum_{k=1}^Km_k&\le\sum_{k=1}^{\lceil(\const_{4}/\eps)\rceil}m_k\le\sum_{k=1}^{\lceil(\const_{4}/\eps)\rceil}k^a\le\int_{k=1}^{(\const_{4}/\eps)+1}x^a dx\le\tfrac{((\const_{4}/\eps)+1)^{a+1}}{a+1}\le\left(\tfrac{\const_{4}}{\eps^{a+1}}\right).
\end{align*}
\qed
\end{proof}

\begin{remark} {In the prior result, we employ a sampling rate $m_k = \lfloor k^a \rfloor$ where $a > 1$. This achieves the optimal rate of convergence. In contrast, the authors in~\cite{IusJofOliTho17} employ a sampling rate, loosely given by $m_k = \lfloor k^{1+a} (\ln(k))^{1+b} \rfloor$ where $a > 0, b \geq -1$ or $a = 0, b > 0$. We observe that when $a > 0$ and $b \geq -1$, the mini-batch size  grows faster than our proposed $m_k$ while it is comparable in the other case.}
\end{remark}

\section{Applications}
\label{sec:Applications}
%
In this section, we compare the proposed scheme with its SA counterparts on a class of monotone two-stage stochastic variational inequality problems (Sec.~\ref{sec:5.1}) and a supervised learning problem (Sec.~\ref{sec:5.2}) and discuss the resulting performance. 

\subsection{Two-stage stochastic variational inequality problems}\label{sec:5.1}
In this section, we describe some preliminary computational results obtained from the (RISFBF) method when applied to a class of two-stage stochastic variational inequality problems, recently introduced by Rockafellar and
Wets~\cite{rockafellar17stochastic}. 
 
Consider an imperfectly competitive market with $N$ firms playing a two-stage game. In the first stage, the firms decide upon their capacity level $x_{i}\in[l_{i},u_{i}]$, anticipating the expected revenues to be obtained in the second stage in which they compete by choosing quantities \`{a} la Cournot. The second-stage market is characterized by uncertainty as the per-unit cost $h_{i}(\xi_{i})$ is realized on the spot and cannot be anticipated. To compute an equilibrium in this game, we assume that each player is able to take stochastic recourse by determining production levels $y_i(\xi)$, contingent on random convex costs and capacity levels $x_i$. In order to bring this into the terminology for our problem, let use define the feasible set for capacity decisions of firm $i$ as $\scrX_{i}\eqdef[l_{i},u_{i}]\subset\R_{+}$. The joint profile of capacity decisions is denoted by an $N$-tuple $x=(x_{1},\ldots,x_{N})\in\scrX\eqdef\prod_{i=1}^{N}\scrX_{i}=\scrX$. The capacity choice of player $i$ is then determined as a solution to the parametrized problem (Play$_i(x_{-i})$) 
\begin{align}\tag{\mbox{Play$_i(x_{-i})$}}
	\min_{x_i \in \mathcal{X}_i} \, c_i(x_i) -\left( p(X)x_i - \Ex_{\xi}[\scrQ_i(x_i,\xi)]\right),
\end{align}
where $c_i: \scrX_i \to \Real_+$ is a $\tilde{L}^c_i$-smooth and convex cost function and $p(\cdot)$ denotes the
inverse-demand function defined as $p(X)= d-rX$,  $d, r > 0$. The function $\scrQ_i(\cdot,\xi)$ denotes the optimal cost function of firm $i$ in scenario $\xi\in\Xi$, assuming a value $\scrQ_i(x_i,\xi)$ when the capacity level $x_{i}$ is chosen. The recourse function $\Ex_{\xi}[\scrQ_i(\us{\cdot},\xi)]$ denotes the expectation of the optimal value of the player $i$'s second stage problem and is defined as
\begin{align}
    \tag{\mbox{Rec$_i(x_{-i})$}}
\scrQ_i(x_i,\xi) &\eqdef \min\{h_{i}(\xi)y_{i}(\xi)\vert y_{i}(\xi)\in[0,x_{i}]\}\\
&=\max\{\pi_{i}(\xi)x_{i}\vert \pi_{i}(\xi)\leq 0,h_{i}(\xi)-\pi_{i}(\xi)\geq 0\}.\nonumber
    \end{align}

A Nash equilibrium of this game is given by a tuple $(x^{\ast}_1, \cdots, x^{\ast}_N)$ where $ x^{\ast}_i \mbox{ solves } (\mbox{Play}_i(x_{-i}^\ast))$ for each $i=1,2,\ldots,N$. A simple computation shows that $Q_{i}(x_{i},\xi)=\min\{0,h_{i}(\xi)x_{i}\}$, and hence it is nonsmooth. In order to obtain a smoothed variant, we introduce $\scrQ_i^{\epsilon}(\cdot,\xi_i)$, defined as 
\[
\scrQ^{\epsilon}_i(x_i,\xi)\eqdef  \max\{ x_i \pi_i(\xi)  - \tfrac{\epsilon}{2} (\pi_i(\xi))^2\vert \pi_i(\xi) \leq 0, \pi_i(\xi) \leq h_i(\xi)\},\quad \epsilon>0.
\]
This is the value function of  a quadratric program, requiring the maximization of an $\epsilon$-strongly concave function. Hence, $\scrQ_{i}^{\epsilon}(x_{i},\xi)$ is single-valued and $\nabla_{x_i}\scrQ^{\epsilon}_i(\cdot,\xi)$ is $\tfrac{1}{\epsilon}$-Lipschitz and $\epsilon$-strongly monotone \cite[][Prop. 12.60]{RocWet98} for all $\xi\in\Xi$. The latter is explicitly given by 
\[
\nabla_{x_i}\scrQ^{\epsilon}_i(x_i,\xi)\eqdef \argmax\{x_i \pi_i(\xi)  - \tfrac{\epsilon}{2} (\pi_i(\xi))^2 \vert  \pi_{i}(\xi)\leq 0,\pi_i(\xi) \leq h_i(\xi)\}.
\]
Employing this smoothing strategy in our two-stage noncooperative game yields the individual decision problem
 \begin{align}\tag{\mbox{Play$^{\epsilon}_i(x_{-i})$}}
 (\forall i\in\{1,\ldots,N\}):   \min_{x_i \in \mathcal{X}_i} \ c_i(x_i) - p(X)x_i + \Ex_{\xi}[\scrQ^{\epsilon}_i(x_i,\xi)].
\end{align}
The necessary and sufficient equilibrium conditions of this $\epsilon$-smoothed game can be compactly represented by the inclusion problem (SGE$^{\epsilon}$)
\begin{align}\tag{SGE$^\epsilon$}
0 \in F^\epsilon(x) \triangleq V^{\epsilon}(x)+T(x), \mbox{ where } V^{\epsilon}(x) = C(x) + R(x) + D^{\epsilon}(x), T(x) =  \NC_{\mathcal X}(x), 
\end{align}
and $C$, $R$, and $D^\epsilon$ are single-valued maps given by 
$$ 
C(x) \eqdef \left(\begin{array}{c}c_1'(x_1) \\
	\vdots \\
    c_N'(x_N)\end{array}\right),\quad R(x) \eqdef r(X\text{\textbf{1}}+ x) - d, \text{ and } D^{\epsilon} (x) \eqdef \left(\begin{array}{c} \Ex_{\xi}[\nabla_{x_1}\mathcal{Q}_1^{\epsilon}(x_1,\xi)] \\
	\vdots \\
\Ex_{\xi}[\nabla_{x_N}\scrQ_N^{\epsilon}(x_N,\xi)]\end{array}\right). 
$$
We note that the interchange between the expectation and the gradient operator can be invoked based on smoothness requirements (cf.~\cite[Th.~7.47]{DenRusSha09}). The problem (SGE$^{\epsilon}$) aligns perfectly with the structured inclusion \eqref{eq:MI}, in which $T$ is a maximal monotone
map and $V$ is an expectation-valued maximally monotone map. In addition, we can quantify the Lipschitz constant of $V$ as $L_V = L_C + L_R
+ L_D^{\epsilon}$ ,where  $L_C = \max_{1\leq i\leq N} \tilde{L}^c_i$, $L_R = r\norm{\Id + \text{\textbf{1}}\text{\textbf{1}}^{\top}}_2 = r(N+1)$ and $L_D^{\epsilon} = \tfrac{1}{\epsilon}$. Here, $\Id$ is the $N\times N$ identity matrix, and $\text{\textbf{1}}$ is the $N\times 1$ vector consisting only of ones.  

\vspace{0.2cm}

\noindent {\em Problem parameters for 2-stage SVI.} \us{Our numerics are based on specifying $N=10$, $r = 0.1$,  and $d = 1$. We consider four problem settings of $L_V$ ranging from $10, \cdots, 10^4$ (See Table~\ref{cpb1}). For each setting, the problem parameters are defined as follows.}
\begin{enumerate}
    \item[(i)] \us{{\em Specification of $h_i(\xi)$.} The cost parameters $h_i(\xi_i) \triangleq \xi_i$ where $\xi_i \sim \Uni[-5,0]$ and $i = 1, \cdots, N$.} 
    \item[(ii)] \us{{\em Specification of $L_V, L_R,$ $L_D^{\epsilon}$, $L_C$, and $\hat{b}_1$.} Since $\norm{\Id + 11^T}_2 = 11$ when $N = 10$,  $L_R = r \norm{\Id + \text{\textbf{1}}\text{\textbf{1}}^{\top}} = 1.1$. Let $\epsilon$ be defined as $\epsilon=\frac{10}{L_{V}}$ and $L_D^{\epsilon} = \tfrac{1}{\epsilon} = \tfrac{L_V}{10}$. It follows that $L_C = L_V-L_R-L_D^{\epsilon}$ and  $\hat{b}_1= L_C$.}
    \item[(iii)] \us{{\em Specification of $c_i(x_i)$.} The cost function $c_i$ is defined as $c_i(x_i) = \tfrac{1}{2}\hat{b}_i x_i^2+a_i{x_i}$ where $a_1, \cdots, a_N \sim \Uni[2,3]$  and  $\hat{b}_2, \cdots, \hat{b}_N \sim \Uni[0,\hat{b}_1].$ Further, $a \triangleq [a_1,\dots,a_N]^T\in\Real^N$ and $B\eqdef \diag(\hat{b}_1,\dots,\hat{b}_N)$ is a diagonal matrix with
nonnegative elements.}   
 
\end{enumerate}

\noindent {\em Algorithm specifications.}
We compare (RISFBF) with a standard stochastic approximation (SA) scheme and a stochastic forward-backward-forward (SFBF) scheme. Solution quality is compared by estimating the residual function  $\residual(x)=\|x-\Pi_\mathcal{X}(x-\lambda V^\epsilon(x))\|$.  All of the schemes were implemented in \texttt{MATLAB} on a PC with 16GB RAM and 6-Core Intel Core i7 processor
(2.6GHz). \\
\smallskip

\noindent {(i) {\bf (SA)}: Stochastic approximation scheme.}  {The {\bf (SA)} scheme utilizes update (SA).
\vspace{-.2cm}
\begin{align}
\tag{SA} X_{k+1}:= \Pi_{\scrX} \left[X_k - \lambda_k \widehat{V}^{\epsilon}(X_k,\xi_k)\right],
\end{align}
where $V^{\epsilon}(X_k) = \Ex_{\xi}[\widehat{V}^{\epsilon}(X_k,\xi)]$ and $\lambda_k = \tfrac{1}{\sqrt{k}}$. The operator $\Pi_{\scrX}[\cdot]$ means the orthogonal projection onto the set $\scrX$. Note that $x_0$ is randomly generated in $[0,1]^N$. 

\noindent {(ii) {\bf (SFBF)}: Variance-reduced stochastic modified forward-backward scheme.
\begin{equation}\tag{SFBF}
\left\{\begin{array}{l} 
Y_{k}=  \Pi_\mathcal{X}[X_k - \lambda_k \us{A_{k}(X_{k})}], \\
X_{k+1}=Y_{k}-\lambda_k(\us{B_{k}(Y_k)-A_{k}(X_{k})}).
\end{array}\right.
\end{equation}
\us{where $A_{k}(X_{k})=\frac{1}{m_{k}}\sum_{t=1}^{m_{k}}\widehat{V}^{\epsilon}(X_k,\xi_k)$, $B_{k}(Y_k)=\frac{1}{m_{k}}\sum_{t=1}^{m_{k}}\widehat{V}^{\epsilon}(Y_k,\eta_k)$}. We choose a constant $\lambda_k\equiv\lambda=\tfrac{1}{4L_V}$. We assume
$m_k=\lfloor k^{1.01}\rfloor$ for merely monotone problems and $m_k=\lfloor 1.01^{k}\rfloor$ for
strongly monotone problems. \\
 
\noindent {(iii) {\bf (RISFBF)}: Relaxed inertial stochastic
forward-backward-forward scheme. We choose a
constant steplength $\lambda_k\equiv\lambda=\tfrac{1}{4L_V}$. In merely monotone settings, we utilize an increasing
sequence $\alpha_k=\alpha_0(1-\tfrac{1}{k+1})$, where $\alpha_0=0.1$, the
relaxation parameter sequence $\rho_k$ defined as 
$\rho_k=\tfrac{3(1-\alpha_0)^2}{2(2\alpha_k^2-\alpha_k+1)(1+L_V\lambda)}$, and 
$m_k=\lfloor k^{1.01}\rfloor$. In strongly monotone regimes, we choose a constant inertial
parameter $\alpha_k\equiv\alpha=0.1$, a constant relaxation parameter
$\rho_k\equiv\rho=1$, and $m_k=\lfloor 1.01^k\rfloor$.

\medskip 

\noindent In Fig.~\ref{pro1}, we compare the three schemes under maximal monotonicity and
strong monotonicity, respectively and examine their senstivities to inertial
and relaxation parameters. \us{Both sets of plots are based on selecting $L_V=10^2$.}

\begin{figure}[htbp]
\includegraphics[width=.48\textwidth,height=5.2cm]{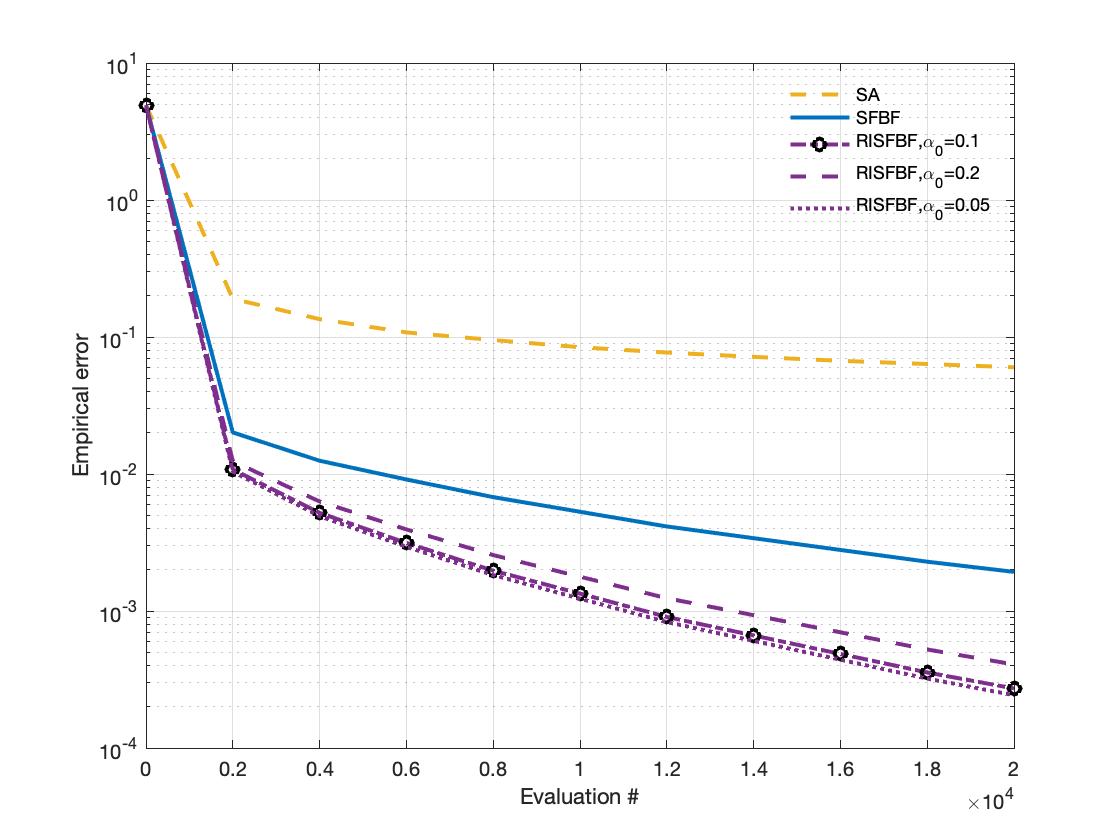}\hfill
\includegraphics[width=.48\textwidth,height=5.2cm]{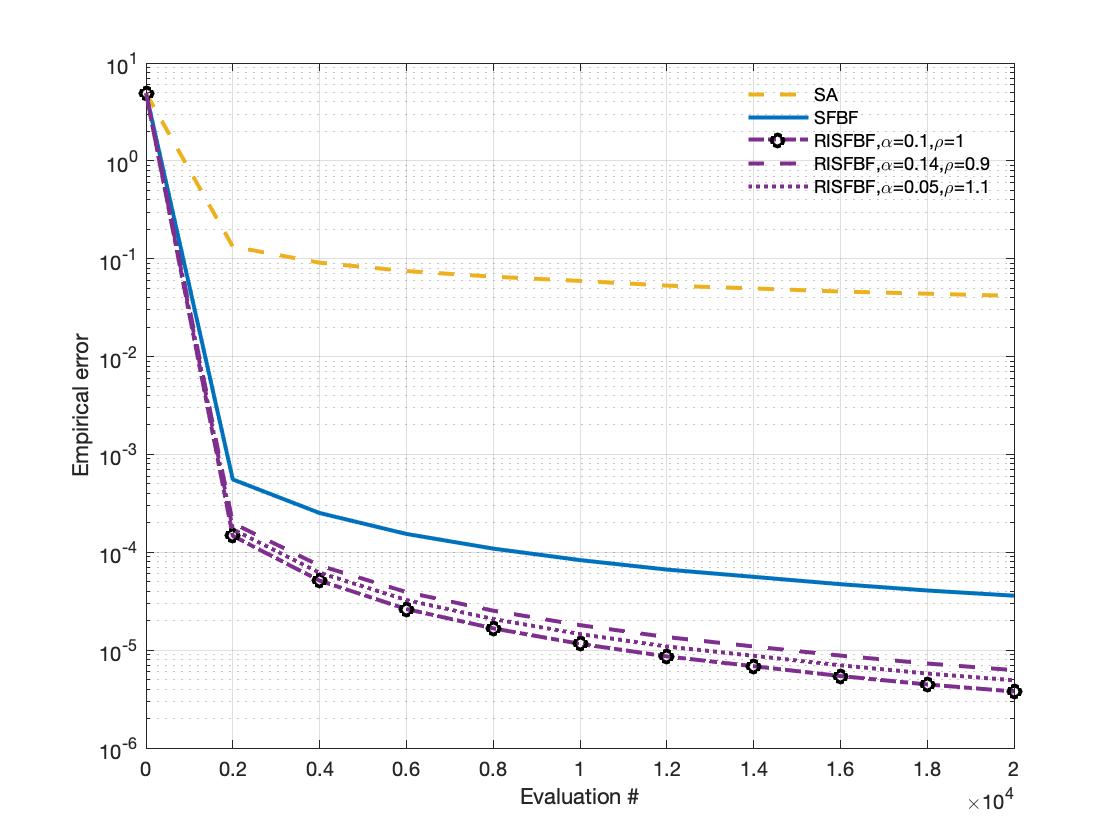}
\caption{Trajectories for (SA), (SFBF), and (RISFBF)  (left: monotone, right: s-monotone)}
\label{pro1}
\vspace{-0.2in}
\end{figure}

\begin{table}[h]
\scriptsize
\caption{Comparison of ({\bf RISFBF}) with (SA) and (SFBF) under various Lipschitz constant}
\begin{center}
    \begin{tabular}[t]{ c | c | c | c | c | c | c | c | c | c}
    \hline
    \multicolumn{10}{c} {merely monotone, 20000 evaluations} \\ \hline
    \multirow{2}{*}{$L_V$} & \multicolumn{3}{|c|} {RISFBF} & \multicolumn{3}{c|} {SFBF} & \multicolumn{3}{c} {SA} \\ \cline{2-10}
      & error & time & CI & error & time & CI & error & time & CI \\ \hline
      1e1 & 2.2e-4 & 2.7 & [2.0e-4,2.5e-4] & 1.6e-3 & 2.6 & [1.3e-3,1.8e-3] & 5.3e-2 & 2.7 & [5.0e-2,5.7e-2] \\ \hline
      1e2 & 2.7e-4 & 2.7 & [2.5e-4,3.0e-4] & 1.9e-3 & 2.6 & [1.6e-3,2.1e-3] & 6.1e-2 & 2.7 & [5.8e-2,6.4e-2]  \\ 
 \hline
      1e3 & 6.9e-4 & 2.7 & [6.7e-3,7.1e-4] & 2.2e-3 & 2.6 & [2.0e-3,2.5e-3] & 7.6e-2 & 2.5 & [7.3e-2,7.9e-2] \\ 
 \hline
 1e4 & 2.7e-3  & 2.7 & [2.5e-3,3.0e-3] & 5.9e-3 & 2.6 & [5.4e-3,6.2e-3] & 9.4e-2 & 2.6 & [9.0e-1,9.7e-1]\\ 
 \hline
 \hline
     \hline
    \multicolumn{10}{c} {strongly monotone, 20000 evaluations} \\ \hline
    \multirow{2}{*}{$L_V$} & \multicolumn{3}{|c|} {RISFBF} & \multicolumn{3}{c|} {SFBF} & \multicolumn{3}{c} {SA} \\ \cline{2-10}
      & error & time & CI & error & time & CI & error & time & CI \\ \hline
     1e1 & 1.5e-6 & 2.6 & [1.3e-6,1.7e-6] & 1.5e-5 & 2.6 & [1.2e-5,1.7e-5] & 2.9e-2 & 2.5 & [2.7e-2,3.1e-2]  \\ \hline
      1e2 & 3.7e-6 & 2.6 & [3.5e-6,3.9e-6] & 3.6e-5 & 2.5 & [3.3e-5,3.9e-5] & 4.1e-2 & 2.5 & [3.8e-2,4.4e-2] \\ 
 \hline
      1e3 & 4.5e-6 & 2.6 & [4.3e-6,4.7e-6] & 5.6e-5 & 2.5 & [4.2e-6,4.7e-6] & 5.5e-2 & 2.4 & [5.2e-2,5.7e-2] \\ 
 \hline
      1e4 & 1.4e-5 & 2.6 & [1.1e-5,1.7e-5] & 7.4e-5 & 2.5 & [7.1e-5,7.7e-5] & 6.0e-2 & 2.5 & [5.7e-2,6.3e-2] \\ 
 \hline
\end{tabular}
\end{center}
\label{cpb1}
\end{table}

\vspace{0.2cm}
\noindent {\bf Key insights.} Several insights may be drawn from Table~\ref{cpb1} and Figure~\ref{pro1}.\\ 

\noindent (a) First, from Table~\ref{cpb1}, one may conclude that on this class of problems,  (RISFBF) and (SFBF) significantly outperform (SA) schemes,
which is less surprising given that both schemes employ an increasing
mini-batch sizes, leading to performance akin to that seen in deterministic
schemes.  {We should note that when $\mathcal{X}$ is somewhat more complicated, the difference in run-times between SA schemes and mini-batch variants becomes for more pronounced; in this instance, the set $\mathcal{X}$ is relatively simple to project onto  and there is little difference in run-time across the three schemes.} \\

\noindent (b) Second, we observe that  while both (SFBF) and (RISFBF) schemes can
contend with poorly conditioned problems, as \us{seen} by noting that as $L_V$
grows, \us{their} performance does not degenerate \us{significantly} in terms of empirical error; However, in both
monotone and strongly monotone regimes, (RISFBF) provide consistently better solutions in terms of empirical error over (SFBF). Figure~\ref{pro1} displays the range of
trajectories obtained for differing relaxation and inertial parameters and in
the instances considered, (RISFBF) shows consistent benefits over (SFBF).   \\

\noindent (c) Third, since such schemes display geometric rates of convergence
for strongly monotone inclusion problems, this improvement is
reflected in terms of the empirical \us{errors} for strongly monotone vs monotone regimes.    

\subsection{Supervised learning with group variable selection}\label{sec:5.2}
Our second numerical example considers the following population risk formulation of a composite absolute penalty (CAP) problem arising in supervised statistical learning~\cite{RocYuZha09}
\begin{equation}\label{eq:CAP}\tag{CAP}
\min_{w\in \scrW} \tfrac{1}{2}\Ex_{(a,b)}[( a^{\top}w-b)^2]+\eta\sum_{\textsl{g}\in \scrS} \|w_\textsl{g}\|_2,
\end{equation}}
\noindent where the feasible set $\scrW\subseteq\R^{d}$ is a Euclidean ball with $\scrW \triangleq \{w
\in \R^{d}\mid \|w\|_2 \le D\}$, $\xi=(a,b)\in\R^{d}\times\R$ denotes the random variable consisting of a set of predictors $a$ and output $b$. The parameter vector $w$ is the sparse linear hypothesis to be learned. The sparsity structure of $w$ is represented by group
$\scrS \in 2^{\{1,\dots,l\}}$. When the groups in $\scrS$ do not
overlap, $\sum_{\textsl{g}\in \us{\scrS}} \|w_\textsl{g}\|_2$ is referred to
as the group lasso penalty~\cite{FriHasTib01,JacOboVer09}. When the groups in $\scrS$
form a partition of the set of predictors, then $\sum_{\textsl{g}\in\scrS} \|w_\textsl{g}\|_2$ is a norm afflicted by singularities when
some components $w_\textsl{g}$ are equal to zero. For any $\textsl{g} \in \{1, \cdots,
l\}$, $w_{\textsl{g}}$ is a sparse vector constructed by components of
$x$ whose indices are in $\textsl{g}$, i.e., $w_\textsl{g} :=
(w_i)_{i\in\textsl{g}}$ with few non-zero components in
$w_\textsl{g}$. Here, we assume that each group $\textsl{g} \in \mathcal{S}$
consists of $k$ elements. Introduce the linear operator $L:\R^{d}\to\R^{k}\underbrace{\times\cdots\times}_{l-\text{times}}\R^{k}$, given by $Lw=[\eta w_{g_{1}},\ldots,\eta w_{g_{l}}]$. Let us also define 
\begin{align*}
 &Q=\Ex_{\xi}[aa^{\top}],q=\Ex_{\xi}[ab],c=\frac{1}{2}\Ex_{\xi}[b^{2}],  \\
 &h(w)\eqdef\frac{1}{2}w^{\top}Qw-w^{\top}q+c, \text{ and } f(w)\eqdef\delta_{\scrW}(w),
 \end{align*}
where $\delta_{\scrW}(\cdot)$ denotes the indicator function with respect to the set $\scrW$.
Then \eqref{eq:CAP} becomes 
\begin{align*}
\min_{w\in\R^{d}}\ \{h(w)+g(Lw)+f(w)\},\quad \text{where }g(y_{1},\ldots,y_{l})\eqdef\sum_{i=1}^{l}\norm{y_{i}}.
\end{align*}
This is clearly seen to be a special instance of the convex programming problem \eqref{eq:primal}. Specifically, we let $\setH_{1}=\R^{d}$ with the standard Euclidean norm, and $\setH_{2}=\R^{k}\underbrace{\times\cdots\times}_{l-\text{times}}\R^{k}$ with the product norm 
\[
\norm{(y_{1},\ldots,y_{l})}_{\setH_{2}}\eqdef \sum_{i=1}^{l}\norm{y_{i}}_{2}.
\]
Since 
\[
g^{\ast}(v_{1},\ldots,v_{l})=\sum_{i=1}^{l}\delta_{\B(0,1)}(v_{i})\qquad\forall v=(v_{1},\ldots,v_{l})\in\R^{k}\underbrace{\times\cdots\times}_{l-\text{times}}\R^{k},
\]
the Fenchel-dual takes the form \eqref{eq:dual}.  Accordingly, a primal-dual pair for \eqref{eq:CAP} is a root of the monotone inclusion \eqref{eq:MI} with 
 \[
    V(w,v)= ( \nabla h(w)+L^{\ast}v, -Lw ) \text{ and }    T(w,v)\eqdef  \partial f(w)\times \partial g^{\ast}(v)
\]
involving $d+kl$ variables. 

%
%

\noindent {\em Problem parameters for \eqref{eq:CAP}:} We simulated data with $d = 82$, covered by 10 groups of 10 variables with 2 variables of overlap between two successive groups: $\{1,\dots,10\},\{9,\dots,18\},\dots,\{73,\dots,82\}$. We assume the nonzeros of $w_{\rm true}$ lie in the union of groups 4 and 5 and sampled from i.i.d. Gaussian variables. The operator $V(w,v)$ is estimated by the mini-batch estimator using $m_{k}$ iid copies of the random input-output pair $\xi=(a,b)\in\R^{d}\times\R$. Specifically, we draw each coordinate of the random vector $a$ from the standard Gaussian distribution $\Normal(0,1)$ and generate $b=a^{\top}w_{\rm true}+\eps$, for $\eps\sim \Normal(0,\sigma^{2}_{\eps})$. In the concrete experiment \us{reported} here, the error variance is taken as $\sigma_{\eps}=0.1$. In all instances, the regularization parameter is chosen as $\eta = 10^{-4}$. The accuracy of feature extraction of algorithm output $w$ is evaluated by the relative error to the ground truth, defined as 
$$\frac{\|w-w_{\rm true}\|_2}{\|w_{\rm true}\|_2}.$$

\smallskip

\noindent {\em Algorithm specifications.} 
We compare (RISFBF) with stochastic extragradient (SEG) and stochastic forward-backward-forward (SFBF) schemes and specify their algorithm parameters.  Again, all the schemes are run on MATLAB 2018b on a PC with 16GB RAM and 6-Core Intel Core i7 processor
(2.6$\times$8GHz). 

\smallskip


\noindent {(i) {\bf (SEG)}: Variance-reduced stochastic extragradient scheme.}  {Set $\scrX\eqdef \scrW\times\dom(g^{\ast})$. The {\bf (SEG)} scheme \cite{IusJofOliTho17} utilizes update (SEG).
\begin{align}\tag{SEG} 
\begin{aligned}
    Y_{k} &:= \Pi_{\scrX} \left[X_k - \lambda_k \us{A_{k}(X_k)}\right],\\
    X_{k+1} &:= \Pi_{\scrX} \left[X_k - \lambda_k \us{B_{k}(Y_{k})}\right],
\end{aligned}
\end{align}
\us{where $A_{k}(X_{k})=\frac{1}{m_{k}}\sum_{t=1}^{m_{k}}V(X_k,\xi_k)$, $B_{k}(Y_k)=\frac{1}{m_{k}}\sum_{t=1}^{m_{k}}V(Y_k,\eta_k)$}. In this scheme, $\lambda_k\equiv\lambda$ is chosen to be $\tfrac{1}{4\us{L_V}}$ ($\us{L_V}$ is the Lipschitz constant of $V$). We assume $m_k=\left\lfloor \tfrac{k^{1.1}}{n}\right\rfloor$.

\noindent {(ii) {\bf (SFBF)}: Variance-reduced stochastic modified forward-backward scheme. We employ the algorithm paramters employed in (i). Specifically, we  
 choose a constant $\lambda_k\equiv\lambda=\tfrac{1}{4\us{L_V}}$  and $m_k=\left\lfloor \tfrac{k^{1.1}}{n}\right\rfloor$. \\
 
  \noindent {(iii) {\bf (RISFBF)}: Relaxed inertial stochastic forward-backward-forward scheme. Here, 
we employ a constant steplength $\lambda_k\equiv\lambda=\tfrac{1}{4\us{L_V}}$, an increasing sequence $\alpha_k=\alpha_0(1-\tfrac{1}{k+1})$, where $\alpha_0=0.85$, a relaxation parameter sequence $\rho_k=\tfrac{3(1-\alpha_0)^2}{2(2\alpha_k^2-\alpha_k+1)(1+\us{L_V}\lambda)}$, and  assume $m_k=\left\lfloor \tfrac{k^{1.1}}{n}\right\rfloor$.

\begin{table}[h] 
\caption{The comparison of the RISFBF, SFBF and SEG algorithms in solving \eqref{eq:CAP} \label{rss}}
\small
\begin{center}
\begin{threeparttable} 

    \begin{tabular}[t]{ l c  l  r c l  rc  l  c }
    Iteration && RISFBF & && SFBF & && SEG & \\ \cline{3-4} \cline{6-7} \cline{9-10}
    $N$ && Rel. error  & CPU && Rel. error  & CPU && Rel. error  & CPU \\  \hline
    \phantom v400 && 5.4e-1 & 0.1 && 34.6 & 0.1 && 34.7 & 0.1  \\
    \phantom v800 && 8.1e-3 & 0.5 && 1.1e-1 & 0.5 && 1.5e-1 & 0.5  \\
    1200 && 6.0e-3 & 1.1 && 2.4e-2 & 1.1 && 2.4e-2 & 1.1  \\
    1600 && 5.2e-3 & 2.0 && 2.0e-2 & 2.0 && 1.9e-2 & 2.0  \\
    2000 && 4.6e-3 & 3.1 && 1.6e-2 & 3.1 && 1.5e-2 & 3.1  \\ \hline

\end{tabular}
\begin{tablenotes}
\small
\item The relative error and CPU time in the table is the average results of 20 runs
    \end{tablenotes}
  \end{threeparttable}
\end{center}
\end{table}

\noindent {\bf Insights.} We compare the performance of the schemes  in Table \ref{rss} \us{and}
observe that (RISFBF) outperforms \us{its competitors} others in extracting the underlying feature of the datasets. In Fig. \ref{tra}, trajectories for (RISFBF), (SFBF) and (SEG) are presented \us{where a consistent benefit of employing (RISFBF) can be seen for a range of choices of $\alpha_0$.}
\begin{figure}[htbp]
\centering
\includegraphics[width=.48\textwidth,height=5.2cm]{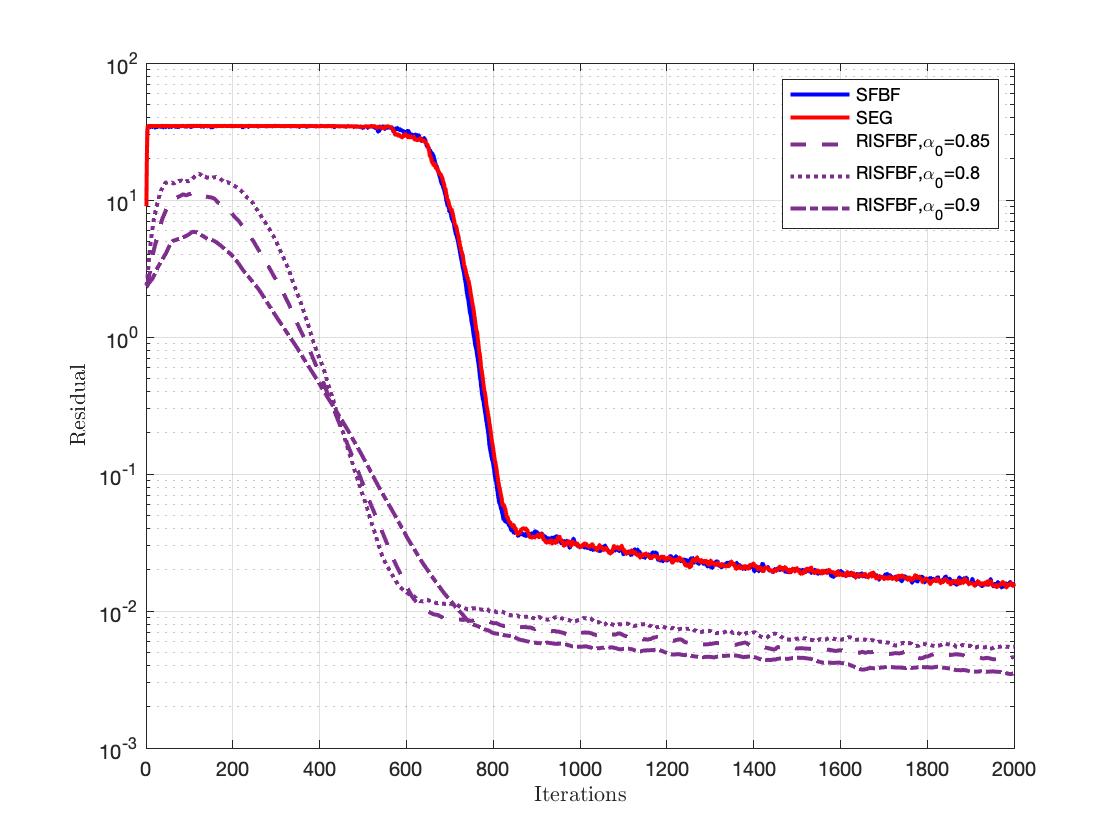}
\caption{Trajectories for (SEG), (SFBF), and (RISFBF) for problem \eqref{eq:CAP}. }
\label{tra}
\vspace{-0.2in}
\end{figure}

\section{Conclusion}
\label{sec:conclusion}
%

In a general structured monotone inclusion setting in Hilbert spaces, we
introduce a relaxed inertial stochastic algorithm based on Tseng's
forward-backward-forward splitting method. Motivated by the gaps in convergence
claims and rate statements in both deterministic and stochastic regimes, we
develop a variance-reduced framework and make the following contributions: (i)
Asymptotic convergence guarantees are provided under both increasing and
constant mini-batch sizes, the latter requiring somewhat stronger assumptions
on $V$; (ii)  When $V$ is monotone, rate statements   provided in terms of a restricted gap function,
inspired by the Fitzpatrick function for inclusions, show that the expected gap
of an averaged sequence diminishes at the rate of $\mathcal{O}(1/k)$ and oracle
complexity of computing an $\epsilon$-solution is
$\mathcal{O}(1/\epsilon^{1+a})$ where $a> 1$; (iii) When $V$ is strongly monotone, a non-asymptotic linear rate statement can be proven with an oracle complexity of $\mathcal{O}(\log(1/\epsilon))$ of computing an $\epsilon$-solution. \us{In addition, a perturbed linear rate is also developed. It is worth emphasizing that the rate statements in the strongly monotone regime accommodate the possibility of a biased SO.} Unfortunately, the growth rates in batch-size may be onerous in some situations, motivating the analysis of a polynomial growth rate in sample-size which is easily modulated. This leads to an associated polynomial rate of convergence.

Various open questions arise from our analysis. First, we exclusively focused
on a variance reduction technique based on increasing mini-batches. From the
point of view of computations and oracle complexity, this approach can become
quite costly. Exploiting different variance reduction techniques, taking
perhaps special structure of the single-valued operator $V$ into account (as in
\cite{PalBac16}), has the potential of improving the computational complexity
of our proposed method. At the same time, this will complicate the analysis of
the variance of the stochastic estimators considerably and \us{consequently,} we leave this as
an important question for future research. 

Second, our analysis needs knowledge about the Lipschitz constant $L$. \us{While in deterministic regimes, line search techniques have obviated such a need, such avenues are far more challenging to adopt in stochastic regimes. Efforts to address this in variational regimes have centered around leveraging empirical process theory~\cite{Iusem:2019ws}. This remains a goal of future research.}  \us{Another avenue emerges in 
applications} where we can gain a reasonably good
estimate about this quantity via some pre-processing of the data (see e.g.
Section 6 in \cite{GhaLanHon16}). \us{Developing such an adaptive framework} robust 
to noise is an important topic for future research.

\section*{Acknowledgments}
The authors thank Radu I. Bo\k{t} and Robert E. Csetnek for valuable discussions on this topic. Special thanks to Patrick Johnstone for helping us to clarify the discussion in Section 4.2.

\begin{appendix}
\section{Technical Facts}
\label{sec:techncial}
%

\begin{lemma}\label{lem:ab}
For $x,y\in\setH$ and scalars $\alpha,\beta\geq 0$ with $\alpha+\beta=1$, it holds that 
\begin{equation}
\norm{\alpha x+\beta y}^2 = \alpha\norm{x}^2 + \beta\norm{y}^2- \alpha \beta \norm{x-y}^2.
\end{equation}
\end{lemma}

We recall the Minkowski inequality: For $X,Y\in L^{p}(\Omega,\scrF,\Pr;\setH),\scrG\subseteq\scrF$ and $p\in[1,\infty]$, 
\begin{equation}
\label{eq:Minkowski}
\Ex[\norm{X+Y}^{p}\vert\scrG]^{1/p}\leq \Ex[\norm{X}^{p}\vert\scrG]^{1/p}+\Ex[\norm{Y}^{p}\vert\scrG]^{1/p}.
\end{equation}

In the convergence analysis, we use the Robbins-Siegmund Lemma~\cite[Lemma 11,~pg.~50]{Pol87}.
\begin{lemma}[Robbins-Siegmund]
\label{lem:RS}
Let $(\Omega,\scrF,\F=(\scrF_{n})_{n\geq 0},\Pr)$ be a discrete stochastic basis. Let $(v_{n})_{n\geq 1},(u_{n})_{n\geq 1}\in\ell^{0}_{+}(\F)$ and $(\theta_{n})_{n\geq 1},(\beta_{n})_{n\geq 1}\in\ell^{1}_{+}(\F)$ be such that  for all $n \geq 0$,
$$
\Ex[v_{n+1}\vert\scrF_{n}]\leq(1+\theta_{n})v_{n}-u_{n}+\beta_{n}\qquad  \Pr-\text{a.s. }.
$$
Then $(v_{n})_{n\geq 0}$ converges a.s. to a random variable $v$, and $(u_{n})_{n\geq 1}\in\ell^{1}_{+}(\F)$.
\end{lemma}
The next technical lemma will be needed in deriving a linear convergence result. 

\begin{lemma}\label{lem:geometric}
Let $z\geq 0$ and $0<q<p<1$. Then, if $D\geq \frac{1}{\exp(1)\ln(p/q)}$, it holds true that $zq^{z}\leq D p^{z}$ for all $z\geq 0$.
\end{lemma}
\begin{proof}
We want to find a positive constant $D_{\min}>0$ such that $D_{\min}\exp(z\ln(p))= z\exp(z\ln(q))$ for all $z>0$. Choosing $D$ larger than this, gives a valid value. Rearranging, this is equivalent to $D=z\left(\frac{q}{p}\right)^{z}\geq 0$ for all $z\geq 0$, or, which is still equivalent to 
$\ln(D)-\ln(z)-z\ln(q/p)=0.$ Define the extended-valued function $f:[0,\infty)\to[-\infty,\infty]$ by 
$f(z)=\ln(D)-\ln(z)-\ln(q/p)$ if $z>0$, and $f(z)=\infty$ if $z=0$. Then, for all $z>0$, simple calculus show $f'(z)=-1/z-\ln(q/p)$ and $f''(z)=1/z^{2}$. Hence, $z\mapsto f(z)$ is a convex function with a unique minimum $z_{\min}=\frac{1}{\ln(p/q)}>0$ and a corresponding function value $f(z_{\min})=\ln(D)+\ln(\ln(p/q))+1$. Hence, for $D\geq D_{\min}=\frac{1}{\exp(1)\ln(p/q)}$, we see that $f(z_{\min})>0$, and thus $zq^{z}\leq D p^{z}$ for all $z\geq 0$. 
\qed
\end{proof}

\end{appendix}

\bibliographystyle{plainnat}
\bibliography{mybib}

\begin{thebibliography}{88}
\providecommand{\natexlab}[1]{#1}
\providecommand{\url}[1]{\texttt{#1}}
\expandafter\ifx\csname urlstyle\endcsname\relax
  \providecommand{\doi}[1]{doi: #1}\else
  \providecommand{\doi}{doi: \begingroup \urlstyle{rm}\Url}\fi

\bibitem[Attouch and Cabot(2019)]{AttCab19}
Hedy Attouch and Alexandre Cabot.
\newblock Convergence of a relaxed inertial forward--backward algorithm for
  structured monotone inclusions.
\newblock \emph{Applied Mathematics \& Optimization}, 80\penalty0 (3):\penalty0
  547--598, 2019.
\newblock \doi{10.1007/s00245-019-09584-z}.
\newblock URL \url{https://doi.org/10.1007/s00245-019-09584-z}.

\bibitem[Attouch and Cabot(2020)]{AttCab20}
Hedy Attouch and Alexandre Cabot.
\newblock Convergence of a relaxed inertial proximal algorithm for maximally
  monotone operators.
\newblock \emph{Mathematical Programming}, 184\penalty0 (1):\penalty0 243--287,
  2020.
\newblock \doi{10.1007/s10107-019-01412-0}.
\newblock URL \url{https://doi.org/10.1007/s10107-019-01412-0}.

\bibitem[Attouch and Maing{\'e}(2011)]{AttMai11}
Hedy Attouch and Paul-Emile Maing{\'e}.
\newblock Asymptotic behavior of second-order dissipative evolution equations
  combining potential with non-potential effects.
\newblock \emph{ESAIM: Control, Optimisation and Calculus of Variations},
  17\penalty0 (3):\penalty0 836--857, 2011.

\bibitem[Attouch and Peypouquet(2019)]{AttPey19}
Hedy Attouch and Juan Peypouquet.
\newblock Convergence of inertial dynamics and proximal algorithms governed by
  maximally monotone operators.
\newblock \emph{Mathematical Programming}, 174\penalty0 (1):\penalty0 391--432,
  2019.
\newblock ISSN 1436-4646.
\newblock \doi{10.1007/s10107-018-1252-x}.
\newblock URL \url{https://doi.org/10.1007/s10107-018-1252-x}.

\bibitem[Attouch et~al.(2010)Attouch, Briceno-Arias, and
  Combettes]{AttBriCom10}
H{\'e}dy Attouch, Luis~M. Briceno-Arias, and Patrick~L. Combettes.
\newblock A parallel splitting method for coupled monotone inclusions.
\newblock \emph{SIAM Journal on Control and Optimization}, 48\penalty0
  (5):\penalty0 3246--3270, 2021/03/28 2010.
\newblock \doi{10.1137/090754297}.
\newblock URL \url{https://doi.org/10.1137/090754297}.

\bibitem[Auslender et~al.(1974)Auslender, Gourgand, and Guillet]{Aus74}
A.~Auslender, M.~Gourgand, and A.~Guillet.
\newblock Resolution numerique d'inegalites variationnelles.
\newblock In \emph{Lecture Notes in Economics and Mathematical Systems
  (Mathematical Economics)}, 1974.

\bibitem[Barty et~al.(2007)Barty, Roy, and Strugarek]{BarRoyStr07}
Kengy Barty, Jean-S{\'e}bastien Roy, and Cyrille Strugarek.
\newblock Hilbert-valued perturbed subgradient algorithms.
\newblock \emph{Mathematics of Operations Research}, 32\penalty0 (3):\penalty0
  551--562, 2020/11/23 2007.
\newblock \doi{10.1287/moor.1070.0253}.
\newblock URL \url{https://doi.org/10.1287/moor.1070.0253}.

\bibitem[Barty et~al.(2009)Barty, Roy, and Strugarek]{BarRoyStru09}
Kengy Barty, Jean-S{\'e}bastien Roy, and Cyrille Strugarek.
\newblock A stochastic gradient type algorithm for closed-loop problems.
\newblock \emph{Mathematical Programming}, 119\penalty0 (1):\penalty0 51--78,
  2009.
\newblock \doi{10.1007/s10107-007-0201-x}.
\newblock URL \url{https://doi.org/10.1007/s10107-007-0201-x}.

\bibitem[Bauschke and Combettes(2016)]{BauCom16}
Heinz~H. Bauschke and Patrick~L. Combettes.
\newblock \emph{Convex Analysis and Monotone Operator Theory in Hilbert
  Spaces}.
\newblock Springer - CMS Books in Mathematics, 2016.

\bibitem[B{\"o}rgens and Kanzow(2021)]{Boergens2021}
Eike B{\"o}rgens and Christian Kanzow.
\newblock Admm-type methods for generalized nash equilibrium problems in
  hilbert spaces.
\newblock \emph{SIAM J. Optim.}, pages 377--403, January 2021.
\newblock ISSN 1052-6234.
\newblock \doi{10.1137/19M1284336}.
\newblock URL \url{https://doi.org/10.1137/19M1284336}.

\bibitem[Borwein and Dutta(2016)]{BorDut16}
Jonathan~M Borwein and Joydeep Dutta.
\newblock Maximal monotone inclusions and fitzpatrick functions.
\newblock \emph{Journal of Optimization Theory and Applications}, 171\penalty0
  (3):\penalty0 757--784, 2016.

\bibitem[Bot and Csetnek(2016)]{BotCse16}
Radu~Ioan Bot and Erno~Robert Csetnek.
\newblock An inertial forward-backward-forward primal-dual splitting algorithm
  for solving monotone inclusion problems.
\newblock \emph{Numerical Algorithms}, 71\penalty0 (3):\penalty0 519--540,
  2016.
\newblock \doi{10.1007/s11075-015-0007-5}.
\newblock URL \url{https://doi.org/10.1007/s11075-015-0007-5}.

\bibitem[Bo{\c t} and Csetnek(2016)]{BotCse16b}
Radu~Ioan Bo{\c t} and Ern{\"o}Robert Csetnek.
\newblock Second order forward-backward dynamical systems for monotone
  inclusion problems.
\newblock \emph{SIAM Journal on Control and Optimization}, 54\penalty0
  (3):\penalty0 1423--1443, 2021/06/24 2016.
\newblock \doi{10.1137/15M1012657}.
\newblock URL \url{https://doi.org/10.1137/15M1012657}.

\bibitem[Bot et~al.(2020)Bot, Sedlmayer, and Vuong]{BotSedVuo00}
Radu~Ioan Bot, Michael Sedlmayer, and Phan~Tu Vuong.
\newblock A relaxed inertial forward-backward-forward algorithm for solving
  monotone inclusions with application to gans.
\newblock \emph{arXiv preprint arXiv:2003.07886}, 2020.

\bibitem[Bo{\c t} et~al.(2021)Bo{\c t}, Mertikopoulos, Staudigl, and
  Vuong]{BotMerStaVu21}
Radu~Ioan Bo{\c t}, Panayotis Mertikopoulos, Mathias Staudigl, and Phan~Tu
  Vuong.
\newblock Minibatch forward-backward-forward methods for solving stochastic
  variational inequalities.
\newblock \emph{Stochastic Systems}, 2021.
\newblock \doi{10.1287/stsy.2019.0064}.
\newblock URL \url{https://doi.org/10.1287/stsy.2019.0064}.

\bibitem[Brice{\~n}o-Arias and Combettes(2011)]{BricCom11}
Luis~M. Brice{\~n}o-Arias and Patrick~L. Combettes.
\newblock A monotone+skew splitting model for composite monotone inclusions in
  duality.
\newblock \emph{SIAM Journal on Optimization}, 21\penalty0 (4):\penalty0
  1230--1250, 2011.
\newblock \doi{10.1137/10081602X}.
\newblock URL \url{https://doi.org/10.1137/10081602X}.

\bibitem[Briceno-Arias and Combettes(2013)]{BricCom13}
Luis~M Briceno-Arias and Patrick~L Combettes.
\newblock \emph{Monotone operator methods for Nash equilibria in non-potential
  games}, pages 143--159.
\newblock Springer, 2013.

\bibitem[Burachik and Mill{\'a}n(2020)]{BurMil20}
Regina~S. Burachik and R.~D{\'\i}az Mill{\'a}n.
\newblock A projection algorithm for non-monotone variational inequalities.
\newblock \emph{Set-Valued and Variational Analysis}, 28\penalty0 (1):\penalty0
  149--166, 2020.
\newblock \doi{10.1007/s11228-019-00517-0}.
\newblock URL \url{https://doi.org/10.1007/s11228-019-00517-0}.

\bibitem[Chen et~al.(2014)Chen, Lan, and Ouyang]{ChenLanOuy2014}
Yunmei Chen, Guanghui Lan, and Yuyuan Ouyang.
\newblock Optimal primal-dual methods for a class of saddle point problems.
\newblock \emph{SIAM Journal on Optimization}, 24\penalty0 (4):\penalty0
  1779--1814, 2021/04/27 2014.
\newblock \doi{10.1137/130919362}.
\newblock URL \url{https://doi.org/10.1137/130919362}.

\bibitem[Chen et~al.(2017)Chen, Lan, and Ouyang]{CheLanOuy17}
Yunmei Chen, Guanghui Lan, and Yuyuan Ouyang.
\newblock Accelerated schemes for a class of variational inequalities.
\newblock \emph{Mathematical Programming}, 2017.
\newblock \doi{10.1007/s10107-017-1161-4}.
\newblock URL \url{https://doi.org/10.1007/s10107-017-1161-4}.

\bibitem[Combettes and Pesquet(2012)]{ComPes12}
Patrick~L Combettes and Jean-Christophe Pesquet.
\newblock Primal-dual splitting algorithm for solving inclusions with mixtures
  of composite, lipschitzian, and parallel-sum type monotone operators.
\newblock \emph{Set-Valued and variational analysis}, 20\penalty0 (2):\penalty0
  307--330, 2012.

\bibitem[Combettes and Pesquet(2015)]{ComPes15}
Patrick~L. Combettes and Jean-Christophe Pesquet.
\newblock Stochastic quasi-fej{\'e}r block-coordinate fixed point iterations
  with random sweeping.
\newblock \emph{SIAM Journal on Optimization}, 25\penalty0 (2):\penalty0
  1221--1248, 2020/07/11 2015.
\newblock \doi{10.1137/140971233}.
\newblock URL \url{https://doi.org/10.1137/140971233}.

\bibitem[Combettes and Pesquet(2019)]{ComPes19}
Patrick~L. Combettes and Jean-Christophe Pesquet.
\newblock Stochastic quasi-fej{\'e}r block-coordinate fixed point iterations
  with random sweeping ii: mean-square and linear convergence.
\newblock \emph{Mathematical Programming}, 174\penalty0 (1):\penalty0 433--451,
  2019.
\newblock \doi{10.1007/s10107-018-1296-y}.
\newblock URL \url{https://doi.org/10.1007/s10107-018-1296-y}.

\bibitem[Cui and Shanbhag(2021{\natexlab{a}})]{CuiSha20}
Shisheng Cui and Uday~V Shanbhag.
\newblock On the computation of equilibria in monotone and potential stochastic
  hierarchical game.
\newblock \emph{arXiv preprint arXiv:2104.07860}, 2021{\natexlab{a}}.

\bibitem[Cui and Shanbhag(2021{\natexlab{b}})]{CuiSha21}
Shisheng Cui and Uday~V Shanbhag.
\newblock On the analysis of variance-reduced and randomized projection
  variants of single projection schemes for monotone stochastic variational
  inequality problems.
\newblock \emph{Set-Valued and Variational Analysis (to appear)},
  2021{\natexlab{b}}.

\bibitem[Davis and Drusvyatskiy(2019)]{DavDru19}
Damek Davis and Dmitriy Drusvyatskiy.
\newblock Stochastic model-based minimization of weakly convex functions.
\newblock \emph{SIAM Journal on Optimization}, 29\penalty0 (1):\penalty0
  207--239, 2019.

\bibitem[Diakonikolas et~al.(2021)Diakonikolas, Daskalakis, and
  Jordan]{Diakonikolas:2021vz}
Jelena Diakonikolas, Constantinos Daskalakis, and Michael Jordan.
\newblock Efficient methods for structured nonconvex-nonconcave min-max
  optimization.
\newblock \emph{International Conference on Artificial Intelligence and
  Statistics}, pages 2746--2754, 2021.

\bibitem[Duvocelle et~al.(2018)Duvocelle, Mertikopoulos, Staudigl, and
  Vermeulen]{DuvMerStaVer18}
Benoit Duvocelle, Panayotis Mertikopoulos, Mathias Staudigl, and Dries
  Vermeulen.
\newblock Learning in time-varying games.
\newblock \emph{arXiv preprint arXiv:1809.03066}, 2018.

\bibitem[Facchinei and Pang(2003)]{FacPan03}
Francisco Facchinei and Jong-shi Pang.
\newblock \emph{Finite-Dimensional Variational Inequalities and Complementarity
  Problems - Volume I and Volume II}.
\newblock Springer Series in Operations Research, 2003.

\bibitem[Fitzpatrick(1988)]{Fit88}
Simon Fitzpatrick.
\newblock Representing monotone operators by convex functions.
\newblock In \emph{Workshop/Miniconference on Functional Analysis and
  Optimization}, pages 59--65. Centre for Mathematics and its Applications,
  Mathematical Sciences Institute~{\ldots}, 1988.

\bibitem[{Franci} et~al.(2020){Franci}, {Staudigl}, and
  {Grammatico}]{FraStaGramECC20}
B.~{Franci}, M.~{Staudigl}, and S.~{Grammatico}.
\newblock Distributed forward-backward (half) forward algorithms for
  generalized nash equilibrium seeking.
\newblock In \emph{2020 European Control Conference (ECC)}, pages 1274--1279,
  2020.
\newblock \doi{10.23919/ECC51009.2020.9143676}.

\bibitem[Friedlander and Schmidt(2012)]{friedlander12hybrid}
M.~P. Friedlander and M.~Schmidt.
\newblock Hybrid deterministic-stochastic methods for data fitting.
\newblock \emph{SIAM J. Scientific Computing}, 34\penalty0 (3):\penalty0
  A1380--A1405, 2012.

\bibitem[Friedman et~al.(2001)Friedman, Hastie, and Tibshirani]{FriHasTib01}
Jerome Friedman, Trevor Hastie, and Robert Tibshirani.
\newblock \emph{The elements of statistical learning}, volume~1.
\newblock Springer series in statistics Springer, Berlin, 2001.

\bibitem[Friesz et~al.(1993)Friesz, Bernstein, Smith, Tobin, and Wie]{Frie93}
Terry~L. Friesz, David Bernstein, Tony~E. Smith, Roger~L. Tobin, and B.~W. Wie.
\newblock Variational inequality formulation of the dynamic network user
  equilibrium.
\newblock \emph{Operations Research}, 41\penalty0 (1):\penalty0 179--191, 1993.

\bibitem[Fukushima(1996)]{Fuk96}
Masao Fukushima.
\newblock The primal douglas-rachford splitting algorithm for a class of
  monotone mappings with application to the traffic equilibrium problem.
\newblock \emph{Mathematical Programming}, 72\penalty0 (1):\penalty0 1--15,
  1996.
\newblock ISSN 1436-4646.
\newblock \doi{10.1007/BF02592328}.
\newblock URL \url{https://doi.org/10.1007/BF02592328}.

\bibitem[Gadat et~al.(2018)Gadat, Panloup, and Saadane]{GadPanSaa18}
S{\'e}bastien Gadat, Fabien Panloup, and Sofiane Saadane.
\newblock {Stochastic heavy ball}.
\newblock \emph{Electronic Journal of Statistics}, 12\penalty0 (1):\penalty0
  461 -- 529, 2018.
\newblock \doi{10.1214/18-EJS1395}.
\newblock URL \url{https://doi.org/10.1214/18-EJS1395}.

\bibitem[Geiersbach and Pflug(2019)]{GeiPfl19}
Caroline Geiersbach and Georg~Ch. Pflug.
\newblock Projected stochastic gradients for convex constrained problems in
  hilbert spaces.
\newblock \emph{SIAM Journal on Optimization}, 29\penalty0 (3):\penalty0
  2079--2099, 2020/11/23 2019.
\newblock \doi{10.1137/18M1200208}.
\newblock URL \url{https://doi.org/10.1137/18M1200208}.

\bibitem[Geiersbach and Wollner(2020)]{Geiersbach:2020tw}
Caroline Geiersbach and Winnifried Wollner.
\newblock A stochastic gradient method with mesh refinement for pde-constrained
  optimization under uncertainty.
\newblock \emph{SIAM Journal on Scientific Computing}, 42\penalty0
  (5):\penalty0 A2750--A2772, 2021/07/16 2020.
\newblock \doi{10.1137/19M1263297}.
\newblock URL \url{https://doi.org/10.1137/19M1263297}.

\bibitem[Ghadimi et~al.(2016)Ghadimi, Lan, and Zhang]{GhaLanHon16}
Saeed Ghadimi, Guanghui Lan, and Hongchao Zhang.
\newblock Mini-batch stochastic approximation methods for nonconvex stochastic
  composite optimization.
\newblock \emph{Mathematical Programming}, 155\penalty0 (1-2):\penalty0
  267--305, 2016.

\bibitem[Gidel et~al.(2018)Gidel, Berard, Vignoud, Vincent, and
  Lacoste-Julien]{gidel2018variational}
Gauthier Gidel, Hugo Berard, Ga{\"e}tan Vignoud, Pascal Vincent, and Simon
  Lacoste-Julien.
\newblock A variational inequality perspective on generative adversarial
  networks.
\newblock \emph{arXiv preprint arXiv:1802.10551}, 2018.

\bibitem[Gower et~al.(2020)Gower, Schmidt, Bach, and
  Richt{\'a}rik]{Gower:2020ty}
R.~M. Gower, M.~Schmidt, F.~Bach, and P.~Richt{\'a}rik.
\newblock Variance-reduced methods for machine learning.
\newblock \emph{Proceedings of the IEEE}, 108\penalty0 (11):\penalty0
  1968--1983, 2020.
\newblock \doi{10.1109/JPROC.2020.3028013}.

\bibitem[Gunzburger et~al.(2014)Gunzburger, Webster, and Zhang]{GunWebZha14}
Max~D Gunzburger, Clayton~G Webster, and Guannan Zhang.
\newblock Stochastic finite element methods for partial differential equations
  with random input data.
\newblock \emph{Acta Numer.}, 23:\penalty0 521--650, 2014.

\bibitem[Han et~al.(2019)Han, Eve, and Friesz]{Han:2019ua}
Ke~Han, Gabriel Eve, and Terry~L. Friesz.
\newblock Computing dynamic user equilibria on large-scale networks with
  software implementation.
\newblock \emph{Networks and Spatial Economics}, 19\penalty0 (3):\penalty0
  869--902, 2019.
\newblock \doi{10.1007/s11067-018-9433-y}.
\newblock URL \url{https://doi.org/10.1007/s11067-018-9433-y}.

\bibitem[Iusem et~al.(2019)Iusem, Jofr{\'e}, Oliveira, and
  Thompson]{Iusem:2019ws}
Alfredo~N. Iusem, Alejandro Jofr{\'e}, Roberto~I. Oliveira, and Philip
  Thompson.
\newblock Variance-based extragradient methods with line search for stochastic
  variational inequalities.
\newblock \emph{SIAM Journal on Optimization}, 29\penalty0 (1):\penalty0
  175--206, 2021/07/16 2019.
\newblock \doi{10.1137/17M1144799}.
\newblock URL \url{https://doi.org/10.1137/17M1144799}.

\bibitem[Iusem et~al.(2017)Iusem, Jofr{\'e}, Oliveira, and
  Thompson]{IusJofOliTho17}
AN~Iusem, Alejandro Jofr{\'e}, Roberto~I Oliveira, and Philip Thompson.
\newblock Extragradient method with variance reduction for stochastic
  variational inequalities.
\newblock \emph{SIAM Journal on Optimization}, 27\penalty0 (2):\penalty0
  686--724 

\bibitem[Jacob et~al.(2009)Jacob, Obozinski, and Vert]{JacOboVer09}
Laurent Jacob, Guillaume Obozinski, and Jean-Philippe Vert.
\newblock Group lasso with overlap and graph lasso.
\newblock \emph{Proceedings of the 26th annual international conference on
  machine learning}, pages 433--440, 2009.

\bibitem[Jalilzadeh et~al.(2018)Jalilzadeh, Shanbhag, Blanchet, and
  Glynn]{jalilzadeh2018smoothed}
Afrooz Jalilzadeh, Uday~V Shanbhag, Jose~H Blanchet, and Peter~W Glynn.
\newblock Smoothed variable sample-size accelerated proximal methods for
  nonsmooth stochastic convex programs.
\newblock \emph{arXiv preprint arXiv:1803.00718}, 2018.

\bibitem[Jiang and Xu(2008)]{sajiang08}
H.~Jiang and H.~Xu.
\newblock Stochastic approximation approaches to the stochastic variational
  inequality problem.
\newblock \emph{IEEE Transactions on Automatic Control}, 53\penalty0
  (6):\penalty0 1462--1475, 2008.
\newblock ISSN 0018-9286.
\newblock \doi{10.1109/TAC.2008.925853}.

\bibitem[Jofr{\'e} and Thompson(2019)]{jofre2019variance}
Alejandro Jofr{\'e} and Philip Thompson.
\newblock On variance reduction for stochastic smooth convex optimization with
  multiplicative noise.
\newblock \emph{Mathematical Programming}, 174\penalty0 (1-2):\penalty0
  253--292, 2019.

\bibitem[Juditsky et~al.(2011)Juditsky, Nemirovski, and Tauvel]{JudNemTau11}
Anatoli Juditsky, Arkadi Nemirovski, and Claire Tauvel.
\newblock Solving variational inequalities with stochastic mirror-prox
  algorithm.
\newblock pages 17--58, 2011.
\newblock \doi{10.1214/10-SSY011}.
\newblock URL \url{http://projecteuclid.org/euclid.ssy/1393252123}.

\bibitem[Kannan and Shanbhag(2019)]{KanSha19}
Aswin Kannan and Uday~V Shanbhag.
\newblock Optimal stochastic extragradient schemes for pseudomonotone
  stochastic variational inequality problems and their variants.
\newblock \emph{Computational Optimization and Applications}, 74\penalty0
  (3):\penalty0 779--820, 2019.

\bibitem[Lan(2020)]{Lan:2020tp}
Guanghui Lan.
\newblock \emph{First-order and Stochastic Optimization Methods for Machine
  Learning}.
\newblock Springer Series in the Data Sciences. Springer Nature, 2020.
\newblock ISBN 3030395685.

\bibitem[Latafat et~al.(2019)Latafat, Freris, and Patrinos]{latafat19}
Puya Latafat, Nikolaos~M. Freris, and Panagiotis Patrinos.
\newblock A new randomized block-coordinate primal-dual proximal algorithm for
  distributed optimization.
\newblock \emph{IEEE Transactions on Automatic Control}, 64\penalty0
  (10):\penalty0 4050--4065, 2019.
\newblock \doi{10.1109/TAC.2019.2906924}.

\bibitem[Lei and Shanbhag(2019)]{lei18game}
Jinlong Lei and Uday~V. Shanbhag.
\newblock Distributed variable sample-size gradient-response and best-response
  schemes for stochastic nash equilibrium problems over graphs.
\newblock \emph{https://arxiv.org/abs/1811.11246}, 2019.

\bibitem[Lei and Shanbhag(2020)]{lei2020asynchronous}
Jinlong Lei and Uday~V Shanbhag.
\newblock Asynchronous variance-reduced block schemes for composite non-convex
  stochastic optimization: block-specific steplengths and adapted batch-sizes.
\newblock \emph{Optimization Methods and Software}, pages 1--31, 2020.

\bibitem[Lorenz and Pock(2015)]{LorPoc15}
Dirk~A. Lorenz and Thomas Pock.
\newblock An inertial forward-backward algorithm for monotone inclusions.
\newblock \emph{Journal of Mathematical Imaging and Vision}, 51\penalty0
  (2):\penalty0 311--325, 2015.
\newblock \doi{10.1007/s10851-014-0523-2}.
\newblock URL \url{https://doi.org/10.1007/s10851-014-0523-2}.

\bibitem[Malitsky(2020)]{Mal20}
Yura Malitsky.
\newblock Golden ratio algorithms for variational inequalities.
\newblock \emph{Mathematical Programming}, 184\penalty0 (1):\penalty0 383--410,
  2020.
\newblock \doi{10.1007/s10107-019-01416-w}.
\newblock URL \url{https://doi.org/10.1007/s10107-019-01416-w}.

\bibitem[Mertikopoulos and Staudigl(2017)]{MerStaCDC17}
Panayotis Mertikopoulos and Mathias Staudigl.
\newblock Convergence to nash equilibrium in continuous games with noisy
  first-order feedback.
\newblock In \emph{56th IEEE Conference on Decision and Control}, 2017.

\bibitem[Mishchenko et~al.(2020)Mishchenko, Kovalev, Shulgin, Richt{\'a}rik,
  and Malitsky]{mishchenko2020revisiting}
Konstantin Mishchenko, Dmitry Kovalev, Egor Shulgin, Peter Richt{\'a}rik, and
  Yura Malitsky.
\newblock Revisiting stochastic extragradient.
\newblock In \emph{International Conference on Artificial Intelligence and
  Statistics}, pages 4573--4582. PMLR, 2020.

\bibitem[Monteiro and Svaiter(2010)]{MonSva10}
Renato D.~C. Monteiro and B.~F. Svaiter.
\newblock On the complexity of the hybrid proximal extragradient method for the
  iterates and the ergodic mean.
\newblock \emph{SIAM Journal on Optimization}, 20\penalty0 (6):\penalty0
  2755--2787, 2020/11/24 2010.
\newblock \doi{10.1137/090753127}.
\newblock URL \url{https://doi.org/10.1137/090753127}.

\bibitem[Nemirovski et~al.(2009)Nemirovski, Juditsky, Lan, and
  Shapiro]{JudLanNemSha09}
Arkadi Nemirovski, Anatoli Juditsky, Guanghui Lan, and Alexander Shapiro.
\newblock Robust stochastic approximation approach to stochastic programming.
\newblock \emph{SIAM Journal on optimization}, 19\penalty0 (4):\penalty0
  1574--1609, 2009.

\bibitem[Nesterov(2004{\natexlab{a}})]{NesConvex}
Y.~Nesterov.
\newblock \emph{Introductory Lectures on Convex Optimization: A Basic Course}.
\newblock Kluwer, Dordrecht, 2004{\natexlab{a}}.

\bibitem[Nesterov(1983)]{Nes83}
Yu. Nesterov.
\newblock A method of solving a convex programming problem with convergence
  rate $o(1/k^{2})$.
\newblock \emph{Soviet Mathematics Doklady}, 27\penalty0 (2):\penalty0
  372--376, 1983.

\bibitem[Nesterov(2004{\natexlab{b}})]{Nes04}
Yurii Nesterov.
\newblock \emph{Introductory Lectures on Convex Optimization: A Basic Course}.
\newblock Number~87 in Applied Optimization. Kluwer Academic Publishers,
  2004{\natexlab{b}}.

\bibitem[Nesterov(2007)]{Nes07}
Yurii Nesterov.
\newblock Dual extrapolation and its applications to solving variational
  inequalities and related problems.
\newblock \emph{Mathematical Programming}, 109\penalty0 (2):\penalty0 319--344,
  2007.
\newblock \doi{10.1007/s10107-006-0034-z}.
\newblock URL \url{https://doi.org/10.1007/s10107-006-0034-z}.

\bibitem[Palaniappan and Bach(2016)]{PalBac16}
Balamurugan Palaniappan and Francis Bach.
\newblock Stochastic variance reduction methods for saddle-point problems.
\newblock In \emph{Advances in Neural Information Processing Systems}, pages
  1416--1424, 2016.

\bibitem[Polyak(1964)]{Pol64}
B.~T. Polyak.
\newblock Some methods of speeding up the convergence of iteration methods.
\newblock \emph{USSR Computational Mathematics and Mathematical Physics},
  4\penalty0 (5):\penalty0 1--17, 1964.
\newblock \doi{https://doi.org/10.1016/0041-5553(64)90137-5}.
\newblock URL
  \url{http://www.sciencedirect.com/science/article/pii/0041555364901375}.

\bibitem[Polyak(1987)]{Pol87}
Boris~T. Polyak.
\newblock \emph{Introduction to Optimization}.
\newblock Optimization Software, 1987.

\bibitem[Rockafellar(1974)]{Roc74}
R.~Tyrrell Rockafellar.
\newblock \emph{Conjugate Duality and Optimization}.
\newblock Society for Industrial and Applied Mathematics, 1974.
\newblock \doi{10.1137/1.9781611970524}.
\newblock URL \url{https://epubs.siam.org/doi/abs/10.1137/1.9781611970524}.

\bibitem[Rockafellar and Wets(2017)]{rockafellar17stochastic}
R.~Tyrrell Rockafellar and Roger~J.{-}B. Wets.
\newblock Stochastic variational inequalities: single-stage to multistage.
\newblock \emph{Math. Program.}, 165\penalty0 (1):\penalty0 331--360, 2017.
\newblock \doi{10.1007/s10107-016-0995-5}.
\newblock URL \url{https://doi.org/10.1007/s10107-016-0995-5}.

\bibitem[Rockafellar and Wets(1998)]{RocWet98}
Tyrrel~R. Rockafellar and Roger J-B. Wets.
\newblock \emph{Variational Analysis}.
\newblock Springer Verlarg - Grundlehren der Mathematischen Wissenschaften,
  Volume 317, 1998.

\bibitem[Rosasco et~al.(2016{\natexlab{a}})Rosasco, Villa, and
  V{\~u}]{RosVilVu16}
Lorenzo Rosasco, Silvia Villa, and B{\u a}ng~C{\^o}ng V{\~u}.
\newblock A stochastic inertial forward--backward splitting algorithm for
  multivariate monotone inclusions.
\newblock \emph{Optimization}, 65\penalty0 (6):\penalty0 1293--1314,
  2016{\natexlab{a}}.
\newblock \doi{10.1080/02331934.2015.1127371}.
\newblock URL \url{https://doi.org/10.1080/02331934.2015.1127371}.

\bibitem[Rosasco et~al.(2016{\natexlab{b}})Rosasco, Villa, and
  V{\~u}]{RosVilVuSBF16}
Lorenzo Rosasco, Silvia Villa, and Bang~C{\^o}ng V{\~u}.
\newblock Stochastic forward--backward splitting for monotone inclusions.
\newblock \emph{Journal of Optimization Theory and Applications}, 169\penalty0
  (2):\penalty0 388--406, 2016{\natexlab{b}}.
\newblock \doi{10.1007/s10957-016-0893-2}.
\newblock URL \url{https://doi.org/10.1007/s10957-016-0893-2}.

\bibitem[Rudin et~al.(1992)Rudin, Osher, and Fatemi]{RudOshFat92}
Leonid~I. Rudin, Stanley Osher, and Emad Fatemi.
\newblock Nonlinear total variation based noise removal algorithms.
\newblock \emph{Physica D: Nonlinear Phenomena}, 60\penalty0 (1):\penalty0
  259--268, 1992.
\newblock ISSN 0167-2789.
\newblock \doi{https://doi.org/10.1016/0167-2789(92)90242-F}.
\newblock URL
  \url{https://www.sciencedirect.com/science/article/pii/016727899290242F}.

\bibitem[Shanbhag(2013)]{shanbhag13stochastic}
Uday~V. Shanbhag.
\newblock \emph{Stochastic Variational Inequality Problems: Applications,
  Analysis, and Algorithms}, chapter Chapter 5, pages 71--107.
\newblock 2013.
\newblock \doi{10.1287/educ.2013.0120}.
\newblock URL
  \url{https://pubsonline.informs.org/doi/abs/10.1287/educ.2013.0120}.

\bibitem[Shapiro et~al.(2009)Shapiro, Dentcheva, and
  Ruszczy{\'n}ski]{DenRusSha09}
Alexander Shapiro, Darinka Dentcheva, and Andrzej 
  Ruszczy{\'n}ski.
\newblock \emph{Lectures on stochastic programming: modeling and theory}.
\newblock SIAM, 2009.

\bibitem[Simons and Zalinescu(2004)]{SimZal04}
S.~Simons and C.~Zalinescu.
\newblock A new proof for rockafellar's characterization of maximal monotone
  operators.
\newblock 132\penalty0 (10):\penalty0 2969--2972, 2004.
\newblock URL \url{http://www.jstor.org/stable/4097259}.

\bibitem[Spall(1992)]{Spall92}
James~C Spall.
\newblock Multivariate stochastic approximation using a simultaneous
  perturbation gradient approximation.
\newblock \emph{IEEE transactions on automatic control}, 37\penalty0
  (3):\penalty0 332--341, 1992.

\bibitem[Spall(1997)]{Spa97}
James~C. Spall.
\newblock A one-measurement form of simultaneous perturbation stochastic
  approximation.
\newblock \emph{Automatica}, 33\penalty0 (1):\penalty0 109--112, 1997.

\bibitem[Staudigl and Mertikopoulos(2019)]{MerStaIFAC19}
Mathias Staudigl and Panayotis Mertikopoulos.
\newblock Convergent noisy forward-backward-forward algorithms in non-monotone
  variational inequalities.
\newblock \emph{IFAC-PapersOnLine}, 52\penalty0 (3):\penalty0 120--125, 2019.

\bibitem[Su et~al.(2016)Su, Boyd, and Candes]{SuBoyCan16}
Weijie Su, Stephen Boyd, and Emmanuel~J. Candes.
\newblock A differential equation for modeling nesterov's accelerated gradient
  method: Theory and insights.
\newblock \emph{Journal of Machine Learning Research}, 2016.

\bibitem[Thong et~al.(2021)Thong, Gibali, Staudigl, and Vuong]{GibStaThoVuo20}
Duong~Viet Thong, Aviv Gibali, Mathias Staudigl, and Phan~Tu Vuong.
\newblock Computing dynamic user equilibrium on large-scale networks without
  knowing global parameters.
\newblock \emph{Networks and Spatial Economics (forthcoming) arXiv preprint
  arXiv:2010.04597}, 2021.

\bibitem[Tibshirani et~al.(2005)Tibshirani, Saunders, Rosset, Zhu, and
  Knight]{TibRosSauZhu05}
Robert Tibshirani, Michael Saunders, Saharon Rosset, Ji~Zhu, and Keith Knight.
\newblock Sparsity and smoothness via the fused lasso.
\newblock \emph{Journal of the Royal Statistical Society: Series B (Statistical
  Methodology)}, 67\penalty0 (1):\penalty0 91--108, 2019/04/05 2005.
\newblock \doi{10.1111/j.1467-9868.2005.00490.x}.
\newblock URL \url{https://doi.org/10.1111/j.1467-9868.2005.00490.x}.

\bibitem[Tibshirani and Taylor(2011)]{TibTay11}
Ryan~J. Tibshirani and Jonathan Taylor.
\newblock The solution path of the generalized lasso.
\newblock \emph{Ann. Statist.}, 39\penalty0 (3):\penalty0 1335--1371, 2011.
\newblock \doi{10.1214/11-AOS878}.
\newblock URL \url{https://projecteuclid.org:443/euclid.aos/1304514656}.

\bibitem[Tseng(2000)]{Tse00}
P.~Tseng.
\newblock A modified forward-backward splitting method for maximal monotone
  mappings.
\newblock \emph{SIAM Journal on Control and Optimization}, 38\penalty0
  (2):\penalty0 431--446, 2018/09/13 2000.
\newblock \doi{10.1137/S0363012998338806}.
\newblock URL \url{https://doi.org/10.1137/S0363012998338806}.

\bibitem[Yi and Pavel(2019)]{YiPav19}
Peng Yi and Lacra Pavel.
\newblock An operator splitting approach for distributed generalized nash
  equilibria computation.
\newblock \emph{Automatica}, 102:\penalty0 111--121, 2019.
\newblock \doi{https://doi.org/10.1016/j.automatica.2019.01.008}.
\newblock URL
  \url{http://www.sciencedirect.com/science/article/pii/S0005109819300081}.

\bibitem[Yousefian et~al.(2017)Yousefian, Nedi{\'c}, and Shanbhag]{YouNedSha17}
Farzad Yousefian, Angelia Nedi{\'c}, and Uday~V. Shanbhag.
\newblock On smoothing, regularization, and averaging in stochastic
  approximation methods for stochastic variational inequality problems.
\newblock \emph{Mathematical Programming}, 165\penalty0 (1):\penalty0 391--431,
  2017.
\newblock \doi{10.1007/s10107-017-1175-y}.
\newblock URL \url{https://doi.org/10.1007/s10107-017-1175-y}.

\bibitem[Zhao et~al.(2009)Zhao, Rocha, and Yu]{RocYuZha09}
Peng Zhao, Guilherme Rocha, and Bin Yu.
\newblock The composite absolute penalties family for grouped and hierarchical
  variable selection.
\newblock \emph{The Annals of Statistics}, 37\penalty0 (6A):\penalty0
  3468--3497, 12 2009.
\newblock \doi{10.1214/07-AOS584}.
\newblock URL \url{https://doi.org/10.1214/07-AOS584}.

\end{thebibliography}
\end{document}